\documentclass[a4paper,12pt]{amsart}
\usepackage{yhmath}
\usepackage{mathtools} 
\usepackage{graphicx}
\usepackage{mathtools}
\usepackage{mathrsfs}
\usepackage{amsmath,amssymb}
\usepackage{amsthm}
\usepackage{amsfonts}
\usepackage{dsfont}
\usepackage{fancyhdr}
\usepackage{harpoon}
\usepackage{amssymb}
\usepackage{enumerate}
\usepackage[bookmarks=true,colorlinks,linkcolor=black]{hyperref}
\usepackage{geometry}
\usepackage{extarrows}
\usepackage{tikz}
\usetikzlibrary{decorations.markings}
\usetikzlibrary{arrows}
\usepackage{blkarray}
\usepackage{multirow}
\usepackage{color}
\usepackage{float}
\usepackage[utf8]{inputenc}
\usepackage{chngcntr}
\usepackage{etoolbox}
\usepackage{tikz-cd}
\usepackage[xindy]{imakeidx}
\usepackage{blkarray}
\usepackage{multirow}
\usepackage{cancel}
\usepackage{array}
\usepackage{longtable}
\usepackage{float}
\usepackage[utf8]{inputenc}
\usepackage{chngcntr}
\usepackage{etoolbox}
\usepackage{caption}
\usepackage{stackrel}

\usepackage{diagbox} 
\usepackage{hyperref}
\hypersetup{  
    colorlinks=true,
    linkcolor=red,
    filecolor=magenta,      
    urlcolor=cyan,
    pdftitle={Overleaf Example},
    pdfpagemode=FullScreen,
    }

\newtheorem{thm}{Theorem}[section]
\newtheorem{cor}[thm]{Corollary}

\newtheorem{cla}[thm]{Claim}
\newtheorem{lem}[thm]{Lemma}
\newtheorem{prop}[thm]{Proposition}
\theoremstyle{definition}
\newtheorem{defn}[thm]{Definition}
\newtheorem{ex}[thm]{Example}
\newtheorem{rem}[thm]{Remark}
\theoremstyle{definition}

\numberwithin{equation}{section}

\DeclareMathOperator{\Imm}{Im}

\DeclareMathOperator{\Ree}{Re}

\DeclareMathOperator{\Hom}{Hom}

\DeclareMathOperator{\Area}{Area}

\DeclareMathOperator{\Span}{Span}

\DeclareMathOperator{\vol}{vol}
\DeclareMathOperator{\Vol}{Vol}

\DeclareMathOperator{\stab}{stab}

\DeclareMathOperator{\Ad}{Ad}

\DeclareMathOperator{\Aut}{Aut}
\DeclareMathOperator{\Lie}{Lie}

\DeclareMathOperator{\Ind}{Ind}
\DeclareMathOperator{\Mod}{Mod}

\DeclareMathOperator{\Gr}{Gr}

\DeclareMathOperator{\Leb}{Leb}

\DeclareMathOperator{\htt}{ht}

\DeclareMathOperator{\supp}{supp}

\DeclareMathOperator{\op}{op}

\DeclareMathOperator{\ord}{ord}

\DeclareMathOperator{\Aff}{Aff}

\DeclareMathOperator{\SL}{SL}
\DeclareMathOperator{\Sp}{Sp}

\DeclareMathOperator{\GL}{GL}

\DeclareMathOperator{\SO}{SO}

\DeclareMathOperator{\disc}{disc}

\DeclareMathOperator{\inj}{inj}

\DeclareMathOperator{\rel}{rel}
\DeclareMathOperator{\hol}{hol}

\DeclareMathOperator{\Id}{Id}

\DeclareMathOperator{\Jac}{Jac}

    {\begin{center}%
     \textsc{Acknowledgements}
\end{center}}%
    {\null}
\begin{document}


\title{Effective   Exponential Drifts on Strata of Abelian Differentials}
\author{Siyuan Tang}%
\address{B\MakeLowercase{eijing} I\MakeLowercase{nternational} C\MakeLowercase{enter} \MakeLowercase{for} M\MakeLowercase{athematical} R\MakeLowercase{esearch}, P\MakeLowercase{eking} U\MakeLowercase{niversity}, B\MakeLowercase{eijing}, 100871}
\email{1992.siyuan.tang@gmail.com,  siyuantang@pku.edu.cn}
\maketitle

\begin{abstract}
We study the dynamics of $\SL_{2}(\mathbb{R})$ on the stratum of translation surfaces $\mathcal{H}(2)$.  
In particular, we prove that 
  an orbit of the upper triangular subgroup of $\SL_{2}(\mathbb{R})$  has a discretized dimension of almost $1$ in a direction transverse to the $\SL_{2}(\mathbb{R})$-orbit.  
  
  The proof proceeds via an effective closing lemma, and the Margulis function technique, which  serves as an effective version of the exponential drift on     $\mathcal{H}(2)$. The idea is based on the use of     McMullen's  classification theorem, together with  Lindenstrauss-Mohammadi-Wang's effective equidistribution theorems in homogeneous dynamics.    
\end{abstract}

\tableofcontents

\section{Introduction}
\subsection{Main results}
Suppose $g\geq 1$, and let $\alpha=(\alpha_{1},\ldots,\alpha_{n})$ be a partition of $2g-2$. Let $\mathcal{H}(\alpha)$ be a stratum of Abelian differentials, i.e. the space of pairs $(M,\omega)$ where $M$ is a Riemann surface and $\omega$ is a holomorphic $1$-form on $M$ whose zeroes have multiplicities $\alpha_{1},\ldots,\alpha_{n}$. Let $\Sigma\subset M$ be the set of zeroes of $\omega$. Hence $|\Sigma|=n$.

Let  $\{\gamma_{1},\ldots,\gamma_{2g+|\Sigma|-1}\}$ be a symplectic $\mathbb{Z}$-basis for the relative homology group $H_{1}(M,\Sigma;\mathbb{Z})$.
Define the \textit{period coordinates}\index{period coordinates} $\Phi:\mathcal{H}(\alpha)\rightarrow H^{1}(M,\Sigma;\mathbb{C})\cong\mathbb{C}^{2g+|\Sigma|-1}$ by
  \[\Phi: x= (M,\omega)\mapsto\left(\int_{\gamma_{1}}\omega,\cdots,\int_{\gamma_{2g+|\Sigma|-1}}\omega\right).\]
The period coordinates imply a local coordinate on $\mathcal{H}(\alpha)$.

 The space $\mathcal{H}(\alpha)$  can be identified with the quotient of the Teichm\"{u}ller space $\mathcal{TH}(\alpha)$ and the mapping class group $\Mod$:
\[\mathcal{H}(\alpha)=\mathcal{TH}(\alpha)/\Mod.\]
 Let $\mathcal{H}_{A}(\alpha)$  be the subset of  translation surfaces of area $A$. 
 
 The space $\mathcal{H}_{1}(\alpha)$ admits an action of the group $\SL_{2}(\mathbb{R})$.  
Let $G=\SL_{2}(\mathbb{R})$, and
\[a_{t}=\left[
            \begin{array}{cccc}
   e^{\frac{1}{2}t} & 0  \\
   0 & e^{-\frac{1}{2}t}   \\
            \end{array}
          \right],  \ \ \ u_{r}=\left[
            \begin{array}{cccc}
   1 & r  \\
   0 & 1   \\
            \end{array}
          \right].\]
For any subset $I\subset\mathbb{R}$, denote $a_{I}=\{a_{t}:t\in I\}$, and  $u_{I}=\{u_{r}:r\in I\}$. Let $P=a_{\mathbb{R}}u_{\mathbb{R}}$. The seminal work of Eskin and Mirzakhani in \cite{eskin2018invariant} shows that every $P$-orbit closure in $\mathcal{H}_{1}(\alpha)$ is in fact $G$-invariant, which is analogous to Ratner's theorem in homogeneous dynamics. The main strategy is the \textit{exponential drift}\index{exponential drift} idea originated from Benoist and Quint \cite{benoist2011mesures}. In this paper, we provide an effective version of the exponential drift, at least on $\mathcal{H}_{1}(\alpha)$ for certain $\alpha$.

Also, by work of McMullen \cite{mcmullen2007dynamics}, the $G$-orbit closure in $\mathcal{H}_{1}(2)$ are either Teichm\"{u}ller curves or $\mathcal{H}_{1}(2)$ itself  (see Theorem \ref{closing2024.12.7}). In \cite{mcmullen2005teichmullerDiscriminant}, a much more detailed description for Teichm\"{u}ller curves in $\mathcal{H}_{1}(2)$ is available. In this paper, we are   interested in the Teichm\"{u}ller curve $\Omega W_{D}$ with   discriminant $D$ (see Section \ref{closing2024.12.8}).

 For $\tilde{x}\in\mathcal{TH}(\alpha)$, Let $\|\cdot\|_{\tilde{x}}$ be the \textit{Avila-Gou\"{e}zel-Yoccoz norm}\index{Avila-Gou\"{e}zel-Yoccoz norm} (or \textit{AGY norm}\index{AGY norm} for short) on $H^{1}(M,\Sigma;\mathbb{C})$,  (see Definition \ref{closing2024.12.9}). It induces a metric $d$ on $\mathcal{H}_{1}(\alpha)$. For $x\in\mathcal{H}_{1}(2)$, let $\ell(x)$ be the shortest length of a saddle connection (see (\ref{effective2024.08.12})). For $\eta>0$, define
 \[      \mathcal{H}^{(\eta)}_{1}(2)\coloneqq\{x\in\mathcal{H}_{1}(2):\ell(x)\geq\eta\}.\]

   Let $p:H^{1}(M,\Sigma;\mathbb{R})\rightarrow H^{1}(M;\mathbb{R})$ be the forgetful map. In what follows, for $x\in \mathcal{H}(\alpha)$, we write $H^{1}_{F}(x)=H^{1}(M_{x},\Sigma_{x};F)$ for $F=\mathbb{R},\mathbb{C}$.
       For $x\in\mathcal{H}(\alpha)$, let
  \[      H^{\perp}(x)\coloneqq\{v\in H^{1}_{\mathbb{R}}(x):p(\Ree x)\wedge p(v)=p(\Imm x)\wedge p(v)=0\}.\]
      Let  $W^{\pm}(x)\subset H^{1}(M,\Sigma;\mathbb{C})\cong H^{1}(M,\Sigma;\mathbb{R})\oplus iH^{1}(M,\Sigma;\mathbb{R})$ be defined by
    \begin{align}
W^{+}(x)&\coloneqq   \mathbb{R}(\Imm x)\oplus  H^{\perp}(x), \;\nonumber\\ 
W^{-}(x)&\coloneqq    i(\mathbb{R}(\Ree x)\oplus  H^{\perp}(x)).\; \nonumber
\end{align}  
  Then $W^{+}[x]$ and $W^{-}[x]$ play the role of the unstable and stable foliations for the action of $a_{t}$ on $\mathcal{H}(\alpha)$ for $t>0$, see e.g. \cite[Lemma 3.5]{eskin2018invariant}. Then we define 
  \[H^{\perp}_{\mathbb{C}}(x)\coloneqq   H^{\perp}(x)\oplus i H^{\perp}(x) \]
  and call it  the \textit{balance space}\index{balance space} at $x$. (See Section \ref{closing2024.12.12} for more details.)

 In this paper, we will show: 
 
\begin{thm}\label{margulis2024.12.1} 
    For $\alpha=(2g-2)$, there exist  $N_{0}=N_{0}(\alpha)$   that satisfy the following. For any
    $\epsilon\in(0,\frac{1}{10})$, \hypertarget{2024.12.k1} $\gamma\in(0,1)$, \hypertarget{2024.12.t1} $\eta\in(0, N_{0}^{-1})$, \hypertarget{2024.04.NN789} $N\geq  N_{0}$, $x_{0}\in \mathcal{H}_{1}(\alpha)$, there exist     
  $ \kappa_{1}=\kappa_{1}(N,\gamma,\alpha,\epsilon)>0$,   $\varkappa =\varkappa (N,\gamma,\alpha)>0$,
 $t_{1}=t_{1}(\gamma,\epsilon,\eta,\alpha,\ell(x_{0}))>0$,
  such that for $\kappa\in(0,\hyperlink{2024.12.k1}{\kappa_{1}})$ and $t\geq \hyperlink{2024.12.t1}{t_{1}}$, at least one of the following holds:  
\begin{enumerate}[\ \ \ (1)]
  \item  There exists $x_{1}\in \mathcal{H}_{1}^{(\eta)}(\alpha)$,   and a finite subset $F\subset B_{H^{\perp}_{\mathbb{C}}(x_{1})}(e^{-\kappa t})$ with 
      \[0\in F,\ \ \ |F|\geq  e^{ \frac{1}{2}t}  \]
   such that
   \begin{equation}\label{margulis2024.12.4}
     d(x_{1}+w,a_{\hyperlink{2024.12.NN}{\varkappa}t}u_{[0,2]}.x_{0})< e^{-\kappa t},\ \ \  
             \sum_{\substack{w^{\prime}\neq w\\ w^{\prime}\in F}} \frac{1}{\|w-w^{\prime}\|^{\gamma}_{x_{1}}}\leq |F|^{1+\epsilon}
   \end{equation} 
    for any $w\in F$.
  \item  There is $y\in \mathcal{H}_{1}(\alpha)$ such that   
          \begin{itemize}
            \item  $d(y,x)\leq e^{-Nt}$,
            \item The Veech group $\SL(y)\subset G$ is a non-elementary Fuchsian group.
          \end{itemize}  
          If we restrict our attention to $\alpha=(2)$, then we further  have
          \begin{itemize}
            \item   $y$ generates a Teichm\"{u}ller curve of discriminant $\leq  e^{N_{0}t}$.
          \end{itemize}
\end{enumerate}
\end{thm} 
Roughly speaking, the estimation of the sum in (\ref{margulis2024.12.4}) implies that the density of $F$ is even. 
For $\alpha=(2)$, together with the projection theorem, Theorem \ref{margulis2024.12.1}(1)  indicates that if a surface $x$   is not too close to a  Teichm\"{u}ller curve with small discriminant, then we can (effectively) find a subset of $W^{+}(x)\cap H^{\perp}_{\mathbb{C}}(x)$ with dimension almost $1$ near the  $P$-orbit $P.x$.

 The main strategy of Theorem \ref{margulis2024.12.1}   is the \textit{Margulis function}\index{Margulis function} idea in the  significant papers by Lindenstrauss, Mohammadi   and Wang  \cite{lindenstrauss2023polynomial,lindenstrauss2022effective}.

 More precisely, in the first step, we deduce an effective closing lemma that has a similar spirit to Einsiedler,   Margulis, and   Venkatesh \cite{einsiedler2009effective}, as well as to the work of Lindenstrauss and Margulis \cite{lindenstrauss2014effective}.
\begin{prop}\label{margulis2024.12.2}
   Let  the notation be as above. 
     Suppose that Theorem   \ref{effective2024.6.01}(2) does not occur.  
  Then  there exists $t_{2}=t_{2}(\ell(x_{0}),\kappa,\eta,\alpha)>0$ such that for any   $t\geq  \hyperlink{2024.12.t2}{t_{2}}$,     there exists a point $y\in \mathcal{H}^{(2\eta)}_{1}(\alpha)$, and  a finite subset $F_{1}\subset B_{H^{\perp}_{\mathbb{C}}(y)}(e^{-\kappa t})$  with
 \[         0\in F_{1},\ \ \   e^{\frac{3}{4}t}\leq |F_{1}|,\] 
    such that   
\[  d(y+w,a_{\hyperlink{2024.12.NN}{\varkappa}t}u_{[0,1]}.x_{0})< e^{-\kappa t},\ \ \    \sum_{\substack{w^{\prime}\neq w\\ w^{\prime}\in F_{1}}} \frac{1}{\|w-w^{\prime}\|^{\gamma}_{y}}\ \ \leq  |F_{1}|^{N}.\]
\end{prop}  
 Proposition \ref{margulis2024.12.2}   gives us some dimension, say $\delta_{1}>0$, in the direction  $W^{+}(x)\cap H^{\perp}_{\mathbb{C}}(x)$. See 
   Proposition \ref{margulis2024.3.07} for more details.
   
   As a side effect, we deduce the following theorem that is occurred in  \cite[Proposition 13.1]{einsiedler2009effective} in the homogeneous dynamics.
 
      \begin{thm}\label{closing2023.12.1}   
 There \hypertarget{2024.12.C311} exist $\varepsilon_{0}>0$ and $N_{1}\geq 1$ such that for   $\varepsilon\in(0,\varepsilon_{0})$, $N\geq \hyperlink{2024.04.N6}{N_{1}}$, $x\in\mathcal{H}_{1}(2)$, there   \hypertarget{2024.04.N6} exist  $T_{0}=T_{0}(\ell(x))>0$, $\kappa>0$ with the following property.
         Let $T\geq T_{0}$. Suppose that $\{g_{1},\ldots,g_{l}\}\subset B_{G}(T)$ is $1$-separated, that $l\geq (\Vol B_{G}(T))^{1-\varepsilon}$, and that for any $1\leq i,j\leq l$,
          \[d(g_{i}x,g_{j}x)< T^{-N}.\]  
          Then there is $y\in \mathcal{H}_{1}(2)$, and $N_{1}>0$ such  that
          \begin{itemize}
            \item  $d(y,x)\leq  T^{\hyperlink{2024.04.N6}{N_{1}}-N}$,
            \item $y$ generates a Teichm\"{u}ller curve of discriminant $\leq  T^{\hyperlink{2024.04.N6}{N_{1}}}$.
          \end{itemize}  
\end{thm}

Lately, Rached \cite{rached2024separation} obtained a similar closing lemma on $\mathcal{H}(2)$ using the effective closing lemma of Teichm\"{u}ller geodesics due to Eskin, Mirzakhani, Rafi \cite[Lemma 8.1]{eskin2019counting} and \cite{hamenstadt2010Bowen}.

 In the second step, we use a Margulis function technique to improve the dimension in the direction  $W^{+}(x)\cap H^{\perp}_{\mathbb{C}}(x)$.  Suppose that there are two nearby points $x_{1},x_{2}\in \mathcal{H}_{1}(\alpha)$ so that $x_{1}-x_{2}$ has nonzero projection to $H^{\perp}_{\mathbb{C}}(x_{2})$. Then under an appropriate unipotent action $u_{r}$, $u_{r}x_{1}-u_{r}x_{2}$ will have nonzero projection to $W^{+}(u_{r}x_{2})\cap H^{\perp}_{\mathbb{C}}(u_{r}x_{2})$. Moreover, by an exponentially small  adjustment of $ u_{r^{\prime}}$,    $a_{t}u_{r^{\prime}}x_{1}-a_{t}u_{r}x_{2}$ tends to the top Lyapunov in the space $W^{+}(a_{t}u_{r}x_{2})\cap H^{\perp}_{\mathbb{C}}(a_{t}u_{r}x_{2})$. This observation is the \textit{exponential drift}\index{exponential drift} technique.
 
 Now let   $w\in F_{1}\subset B_{H^{\perp}_{\mathbb{C}}(y)}(e^{-\kappa t})$ be as in Proposition \ref{margulis2024.12.2}.  Enlighten by the exponential drift idea,   one may choose an appropriate $r\in[0,1]$ so that 
 \[u_{r}(x+w)-u_{r}x=u_{r}w\]
  projects to a large scale in $W^{+}(x)\cap H^{\perp}_{\mathbb{C}}(x)$ for most of $w$. Then for these $w$,
  \[\|a_{t}u_{r}(x+w)-a_{t}u_{r}x\|_{a_{t}u_{r}x}=\|a_{t}u_{r}w\|_{a_{t}u_{r}x}\geq e\|u_{r}w\|_{u_{r}x}\]
  for a bounded time $t$. A detailed analysis then shows that 
  \[\sum_{\substack{w^{\prime}\neq w\\ w^{\prime}\in F_{1}}} \frac{1}{\|a_{t}u_{r}(w-w^{\prime})\|^{\gamma}_{y}}  \leq  \frac{1}{e} \sum_{\substack{w^{\prime}\neq w\\ w^{\prime}\in F_{1}}} \frac{1}{\|w-w^{\prime}\|^{\gamma}_{y}}.\]
 This leads to the relation of a Margulis function (Proposition \ref{margulis2024.4.16}). Also, we gain a finite set  $F_{2}\in B_{H^{\perp}_{\mathbb{C}}(a_{t}u_{r}y)}(e^{-\kappa t})$ near  $a_{t}u_{r}(F_{1})\subset a_{t}u_{r} a_{\hyperlink{2024.12.NN}{\varkappa}t}u_{[0,1]}.x_{0}$ with higher dimension $\delta_{2}>\delta_{1}$. One may consider this is an effective version of the exponential drift technique. See Section \ref{margulis2024.3.04} for more details.
  
  On the other hand, recall from \cite{eskin2018invariant} that after using the exponential drift technique, one may apply the symplectic nature of the Lyapunov exponents of the \textit{Kontsevich-Zorich cocycle}\index{Kontsevich-Zorich cocycle}. Then one may apply the exponential drift again, or conclude that the additional invariance  respects the symplectic nature of the Lyapunov exponents. It would be interesting to know if it is possible to quantify the symplectic nature of the Lyapunov exponents. Then combining with the Margulis function,   we can (effectively) find a subset of $W^{+}(x)\cap H^{\perp}_{\mathbb{C}}(x)$ with dimension almost $2$ near the  $P$-orbit $P.x$. Then apply the result of Sanchez \cite{sanchez2023effective} to get the  effective equidistribution of unstable foliation of $P$-orbits, similar to the proofs in \cite{lindenstrauss2023polynomial,lindenstrauss2022effective}.

\subsection{Outline of   Theorem \ref{margulis2024.12.1}}  
Let  
          \[H^{1}_{\mathbb{C}}(x)=H^{1}(M;\mathbb{R})\oplus i H^{1}(M;\mathbb{R}).\]
          Then for $w\in H(x)$, we write
          \[w=a+ib\]
          for $a,b\in H^{1}(M;\mathbb{R})$. Also, we consider $H^{1}(M;\mathbb{R})=W^{+}(x)\oplus\mathbb{R}(\Ree x)$ and $W^{+}(x)=\mathbb{R}(\Imm x)\oplus H^{\perp}(x)$. 
          
   In  case (1) of  Theorem \ref{margulis2024.12.1}, $\gamma$ implies the discretized fractal dimension of $F$ in $H_{\mathbb{C}}^{\perp}(x_{1})$. If we get an $F$ with dimension near $2$, then we may apply  \cite{sanchez2023effective} and obtain the effective density and equidistribution.   
    
The idea of the proof  goes as follows. We consider a long horocycle: 
\[a_{t}u_{[0,1]}x.\]
 First, let $B_{G}(\epsilon)$ be an $\epsilon$-ball of the identity in $G=\SL_{2}(\mathbb{R})$. Then we obtain a  small $G$-thickening of the long horocycle:
\[B_{G}(e^{-\kappa t})a_{t}u_{[0,1]}x.\]

 Then there are two possibilities:
 \begin{enumerate}[\ \ \ (i)]
   \item  either $B_{G}(e^{-\kappa t})a_{t}u_{[0,1]}x$  self intersects   somewhere in $\mathcal{H}_{1}(2)$, or
   \item $B_{G}(e^{-\kappa t})a_{t}u_{[0,1]}x$  can be locally observed in the $H_{\mathbb{C}}^{\perp}$-direction in the sense that there is some $x_{1}\in \mathcal{H}_{1}(2)$, and a finite set $F\subset B_{H_{\mathbb{C}}^{\perp}(x_{1})}(e^{-\kappa t})$ such that 
       \[x_{1}+F=B_{G}(e^{-\kappa t})a_{t}u_{[0,1]}x\cap [x_{1}+B_{H_{\mathbb{C}}^{\perp}(x_{1})}(e^{-\kappa t})]\]
       is a finite set with cardinality $|F|>1$. In fact, $F$ can be chosen so that $|F|\geq e^{t}$ for sufficiently large $t$.
       
\begin{figure}[H]
\centering

\tikzset{every picture/.style={line width=0.75pt}} 

\begin{tikzpicture}[x=0.75pt,y=0.75pt,yscale=-1,xscale=1]

\draw    (362.29,501.14) -- (407.24,501.14) ;
\draw    (264,382.29) -- (264,615.57) ;
\draw [color={rgb, 255:red, 255; green, 255; blue, 255 }  ,draw opacity=1 ][fill={rgb, 255:red, 255; green, 255; blue, 255 }  ,fill opacity=1 ]   (467.9,527.21) -- (317.24,527.14) -- (120.88,590.5) -- (271.55,590.57) -- cycle ;
\draw [color={rgb, 255:red, 0; green, 0; blue, 0 }  ,draw opacity=1 ][fill={rgb, 255:red, 255; green, 255; blue, 255 }  ,fill opacity=1 ]   (120.88,590.5) -- (271.55,590.57) ;
\draw [color={rgb, 255:red, 0; green, 0; blue, 0 }  ,draw opacity=1 ][fill={rgb, 255:red, 255; green, 255; blue, 255 }  ,fill opacity=1 ]   (120.88,590.5) .. controls (160.88,560.5) and (278.24,523.14) .. (317.24,527.14) ;
\draw [color={rgb, 255:red, 0; green, 0; blue, 0 }  ,draw opacity=1 ][fill={rgb, 255:red, 255; green, 255; blue, 255 }  ,fill opacity=1 ]   (271.55,590.57) .. controls (311.55,560.57) and (428.9,523.21) .. (467.9,527.21) ;
\draw [color={rgb, 255:red, 0; green, 0; blue, 0 }  ,draw opacity=1 ][fill={rgb, 255:red, 255; green, 255; blue, 255 }  ,fill opacity=1 ]   (317.24,527.14) -- (467.9,527.21) ;

\draw [color={rgb, 255:red, 0; green, 0; blue, 0 }  ,draw opacity=1 ][fill={rgb, 255:red, 255; green, 255; blue, 255 }  ,fill opacity=1 ]   (202.88,564.5) .. controls (242.88,534.5) and (360.24,497.14) .. (399.24,501.14) ;
\draw [color={rgb, 255:red, 255; green, 255; blue, 255 }  ,draw opacity=1 ][fill={rgb, 255:red, 255; green, 255; blue, 255 }  ,fill opacity=1 ]   (478.9,475.21) -- (328.24,475.14) -- (131.88,538.5) -- (282.55,538.57) -- cycle ;
\draw [color={rgb, 255:red, 0; green, 0; blue, 0 }  ,draw opacity=1 ][fill={rgb, 255:red, 255; green, 255; blue, 255 }  ,fill opacity=1 ]   (131.88,538.5) -- (282.55,538.57) ;
\draw [color={rgb, 255:red, 0; green, 0; blue, 0 }  ,draw opacity=1 ][fill={rgb, 255:red, 255; green, 255; blue, 255 }  ,fill opacity=1 ]   (131.88,538.5) .. controls (171.88,508.5) and (289.24,471.14) .. (328.24,475.14) ;
\draw [color={rgb, 255:red, 0; green, 0; blue, 0 }  ,draw opacity=1 ][fill={rgb, 255:red, 255; green, 255; blue, 255 }  ,fill opacity=1 ]   (282.55,538.57) .. controls (322.55,508.57) and (439.9,471.21) .. (478.9,475.21) ;
\draw [color={rgb, 255:red, 0; green, 0; blue, 0 }  ,draw opacity=1 ][fill={rgb, 255:red, 255; green, 255; blue, 255 }  ,fill opacity=1 ]   (328.24,475.14) -- (478.9,475.21) ;

\draw [color={rgb, 255:red, 255; green, 255; blue, 255 }  ,draw opacity=1 ][fill={rgb, 255:red, 255; green, 255; blue, 255 }  ,fill opacity=1 ]   (503.9,395.21) -- (353.24,395.14) -- (156.88,458.5) -- (307.55,458.57) -- cycle ;
\draw [color={rgb, 255:red, 0; green, 0; blue, 0 }  ,draw opacity=1 ][fill={rgb, 255:red, 255; green, 255; blue, 255 }  ,fill opacity=1 ]   (156.88,458.5) -- (307.55,458.57) ;
\draw [color={rgb, 255:red, 0; green, 0; blue, 0 }  ,draw opacity=1 ][fill={rgb, 255:red, 255; green, 255; blue, 255 }  ,fill opacity=1 ]   (156.88,458.5) .. controls (196.88,428.5) and (314.24,391.14) .. (353.24,395.14) ;
\draw [color={rgb, 255:red, 0; green, 0; blue, 0 }  ,draw opacity=1 ][fill={rgb, 255:red, 255; green, 255; blue, 255 }  ,fill opacity=1 ]   (307.55,458.57) .. controls (347.55,428.57) and (464.9,391.21) .. (503.9,395.21) ;
\draw [color={rgb, 255:red, 0; green, 0; blue, 0 }  ,draw opacity=1 ][fill={rgb, 255:red, 255; green, 255; blue, 255 }  ,fill opacity=1 ]   (353.24,395.14) -- (503.9,395.21) ;

\draw    (403,501.14) -- (361.5,501.14) ;
\draw [color={rgb, 255:red, 0; green, 0; blue, 0 }  ,draw opacity=1 ][fill={rgb, 255:red, 255; green, 255; blue, 255 }  ,fill opacity=1 ] [dash pattern={on 0.84pt off 2.51pt}]  (207.21,538.54) .. controls (247.21,508.54) and (364.57,471.18) .. (403.57,475.18) ;
\draw    (264,483.71) -- (264,514.14) ;
\draw [shift={(264,514.14)}, rotate = 90] [color={rgb, 255:red, 0; green, 0; blue, 0 }  ][fill={rgb, 255:red, 0; green, 0; blue, 0 }  ][line width=0.75]      (0, 0) circle [x radius= 3.35, y radius= 3.35]   ;
\draw [color={rgb, 255:red, 0; green, 0; blue, 0 }  ,draw opacity=1 ][fill={rgb, 255:red, 255; green, 255; blue, 255 }  ,fill opacity=1 ] [dash pattern={on 0.84pt off 2.51pt}]  (196.21,590.54) .. controls (236.21,560.54) and (353.57,523.18) .. (392.57,527.18) ;
\draw    (263.98,573.93) -- (263.98,564.93) ;
\draw [shift={(263.98,573.93)}, rotate = 270] [color={rgb, 255:red, 0; green, 0; blue, 0 }  ][fill={rgb, 255:red, 0; green, 0; blue, 0 }  ][line width=0.75]      (0, 0) circle [x radius= 3.35, y radius= 3.35]   ;
\draw [color={rgb, 255:red, 255; green, 255; blue, 255 }  ,draw opacity=1 ][fill={rgb, 255:red, 255; green, 255; blue, 255 }  ,fill opacity=1 ]   (549.9,501.21) -- (399.24,501.14) -- (202.88,564.5) -- (353.55,564.57) -- cycle ;
\draw [color={rgb, 255:red, 0; green, 0; blue, 0 }  ,draw opacity=1 ][fill={rgb, 255:red, 255; green, 255; blue, 255 }  ,fill opacity=1 ]   (393.24,501.14) -- (543.9,501.21) ;
\draw [color={rgb, 255:red, 0; green, 0; blue, 0 }  ,draw opacity=1 ][fill={rgb, 255:red, 255; green, 255; blue, 255 }  ,fill opacity=1 ]   (353.55,564.57) .. controls (393.55,534.57) and (510.9,497.21) .. (549.9,501.21) ;
\draw [color={rgb, 255:red, 0; green, 0; blue, 0 }  ,draw opacity=1 ][fill={rgb, 255:red, 255; green, 255; blue, 255 }  ,fill opacity=1 ]   (202.88,564.5) -- (353.55,564.57) ;
\draw [color={rgb, 255:red, 0; green, 0; blue, 0 }  ,draw opacity=1 ][fill={rgb, 255:red, 255; green, 255; blue, 255 }  ,fill opacity=1 ] [dash pattern={on 0.84pt off 2.51pt}]  (278.21,564.54) .. controls (318.21,534.54) and (435.57,497.18) .. (474.57,501.18) ;
\draw    (264,538.43) -- (264,545.43) ;
\draw [shift={(264,545.43)}, rotate = 90] [color={rgb, 255:red, 0; green, 0; blue, 0 }  ][fill={rgb, 255:red, 0; green, 0; blue, 0 }  ][line width=0.75]      (0, 0) circle [x radius= 3.35, y radius= 3.35]   ;
\draw [color={rgb, 255:red, 0; green, 0; blue, 0 }  ,draw opacity=1 ][fill={rgb, 255:red, 255; green, 255; blue, 255 }  ,fill opacity=1 ] [dash pattern={on 0.84pt off 2.51pt}]  (232.21,458.54) .. controls (272.21,428.54) and (389.57,391.18) .. (428.57,395.18) ;
\draw    (263.9,435.98) -- (263.9,447.05) ;
\draw [shift={(263.9,447.05)}, rotate = 90] [color={rgb, 255:red, 0; green, 0; blue, 0 }  ][fill={rgb, 255:red, 0; green, 0; blue, 0 }  ][line width=0.75]      (0, 0) circle [x radius= 3.35, y radius= 3.35]   ;
\draw [color={rgb, 255:red, 255; green, 255; blue, 255 }  ,draw opacity=1 ][fill={rgb, 255:red, 255; green, 255; blue, 255 }  ,fill opacity=1 ]   (588.4,372.71) -- (437.74,372.64) -- (241.38,436) -- (392.05,436.07) -- cycle ;
\draw [color={rgb, 255:red, 0; green, 0; blue, 0 }  ,draw opacity=1 ][fill={rgb, 255:red, 255; green, 255; blue, 255 }  ,fill opacity=1 ]   (241.38,436) -- (392.05,436.07) ;
\draw [color={rgb, 255:red, 0; green, 0; blue, 0 }  ,draw opacity=1 ][fill={rgb, 255:red, 255; green, 255; blue, 255 }  ,fill opacity=1 ]   (241.38,436) .. controls (281.38,406) and (398.74,368.64) .. (437.74,372.64) ;
\draw [color={rgb, 255:red, 0; green, 0; blue, 0 }  ,draw opacity=1 ][fill={rgb, 255:red, 255; green, 255; blue, 255 }  ,fill opacity=1 ]   (392.05,436.07) .. controls (432.05,406.07) and (549.4,368.71) .. (588.4,372.71) ;
\draw [color={rgb, 255:red, 0; green, 0; blue, 0 }  ,draw opacity=1 ][fill={rgb, 255:red, 255; green, 255; blue, 255 }  ,fill opacity=1 ]   (437.74,372.64) -- (588.4,372.71) ;

\draw    (263.9,410.98) -- (263.9,429.48) ;
\draw [shift={(263.9,429.48)}, rotate = 90] [color={rgb, 255:red, 0; green, 0; blue, 0 }  ][fill={rgb, 255:red, 0; green, 0; blue, 0 }  ][line width=0.75]      (0, 0) circle [x radius= 3.35, y radius= 3.35]   ;
\draw [color={rgb, 255:red, 0; green, 0; blue, 0 }  ,draw opacity=1 ][fill={rgb, 255:red, 255; green, 255; blue, 255 }  ,fill opacity=1 ] [dash pattern={on 0.84pt off 2.51pt}]  (316.71,436.04) .. controls (356.71,406.04) and (474.07,368.68) .. (513.07,372.68) ;

\draw (370,408) node [anchor=north west][inner sep=0.75pt]    {$a_{t}u_{[0,1]}x$};

\draw (200,355) node [anchor=north west][inner sep=0.75pt]    {$x_{1}+B_{H_{\mathbb{C}}^{\perp}(x_{1})}(e^{-\kappa t})$};
\draw (440,530) node [anchor=north west][inner sep=0.75pt]    {$B_{G}(e^{-\kappa t})a_{t}u_{[0,1]}x$};
\draw (275,510) node [anchor=north west][inner sep=0.75pt]    {$x_{1}$};

\end{tikzpicture}

  \caption{Case (ii).}
\label{margulis2025.1.3}
\end{figure}

 \end{enumerate}

 In case (i), we can apply the effective closing lemma (see Theorem \ref{effective2024.6.01}). It states that the $x$  is $e^{-Nt}$-close to a Teichm\"{u}ller curve with discriminant smaller than $e^{t}$.

 Now suppose that  case (i) does not occur. Then we obtain a finite set $F\subset H_{\mathbb{C}}^{\perp}(x_{1})$. For simplicity, we further assume that $F\subset H^{\perp}(x_{1})\subset W^{+}(x_{1})$. In the following, we briefly show that  the discretized fractal
dimension of $F$ can be improved via the Teichm\"{u}ller flow.  
To obtain this, we need to observe that the set $a_{t}F$ tends to be evenly distributed.  

Suppose  $y\in F$ is a highly concentrated point, i.e.
its local density
\[f_{\gamma}(y)=\sum_{w\in (x_{1}-y)+F}\|w\|_{y}^{-\gamma} \ \ \ (\gamma>0)\] 
is large. To decrease $f_{\gamma}(y)$, we apply the $a_{t}$-action. For simplicity, we assume for the moment that the expanding speed of the Teichm\"{u}ller flow is constant. Let $\lambda=\frac{1}{3}$ be the second Lyapunov exponent of the  Kontsevich-Zorich cocycle of $\mathcal{H}(2)$.
There is a sufficiently large $t>0$ such that 
      \[e^{(1-\lambda)t}\leq\frac{\|e^{t}w\|_{a_{t}y}}{\|w\|_{y}}\leq \frac{\|a_{t}w\|_{a_{t}y}}{\|w\|_{y}} \]
      for any $w  \in (x_{1}-y)+F$.
      Then  
      \[\|a_{t}w\|_{a_{t}y}^{-\gamma}\leq e^{-\gamma(1-\lambda)t}\|w\|_{y}^{-\gamma}\leq e^{-1}\|w\|_{y}^{-\gamma}\]
      for sufficiently large $t\geq t_{0}$.   Further, one may solve that $t_{0}$:
      \[t_{0}= \frac{1}{1-\lambda}.\]
       Then for points $z\in  F$ and $t\geq t_{0}$, we have 
      \[d(a_{t}y,a_{t}z)\geq e\cdot d(y,z).\]
      
   On the other hand, for points
   \[z_{c}\in B_{G}(\epsilon)a_{t}u_{[0,1]}x\cap B_{x_{1}}(e^{-\kappa t})^{c},\] we have 
   \begin{equation}\label{margulis2025.1.4}
     d(a_{t}y,a_{t}z_{c})\geq e^{-t_{0}}d(y,z_{c})=e^{-t_{0}}e^{-\kappa t}.
   \end{equation} 
      In other words, for points $z_{c}$ not coming from $F$, $a_{t}z_{c}$ is not too concentrated on $a_{t}y$ (e.g. a Cantor set). If it is the case for all points $y\in F$, then we see that $a_{t}F$ tends to be more ``evenly distributed" than $F$. Therefore, by replacing $F$ by $a_{t}F$, we may improve the dimension of $F$. See   Proposition \ref{margulis2024.4.16} for more details.
      
The general situation is much more complicated.
\begin{itemize}
        \item  The first difficulty is that $F$ belongs to $H_{\mathbb{C}}^{\perp}(x_{1})$ but not necessarily $H^{\perp}(x_{1})$. To fix this, we apply an exponential drift. That is, by replacing $F$ with $u_{r}F$ for some $r\in[0,1]$, we may assume that most points in $F$ have a big factor in $H_{\mathbb{C}}^{\perp}(x_{1})$. Then the previous argument is still applicable. However, we claim that
     \begin{cla}\label{margulis2025.1.5}
        The idea only works for the density function  
\[f_{\gamma}(y)=\sum_{w\in (x_{1}-y)+F}\|w\|_{y}^{-\gamma}\]
for any $\gamma\in(0,1)$. 
     \end{cla} 
\begin{proof}[Sketch proof of the claim] For $y\in F$, write $y=a+ib\in H^{\perp}(x_{1})\oplus i H^{\perp}(x_{1})$. Then $u_{r}y=(a+rb)+ib$ for $r\in[0,1]$.
  The worst case scenario is that $a$ is evenly distributed on $[0,1].(-b)$, e.g.
  \[y_{n}=-\frac{n}{|F|}b+ib,\ \ \ n=0,1,\ldots,|F|,\ \ \ \|b\|_{a_{t}x_{1}}\approx e^{-\lambda t}\|b\|_{x_{1}}.\]
  Then for any $r\in[0,1]$, the shift $u_{r}F$ always has ``bad points" $y=a+ib\in u_{r}F$, namely, $\frac{\|a\|}{\|y\|}$ is extremely small. It follows that in order to obtain
 \[\|a_{t}w\|_{a_{t}y}^{-\gamma}\leq  e^{-1}\|w\|_{y}^{-\gamma},\]
 the time $t_{0}$ becomes extremely large so that the lower bound in (\ref{margulis2025.1.4}) becomes useless. For example, if we choose $t=\frac{1}{2}\log|F|$, then  
   
\begin{align}
  f_{\gamma}(a_{t}x_{1})\geq& \sum_{a_{t}w\in a_{t}F}\|a_{t}w\|_{a_{t}x_{1}}^{-\gamma} \;\nonumber\\
\gtrsim  &   \frac{1}{b^{\gamma}}\sum_{k=1}^{|F|}\frac{1}{e^{(1-\lambda)\gamma t}(\frac{k}{|F|})^{\gamma}+e^{-(1+\lambda)\gamma t}}\;\nonumber\\
 =&\frac{|F|^{\frac{(1+\lambda)\gamma}{2}}}{b^{\gamma}}\sum_{k=1}^{|F|}\frac{1}{k^{\gamma}+1} .\;  \nonumber
\end{align}
Note that we want to choose a large $t$ so that the local density $f_{\gamma}(a_{t}x_{1})$ is smaller than $ f_{\gamma}(x_{1})$.
Then  for $\gamma\geq 1$, as $|F|\rightarrow\infty$,  $t$ has to be larger than $\frac{1}{2}\log|F|$ (for otherwise $f_{\gamma}(a_{t}x_{1})\rightarrow\infty$). But then $t\rightarrow\infty$ as well. 
 To avoid this, $\gamma$ has to be smaller than $1$. 
\end{proof}
\item The second difficulty is that $F$ the expanding speed of the Teichm\"{u}ller flow is not constant. To fix this, we need to apply the quantitative estimates of the Lyapunov exponents of the Teichm\"{u}ller flow.
      \end{itemize}

      Finally, due to  Claim \ref{margulis2025.1.5} (ultimately due to the $G$-action on $\mathcal{H}_{1}(2)$), this method can only show that the discretized fractal
dimension of $F$ is $\gamma \approx 1$. It would be interesting to know if there is a way to  quantify the symplectic structure of $H^{\perp}_{\mathbb{C}}$, and then bootstrap the exponential drift.

\subsection{Structure of the paper}
In Section \ref{closing2024.3.23} we recall basic definitions, including  some basic material on the translation surfaces and Teichm\"{u}ller curves (in Section \ref{closing2024.3.24}, Section \ref{closing2024.3.25}). In particular, we study the period map via triangulation (Section \ref{closing2024.3.17}) and gives quantitative estimates in terms of Avila-Gou\"{e}zel-Yoccoz norm (Section \ref{closing2024.3.26}).

In Section \ref{closing2024.3.27}, we review the effectiveness of the geodesic flows on the homogeneous dynamics. It provides    the Zariski density of  Veech groups via an effective lattice point counting, which  shall serve as a criterion of  Teichm\"{u}ller curves in $\mathcal{H}(2)$. 
 
In Sections  \ref{closing2024.3.28} and \ref{closing2024.3.5}, we review the dynamics over $\mathcal{H}(2)$. We first recall the McMullen's classification of Teichm\"{u}ller curves in $\mathcal{H}(2)$. From this, we deduce some   quantitative estimates of  Teichm\"{u}ller curves. In particular, in Corollary \ref{closing2024.3.15}, we deduce a quantitative discreteness of Teichm\"{u}ller curves with bounded discriminants.

In Sections \ref{closing2024.3.29} and \ref{margulis2024.12.3}, we provide effective closing lemmas  in Teichm\"{u}ller dynamics, which  shall serve as the initial dimension  in the direction  $W^{+}(x)\cap H^{\perp}_{\mathbb{C}}(x)$ (Proposition \ref{margulis2024.12.2}).   In particular, we prove Theorem \ref{closing2023.12.1}  by using the effective lattice point counting deduced in Section \ref{closing2024.3.27}.

  Finally,   we present in Section    \ref{margulis2024.3.04} the proof of Theorem \ref{margulis2024.12.1}. More precisely, we effectively improve the dimension of $P$-orbits in the direction  $W^{+}(x)\cap H^{\perp}_{\mathbb{C}}(x)$ via the Margulis function technique.

 \noindent
 \textbf{Acknowledgements.} 
 I would like to thank Pengyu Yang for the helpful discussion.  
I would also like to thank Giovanni Forni for sharing with me his insights. Last, I would like to express my appreciation for the support of Junyi Xie and Disheng Xu.

\section{Preliminaries}\label{closing2024.3.23}  
\subsection{Notation}
We will denote the metric on all relevant metric spaces by $d(\cdot,\cdot)$; where this may cause confusion, we will give the metric space as a subscript, e.g. $d_{X}(\cdot,\cdot)$ etc. Next, $B(x,r)$ denotes the open ball of radius $r$ in the metric space $x$ belongs to; where needed, the space we work in will be given as a subscript, e.g. $B_{T}(x,r)$. We will assume implicitly that for any $x\in X$ (as well as any other locally compact metric space we will consider) and $r>0$ the ball $B_{X}(x,r)$ is relatively compact.

We will use the asymptotic notation $A=O(B)$, $A\ll B$, or $A\gg B$,    for  positive quantities $A,B$ to mean  the estimate $|A|\leq CB$ for some constant $C$ independent of $B$. In some cases, we will need this constant $C$ to depend on a parameter (e.g. $d$), in which case we shall indicate this dependence by subscripts, e.g. $A=O_{d}(B)$ or  $A\ll_{d} B$. We also sometimes use $A\asymp B$ as a synonym for $A\ll B\ll A$.

Let $G=\SL_{2}(\mathbb{R})$, $\Gamma=\SL_{2}(\mathbb{Z})$.    Besides, we define 
\[B_{G}(T)\coloneqq\{g\in G:\|g-e\|\leq T\}\]
 where $e$ is the identity of $G$ and $\|\cdot\|$ is a fixed norm on the Euclidean space, e.g. $\|\cdot\|$ can be defined by
 \[\|g\|\coloneqq\max_{ij}\{|g_{ij}|,|g_{ij}^{-1}|\}\]
  where $g_{ij}$ is the $ij$-th entry of $g$. 
  Consider the diagonal $G$-action  on $G/\Gamma\times G/\Gamma$. Let $d_{G}(\cdot,\cdot)$ be a right-invariant metric on $G$.  We  abuse notation and use the same symbol $d_{G}$ to refer to the  metric on $G/\Gamma\times G/\Gamma$.

  We shall frequently use the following linear algebra lemma. Let
\begin{equation}\label{margulis2024.4.34}
  \mathsf{Q}_{G}(\delta,\tau)  \coloneqq \bar{u}_{[- \frac{\delta}{\tau}, \frac{\delta}{\tau}]}\cdot a_{[-\delta,\delta]}\cdot u_{[-\delta,\delta]}.
\end{equation}  
  \begin{lem}\label{margulis2024.4.35} Let $\delta,\epsilon\in(0,\frac{1}{100})$, $\tau\geq 1$, $r\in[0,2]$. Then   
\begin{alignat*}{8}
B_{G}(\delta)\cdot B_{G}(\epsilon)& \subset B_{G}(2(\delta+\epsilon)),  & \ \ \ &     & B_{G}(\delta\epsilon)\cdot B_{G}(\delta-2\delta\epsilon)& \subset B_{G}(\delta),  \\
\mathsf{Q}_{G}(\delta,\tau)^{\pm 1}\cdot \mathsf{Q}_{G}(\delta,\tau)^{\pm 1}& \subset \mathsf{Q}_{G}(10\delta,\tau), & & &    \mathsf{Q}_{G}(\delta,\tau)^{\pm 1}\cdot a_{\tau}u_{r}& \subset a_{\tau}u_{r} B_{G}(10\delta).
\end{alignat*}
  \end{lem}
  \begin{proof} The proof is straightforward with the following observation:
     For any $a,b,c,d\in\mathbb{R}$ with $ad-bc=1$ and $a\neq0$, we have
     \[\begin{bmatrix}
 a & b  \\
   c & d
\end{bmatrix}=\begin{bmatrix}
 1 & 0 \\
 c/a & 1
\end{bmatrix}\begin{bmatrix}
 a & 0  \\
   0 & 1/a
\end{bmatrix}\begin{bmatrix}
 1 & b/a  \\
   0 & 1
\end{bmatrix}.\]
The claim  follows from this identity.
  \end{proof}
  \subsection{Translation surfaces}\label{closing2024.3.24}
 Let $M$ be a compact oriented surface of genus $g$, and let $\Sigma\subset M$ be a nonempty finite set, called the set of zeroes. We make the convention that the points of $\Sigma$ are labeled. Let $\alpha=\{\alpha_{\sigma}:\sigma\in\Sigma\}$ be a partition of $2g-2$, so $\sum_{\sigma\in\Sigma} \alpha_{\sigma}=2g-2$.
 \begin{defn}[Translation surface]
   A surface $M$ is called a \textit{translation surface of type $\alpha$}\index{translation surface} if it has an affine atlas, i.e. a family of orientation preserving charts $\{(U_{a},z_{a})\}_{a}$ such that
   \begin{itemize}
     \item   the $U_{a}\subset M\setminus\Sigma$ are open and cover $M\setminus\Sigma$,
     \item the transition maps $z_{a}\circ z_{b}^{-1}$ have the form $z\mapsto z+c$.
   \end{itemize}  
   In addition, the planar structure of $M$ in a neighborhood of each $\sigma\in\Sigma$ completes to a cone angle singularity of total cone angle $2\pi(\alpha_{\sigma}+1)$. 
 \end{defn} 
 There are many equivalent definitions of a translation surface, and a convenient one is a pair $(M,\omega)$ consisting of a compact Riemann surface and  a holomorphic $1$-form $\omega$. We shall use these definitions interchangeably.
 
 There is a natural $\GL^{+}_{2}(\mathbb{R})$ action on the translation surfaces. Let $M$ be a translation surface with an atlas $\{(U_{a},z_{a})\}_{a}$.
  Since any matrix $h\in\GL^{+}_{2}(\mathbb{R})$ acts  on $\mathbb{C}=\mathbb{R}+i\mathbb{R}$, we obtain a new atlas $\{(U_{a},h\circ z_{a})\}_{a}$, which induces a new translation surface $hM$.

An \textit{affine isomorphism}\index{affine isomorphism} is an orientation preserving homeomorphism $f:M_{1}\rightarrow M_{2}$ which is affine  in each chart. If $M_{1}= M_{2}$, it is called an \textit{affine automorphism}\index{affine automorphism} instead. Let $\Aff(M)$ denote the set of affine automorphisms of $M$. 
 If an affine isomorphism whose linear part is $\pm\Id$ (for translation surfaces, $\Id$), it is called a \textit{translation equivalence}\index{translation equivalence}. Let $\mathcal{H}(\alpha)=\Omega \mathcal{M}_{g}(\alpha)$ denote the space of equivalence classes of translation surfaces of type $\alpha$. We refer to $\mathcal{H}(\alpha)$ as the \textit{moduli space of translation surfaces of type $\alpha$}\index{moduli space}.
 
  A \textit{saddle connection}\index{saddle connection} of $M$ is a geodesic segment joining two zeroes in $\Sigma$  or a zero to itself which has no zeroes in its interior.

 We fix a compact surface $(S,\Sigma)$ and  refer to it  as the model surface. 
   A \textit{marking map}\index{marking map} of a   surface $M$ is a homeomorphism $\varphi:(S,\Sigma)\rightarrow(M,\Sigma_{M})$ which preserves labels on $\Sigma$.  (We sometimes drop the subscript and use the same symbol $\Sigma$ to denote finite subsets of $S$ and of $M$, if no confusion  arises.) Two marking maps $\varphi_{1}:(S,\Sigma)\rightarrow (M_{1},\Sigma_{M_{1}})$ and  $\varphi_{2}:(S,\Sigma)\rightarrow (M_{2},\Sigma_{M_{2}})$ are said to be \textit{equivalent}\index{equivalent marking maps} if there is a translation equivalence $f:M_{1}\rightarrow M_{2}$ such that
   \begin{itemize}
     \item   $f\circ \varphi_{1}$ is isotopic to $\varphi_{2}$,
     \item $f$ maps $\Sigma_{M_{1}}\rightarrow \Sigma_{M_{2}}$ respecting the labels.
   \end{itemize}  
   An equivalent class of translation surfaces with marking maps is a \textit{marked translation surface}\index{marked translation surface}. The space of marked translation surfaces of type $\alpha$ is denoted by $\mathcal{TH}(\alpha)=\Omega \mathcal{T}_{g}(\alpha)$. We refer to $\mathcal{TH}(\alpha)$ as the \textit{Teichm\"{u}ller space of marked translation surfaces of type $\alpha$}\index{Teichm\"{u}ller space}. By forgetting the marking maps, we get a natural map  $\pi:\mathcal{TH}(\alpha)\rightarrow \mathcal{H}(\alpha)$.

  We can locally identify $\mathcal{TH}(\alpha)$ (and so $\mathcal{H}(\alpha)$) with $H^{1}(M,\Sigma;\mathbb{C})$. Let $\tilde{x}\in\mathcal{TH}(\alpha)$ be a marked translation surface with the marking $\varphi:(S,\Sigma)\rightarrow(M,\Sigma)$. 
   Suppose $M$ is equipped with a holomorphic $1$-form $\omega$.  Then the \textit{period map}\index{period map} 
   \[\omega^{\prime}\mapsto \left(\gamma\mapsto\int_{\gamma}\omega^{\prime}\right)\]
   from a neighborhood of $\omega$ to $H^{1}(M,\Sigma;\mathbb{C})$ gives a local homeomorphism.  
   Let $\tilde{x}\in\mathcal{TH}(\alpha)$ be a marked translation surface with the marking $\varphi:(S,\Sigma)\rightarrow(M,\Sigma)$. Suppose $M$ is equipped with a holomorphic $1$-form $\omega$.  
   Then after using the marking map $\varphi$ to pullback $\omega$, we get a distinguished element $\hol_{\tilde{x}}=\varphi^{\ast}(\omega)\in H^{1}(S,\Sigma;\mathbb{R}^{2})\cong H^{1}(S,\Sigma;\mathbb{C})$. Thus, if     $\gamma\in H_{1}(S,\Sigma;\mathbb{Z})$ is an oriented curve in $S$ with  endpoints in $\Sigma$, then 
  \[\hol_{\tilde{x}}(\gamma)=\tilde{x}(\gamma)\coloneqq\omega(\varphi(\gamma)).\]   
  We also refer to the map $\hol:\mathcal{TH}(\alpha)\rightarrow H^{1}(S,\Sigma;\mathbb{C})$ as the \textit{developing map}\index{developing map} or \textit{period map}\index{period map}. It is a local homeomorphism (see  Lemma \ref{effective2023.9.20}).
  If we fix  $2g+|\Sigma|-1$ curves $\gamma_{1},\ldots, \gamma_{2g+|\Sigma|-1}$ that form a basis for $H_{1}(S,\Sigma;\mathbb{Z})$, then  it defines the  \textit{period coordinates}\index{period coordinates} $\phi:\mathcal{TH}(\alpha)\rightarrow  \mathbb{C}^{2g+|\Sigma|-1}$ by
   \[\phi:\tilde{x}\mapsto\left(  \hol_{\tilde{x}}(\gamma_{i})\right)_{i=1}^{2g+|\Sigma|-1}.\]
    It is convenient to assume that the basis is obtained by fixing a triangulation $\tau$ of the surface  by saddle connections of $x$ (see Definition \ref{closing2024.1.3}).  Via the \textit{Gauss-Manin connection}\index{Gauss-Manin connection}, period coordinates endow $\mathcal{TH}(\alpha)$ with a canonical complex affine structure.

  Let $\Gamma=\Mod(M,\Sigma)$ be the group of isotopy classes of homeomorphisms $M$ which fix $\Sigma$ pointwise for a representative $(M,\Sigma)$ of the stratum $\alpha$. We will call this group the \textit{mapping class group}\index{mapping class group}. It acts on the right on $\mathcal{TH}(\alpha)$: letting $\tilde{x}\in \mathcal{TH}(\alpha)$ with a marking map $\varphi: (S,\Sigma)\rightarrow(M,\Sigma)$, $\gamma\in\Gamma$, we have the action \[\gamma.\varphi=\varphi\circ\gamma.\]
    The $\Gamma$-action on $\mathcal{TH}(\alpha)$ is properly discontinuous (e.g. \cite[Theorem 12.2]{farb2011primer}). Hence, $\mathcal{H}(\alpha)=\mathcal{TH}(\alpha)/\Gamma$ has an orbifold structure. 
   We   choose a fundamental domain $\mathcal{D}$ on $\mathcal{TH}(\alpha)$ for the action of $\Gamma$. Note that $\Gamma$ also acts on the right by linear automorphisms on $H^{1}(S,\Sigma;\mathbb{R})$. Let $R:\Gamma\rightarrow\Aut(H^{1}(S,\Sigma;\mathbb{R}))\cong\GL(2g+|\Sigma|-1,\mathbb{R})$. Since each element of $\Gamma$ is represented by an orientation-preserving homeomorphism of $S$, it follows that the image of $R$ lies in $\SL(2g+|\Sigma|-1,\mathbb{R})$.

    Let $\tilde{x}\in \mathcal{D}\subset \mathcal{TH}(\alpha)$ and $h\in\GL^{+}_{2}(\mathbb{R})$. Then there is a unique element $\gamma\in\Gamma$ so that $h\tilde{x}\gamma\in \mathcal{D}$. The   \textit{Kontsevich-Zorich cocycle}\index{Kontsevich-Zorich cocycle} is then defined by
   \[\Gamma(h,\tilde{x})\coloneqq \Gamma(\gamma).\] 
    Then for $\tilde{x}\in \mathcal{D}\subset \mathcal{TH}(\alpha)$,   the $G$-action becomes 
\begin{equation}\label{closing2023.08.6}
   h_{\tilde{x}}:\begin{bmatrix}
x_{1} & \cdots & x_{2g+|\Sigma|-1}\\
y_{1} & \cdots & y_{2g+|\Sigma|-1}
\end{bmatrix}\mapsto h\begin{bmatrix}
x_{1} & \cdots & x_{2g+|\Sigma|-1}\\
y_{1} & \cdots & y_{2g+|\Sigma|-1}
\end{bmatrix}\Gamma(h,\tilde{x}).
\end{equation} 

We say that a cocycle  is \textit{reductive}\index{reductive cocycle} under a representation if the representation  is semisimple, i.e. any invariant subspace has a complement.  
It is well known that the algebraic hull of Kontsevich-Zorich cocycle for the $\SL_{2}(\mathbb{R})$-action is reductive (see \cite[Appendix A]{eskin2018invariant}, \cite[Theorem 1.5]{avila2017symplectic}; 
   see also \cite{eskin2018algebraic} for the more precise algebraic hull of Kontsevich-Zorich cocycle). 
It leads to the fact that any $\SL_{2}(\mathbb{R})$-invariant subbundle has an invariant complement: 
\begin{thm}[Semisimplicity of $\SL_{2}(\mathbb{R})$-invariant subbundles, {\cite[Theorem 1.4]{filip2016semisimplicity}}]\label{closing2024.11.1} Let $E$ be a $\SL_{2}(\mathbb{R})$-invariant subbundle of any tensor power of the Hodge bundle over an affine invariant submanifold $\mathcal{M}$. Then it has a complement. 
\end{thm} 
   See \cite{filip2016semisimplicity} for the more precise discussion of   Kontsevich-Zorich cocycle. 
   
   In the literature, we sometimes refer to $\mathcal{TH}(\alpha)$ and $\mathcal{H}(\alpha)$ as a stratum of $\mathcal{TH}^{g}$ and $\mathcal{H}^{g}$, namely the Teichm\"{u}ller and  moduli spaces of translation surfaces of genus $g$, respectively. This is because we have the stratification
   \[\mathcal{TH}^{g}=\bigsqcup_{\alpha_{1}+\cdots+\alpha_{\sigma}=2g-2}\mathcal{TH}(\alpha),\ \ \ \ \ \ \mathcal{H}^{g}=\bigsqcup_{\alpha_{1}+\cdots+\alpha_{\sigma}=2g-2}\mathcal{H}(\alpha).\]
  
  On the other hand, let $\mathcal{T}_{g}$, $\mathcal{M}_{g}$ denote the Teichm\"{u}ller and moduli spaces  of Riemann surfaces of genus $g$ respectively. Let   $\Omega(M)$ denote the $g$-dimensional vector space of all holomorphic $1$-forms of $M$. Then we may consider $\mathcal{TH}^{g}$ and $\mathcal{H}^{g}$ as vector bundles over  $\mathcal{T}_{g}$, $\mathcal{M}_{g}$:
  \[\mathcal{TH}^{g}=\Omega \mathcal{T}_{g}\rightarrow \mathcal{T}_{g},\ \ \ \mathcal{H}^{g}=\Omega \mathcal{M}_{g}\rightarrow \mathcal{M}_{g}\] 
   whose fiber over $M$ is $\Omega(M)\setminus\{0\}$.
  
  Suppose that  $x=(M,\omega)\in\mathcal{H}(\alpha)$ is a translation surface of type $\alpha$.   Let $\Area(M,\omega)$ be the area of translation surface given by
  \[\Area(M,\omega)\coloneqq\frac{i}{2}\int_{M}\omega\wedge\bar{\omega}=\frac{i}{2}\sum_{j=1}^{g}(A_{j}(\omega)\bar{B}_{j}(\omega)-B_{j}(\omega)\bar{A}_{j}(\omega))\]
where $A_{j}(\omega), B_{j}(\omega)$ form a canonical basis of absolute periods of $\omega$, i.e.
\[A_{j}(\omega)=\int_{\alpha_{j}}\omega,\ \ \ B_{j}(\omega)=\int_{\beta_{j}}\omega \]
and $\{\alpha_{j},\beta_{j}\}_{j=1}^{g}$ is a symplectic basis of $H_{1}(M;\mathbb{R})$ (with respect to the intersection form). Let 
\[\mathcal{H}_{1}(\alpha)\coloneqq\{(M,\omega)\in\mathcal{H}(\alpha):\Area(M,\omega)=1\}.\] We see that  the normalized stratum $\mathcal{H}_{1}(\alpha)$ resembles more a ``unit hyperboloid".  Note that  $\mathcal{H}_{1}(\alpha)$ is a codimension one sub-orbifold of $\mathcal{H}(\alpha)$ but it is \textbf{not} an affine sub-orbifold.  Let $\pi_{1}:\mathcal{H}(\alpha)\rightarrow \mathcal{H}_{1}(\alpha)$ be the normalization of the area.  We  abuse notation and use the same symbol $\pi_{1}:\mathcal{TH}(\alpha)\rightarrow \mathcal{H}_{1}(\alpha)$ to refer to the composition of the projection and normalization. 
   
  Let $\mu$ be the measure on $\mathcal{H}(\alpha)$ which is given by the pullback of the Lebesgue measure on $H^{1}(S,\Sigma;\mathbb{C})\cong \mathbb{C}^{2g+|\Sigma|-1}$.  We refer to $\mu$  as the \textit{Lebesgue} or the \textit{Masur-Veech measure}\index{Masur-Veech measure} on $\mathcal{H}(\alpha)$. Let $\mu_{(1)}$ be the $\SL(2,\mathbb{R})$-invariant Lebesgue (probability) measure on the ``hyperboloid" $\mathcal{H}_{1}(\alpha)$ defined by the disintegration of the Lebesgue measure $\mu$ on  $\mathcal{H}_{1}(\alpha)$, namely
      \[d\mu=r^{2g+|\Sigma|-2} d r\cdot d\mu_{(1)}.\] 
 
 \subsection{Dynamics of Teichm\"{u}ller flow}\label{closing2024.12.12}
   For $\tilde{x}\in \mathcal{TH}(\alpha)$, let $V(\tilde{x})$ be a subspace of $H^{1}(M,\Sigma;\mathbb{R}^{2})$.
   Let $V[\tilde{x}]$ be the image of $V(\tilde{x})$ under the     \textit{affine exponential map}\index{affine exponential map}, i.e.
    \[V[\tilde{x}]\coloneqq\{\tilde{y}\in \mathcal{TH}(\alpha):\tilde{y}-\tilde{x}\in V(\tilde{x})\}.\] 
  Depending on the context, we sometimes consider $V[x]$ to be a subset of $\mathcal{H}(\alpha)$.
   
      Let $p:H^{1}(M,\Sigma;\mathbb{R})\rightarrow H^{1}(M;\mathbb{R})$ be the forgetful map. In what follows, for $x\in \mathcal{H}(\alpha)$, we write $H^{1}_{F}(x)=H^{1}(M_{x},\Sigma_{x};F)$ for $F=\mathbb{R},\mathbb{C}$.
       For $x\in\mathcal{H}(\alpha)$, let
    \begin{equation}\label{invariant2022.6.8}
      H^{\perp}(x)\coloneqq\{v\in H^{1}_{\mathbb{R}}(x):p(\Ree \tilde{x})\wedge p(v)=p(\Imm \tilde{x})\wedge p(v)=0\}.
    \end{equation}
      Let  $W^{\pm}(x)\subset H^{1}(M,\Sigma;\mathbb{C})\cong H^{1}(M,\Sigma;\mathbb{R})\oplus iH^{1}(M,\Sigma;\mathbb{R})$ be defined by
    \begin{align}
W^{+}(x)\coloneqq&   \mathbb{R}(\Imm x)\oplus  H^{\perp}(x)=\{v\in H^{1}_{\mathbb{C}}(x): p(\Imm x)\wedge p(v)=0\}, \;\nonumber\\ 
W^{-}(x)\coloneqq &   i(\mathbb{R}(\Ree x)\oplus  H^{\perp}(x))= \{iv\in H^{1}_{\mathbb{C}}(x): p(\Ree x)\wedge p(v)=0\}.\; \nonumber
\end{align} 
  Then $W^{+}[x]$ and $W^{-}[x]$ have a global affine structure on $\mathcal{H}_{1}(\alpha)$ by the affine exponential map, and play the role of the unstable and stable foliations for the Teichm\"{u}ller flow $a_{t}$ on $\mathcal{H}_{1}(\alpha)$ for $t>0$ (e.g. \cite[\S3]{eskin2018invariant}, \cite[\S4]{avila2013small}). We also abuse notation and consider $H^{\perp}(x)$ as a subspace of the unstable leaf $W^{+}(x)$. Also, we define 
  \[H^{\perp}_{\mathbb{C}}(x)\coloneqq   H^{\perp}(x)\oplus i H^{\perp}(x) \]
  and call it  the \textit{balance space}\index{balance space} at $x$.
      
     Moreover, the Lyapunov spectrum of Teichm\"{u}ller flow $a_{t}$ with respect to an ergodic probability measure $\nu$ supported on a stratum $\mathcal{H}_{1}(\alpha)$ (with $\alpha=(\alpha_{1},\ldots,\alpha_{\sigma})$) has the form (\cite[Section 7]{kontsevich1997lyapunov}, \cite[Section 5]{zorich1994asymptotic})
\begin{multline}
  2\geq 1+\lambda_{2}^{\nu}\geq\cdots\geq 1+\lambda_{g}^{\nu}\geq \overbrace{1=\cdots=1}^{\sigma-1}\geq  1-\lambda_{g}^{\nu}\geq\cdots\geq 1-\lambda_{2}^{\nu}\geq 0=0\\
  \geq -1+\lambda_{2}^{\nu}\geq\cdots\geq -1+\lambda_{g}^{\nu}\geq \overbrace{-1=\cdots=-1}^{\sigma-1}\geq  -1-\lambda_{g}^{\nu}\geq\cdots\geq -1-\lambda_{2}^{\nu}\geq -2\label{teichmuller2022.6.3}
\end{multline}
where $1\geq \lambda_{2}^{\nu}\geq \cdots\geq \lambda_{g}^{\nu}\geq0$ are the nonnegative exponents of the KZ cocycle    with respect to the probability measure $\mu$ on $\mathcal{H}_{1}(\alpha)$.
      
      Then for instance, $H^{\perp}$ is direct sum of the Lyapunov subspaces corresponding to $1+\lambda_{2}^{\nu},1+\lambda_{3}^{\nu},\ldots, 1-\lambda_{2}^{\nu}$, and $iH^{\perp}$ is direct sum of the Lyapunov subspaces corresponding to $-1+\lambda_{2}^{\nu},-1+\lambda_{3}^{\nu},\ldots, -1-\lambda_{2}^{\nu}$. 
      
      In \cite{forni2002deviation}, Forni developed an effective control of $\lambda_{2}^{\nu}$:
      \begin{thm}[{\cite[Corollary 2.2]{forni2002deviation}}]\label{closing2024.12.6} There is a function $\Lambda^{+}:\mathcal{H}_{1}(\alpha)\rightarrow[0,1)$ such that the following property holds. 
      Let $\nu$ be any $a_{t}$-invariant ergodic probability measure on $\mathcal{H}_{1}(\alpha)$. Then  
      \[\lambda_{2}^{\nu}(x)\leq \int_{\mathcal{H}_{1}(\alpha)}\Lambda^{+}(\omega)d\nu(\omega)<1\]
      for $\nu$-a.e. $x\in \mathcal{H}_{1}(\alpha)$.
      \end{thm}

  \subsection{Teichm\"{u}ller curves}\label{closing2024.3.25} As we have seen, there is a natural
  $G=\SL_{2}(\mathbb{R})$ action on $\mathcal{H}_{1}(\alpha)$. We are then interested in its smallest $G$-orbit closure:
  \begin{defn}[Teichm\"{u}ller curve]
      A \textit{Teichm\"{u}ller curve}\index{Teichm\"{u}ller curve} $f:V\rightarrow\mathcal{M}_{g}$ is a finite volume hyperbolic Riemann surface $V$ equipped with a holomorphic, totally geodesic, generically 1-1 immersion into moduli space.  
  \end{defn}
 Let $(M,\omega)\in\mathcal{H}^{g}$ be a translation surface.  Recall that $\Aff(M)$ denotes the set of affine automorphisms. Consider the map $D:\Aff(M)\rightarrow G$ which assigns to an affine automorphism its linear part. It  has a finite kernel $\Gamma_{M}$, consisting of translation equivalences of $M$. The image $\SL(M,\omega)\coloneqq D(\Aff(M))$ is called the \textit{Veech group}\index{Veech group} of $M$. Then we have the short exact sequence:
 \begin{equation}\label{closing2024.3.20}
   0\rightarrow \Gamma_{M}\rightarrow\Aff(M)\rightarrow \SL(M,\omega)\rightarrow0.
 \end{equation} 
      The equivalent conditions for the lattice property of $\SL(M,\omega)$ has been studied by a vast literature (e.g. \cite{smillie2010characterizations} and references therein):  
  \begin{thm}\label{effective2023.10.3}
     For $x\in\Omega\mathcal{M}_{g}$, the following are equivalent:
     \begin{itemize}
       \item  The group $\SL(x)$ is a lattice in $G=\SL_{2}(\mathbb{R})$.
       \item The orbit $G.x$ is closed in $\Omega\mathcal{M}_{g}$.
       \item The projection of the orbit to $\mathcal{M}_{g}$ is a Teichm\"{u}ller curve.
     \end{itemize}
  \end{thm}
  In this case, we say $x$ \textit{generates}\index{generates the Teichm\"{u}ller curve} the Teichm\"{u}ller curve $V=\mathbb{H}/\SL(x)\rightarrow\mathcal{M}_{g}$. It follows that 
  \[\Omega V=\GL_{2}^{+}(\mathbb{R}).x\cong \GL_{2}^{+}(\mathbb{R})/\SL(x)\]
  can be regarded as a bundle over $V$. Thus, we also abuse the notion and refer to $f(V)$, or the circle bundle $\Omega V$ (and $\Omega_{1} V$) as a Teichm\"{u}ller curve, if no confusion arise.
  
  \subsection{Nondivergence}\label{closing2024.3.65}
  
      Define the systole function $\ell:\mathcal{TH}(\alpha)\rightarrow\mathbb{R}^{+}$ by the shortest length of a  saddle connection. Note that for all $\epsilon>0$, the set 
         \begin{equation}\label{effective2024.08.12}
      \mathcal{H}^{(\epsilon)}_{1}(\alpha)\coloneqq\{x\in\mathcal{H}_{1}(\alpha):\ell(x)\geq\epsilon\}
     \end{equation}  
          is compact. To say it differently, a sequence $x_{n}\in \mathcal{H}_{1}(\alpha)$ diverges to infinity iff $\ell(x_{n})\rightarrow0$. In addition, by the Siegel-Veech formula (see e.g. \cite{eskin2001asymptotic,avila2006exponential}), we have
       \begin{equation}\label{closing2023.07.13}
        \mu_{(1)}(\mathcal{H}_{1}(\alpha)\setminus \mathcal{H}^{(\epsilon)}_{1}(\alpha))=\mu_{(1)}\{x\in\mathcal{H}_{1}(\alpha):\ell(x)<\epsilon\}\asymp O(\epsilon^{2}).
       \end{equation}

      In this section, we follow the idea in \cite{eskin2001asymptotic} and \cite{athreya2006quantitative} to  discuss the non-divergence results. See also \cite[\S6]{avila2013small} and \cite[\S2]{eskin2022effective}.
      
      \begin{thm}[{\cite{eskin2001asymptotic,athreya2006quantitative}}]\label{effective2023.11.2} \hypertarget{2024.12.k2}  
        There exist a continuous function $V:\mathcal{H}_{1}(\alpha)\rightarrow [2,\infty)$, a compact subset $K^{\prime}_{\alpha}\subset \mathcal{H}_{1}(\alpha)$ and some $\kappa_{2}>0$ with the following property. For every $t^{\prime}$ and every $x\in\mathcal{H}_{1}(\alpha)$, there exist 
        \[s\in[0,1/2],\ \ \ t^{\prime}\leq t\leq \max\{2t^{\prime},\hyperlink{2024.12.k2}{\kappa_{2}}\log V(x)\}\]
        such that $a_{t}u_{s}x\in K^{\prime}_{\alpha}$. Further, there exists a constant \hypertarget{2023.11.C1}  $C_{1}>1$ such that 
        \[ \hyperlink{2023.11.C1}{C_{1}}^{-1}\leq \frac{V(x)}{\max\{\ell(x)^{-5/4},1\}}\leq \hyperlink{2023.11.C1}{C_{1}}.\]
        where $\ell$ denotes the systole function.
      \end{thm}

      We also need the following averaging nondivergence of horocyclic flows. 
      \begin{thm}[{\cite[Theorem 6.3]{minsky2002nondivergence}}]\label{closing2024.3.60} \hypertarget{2024.12.k3} 
         There are positive constants $C_{5},\kappa_{3},\rho_{0}$, depending only on $\alpha$, such that if $\tilde{x}\in\mathcal{TH}_{1}(\alpha)$, an interval  \hypertarget{2024.3.C5} $I\subset\mathbb{R}$, and $\rho\in(0,\rho_{0}]$ satisfy: 
         \[\sup_{s\in I}\ell(u_{s}\tilde{x})\geq\rho,\] 
         then  for any $\epsilon\in(0,\rho)$, we have 
         \begin{equation}\label{closing2024.3.61}
           |\{s\in I:\ell(u_{s}\tilde{x})<\epsilon\}|\leq \hyperlink{2024.3.C5}{C_{5}}\left(\frac{\epsilon}{\rho}\right)^{\hyperlink{2024.12.k3}{\kappa_{3}}}|I|.
         \end{equation} 
      \end{thm}
      
      \begin{cor}\label{closing2024.3.63} \hypertarget{2024.12.k4}
          There are positive constants $C_{6},\kappa_{4}$, depending only on $\alpha$, with the following property. Let $\epsilon>0$, $\eta>0$, and  $\tilde{x}\in\mathcal{TH}_{1}(\alpha)$. Let $I\subset[-10,10]$ be  an interval \hypertarget{2024.3.C6} with $|I|\geq \eta$. Then  we have
         \[|\{r\in I:\ell(a_{t}u_{r}\tilde{x})<\epsilon\}|\leq \hyperlink{2024.3.C6}{C_{6}}\epsilon^{\hyperlink{2024.12.k3}{\kappa_{3}}}|I|\]
         whenever $t\geq \hyperlink{2024.12.k4}{\kappa_{4}}|\log \ell(x)|+|\log \eta|+\hyperlink{2024.3.C6}{C_{6}}$.
      \end{cor}
      \begin{proof} Assume for simplicity that $I=[0,1]$. More general situation follows from a similar argument.
         Let $\epsilon_{1}>0$ satisfy $K^{\prime}_{\alpha}\subset \mathcal{H}_{1}^{(\epsilon_{1})}(\alpha)$. Let $\tilde{x}\in\mathcal{TH}_{1}(\alpha)$. Without loss of generality, we assume that $\ell(\tilde{x})\ll 1$. Let $t^{\prime}=\max\{1,\frac{\hyperlink{2024.12.k2}{\kappa_{2}}}{2}\log V(x)\}$. Then applying Theorem \ref{effective2023.11.2} to $x$ and $t^{\prime}$, there exist 
         \[s_{0}\in[0,1/2],\ \ \   t_{0}\in[1, \hyperlink{2024.12.k2}{\kappa_{2}}\log \hyperlink{2023.11.C1}{C_{1}}\ell(x)^{-5/4}]\]
         such that $x_{0}\coloneqq a_{t_{0}}u_{s_{0}}x\in K^{\prime}_{\alpha}$.
         
         Next, let $t^{\prime}=\max\{1,\frac{\hyperlink{2024.12.k2}{\kappa_{2}}}{2}\log V(x_{0})\}$, $C=\hyperlink{2024.12.k2}{\kappa_{2}}\log \hyperlink{2023.11.C1}{C_{1}}\epsilon_{1}^{-5/4}$. Applying Theorem \ref{effective2023.11.2} again to $x_{0}$ and $t^{\prime}$, we obtain that there exist
           \[s_{1}\in[0,1/2],\ \ \   t_{1}\in[1, C]\]
            such that $x_{1}\coloneqq a_{t_{1}}u_{s_{1}}x_{0}\in K^{\prime}_{\alpha}$. Repeating the argument, we further obtain   that there exist
           \[s_{i}\in[0,1/2],\ \ \   t_{i}\in[1, C]\]
            such that $x_{i}\coloneqq a_{t_{i}}u_{s_{i}}x_{i-1}\in K^{\prime}_{\alpha}$ for all $i\in\mathbb{N}$. One calculates that $x_{i}= a_{t(i)}u_{s(i)}x$ where
            \[t(i)\coloneqq t_{i}+\cdots +t_{0},\ \ \ \text{ and }\ \ \  s(i)\coloneqq\sum_{j=0}^{i}\frac{s_{j}}{e^{t_{j-1}+\cdots+t_{0}}}\leq\sum_{j=0}^{i}\frac{s_{j}}{e^{j}}\leq 1.\]
            
            Then we see that 
              \[\sup_{s\in [0,1]}\ell(u_{e^{t(i)}s}a_{t(i)}\tilde{x})=\sup_{s\in [0,1]}\ell(a_{t(i)}u_{s}\tilde{x})\geq\epsilon_{1}.\] 
              Moreover, note that $t(i)-t(i-1)=t_{i}\leq  C$. Then for any $t\geq t_{0}$, we have 
              \[\sup_{s\in [0,1]}\ell(u_{e^{t}s}a_{t}\tilde{x})=\sup_{s\in [0,1]}\ell(a_{t}u_{s}\tilde{x})\geq\epsilon_{1}e^{-C}.\] 
              Now applying Theorem \ref{closing2024.3.60}, we obtain that 
                   \[|\{s\in [0,1]:\ell(a_{t}u_{s}\tilde{x})<\epsilon\}|\leq \hyperlink{2024.3.C5}{C_{5}}\left(\epsilon_{1}e^{-C}\right)^{-\hyperlink{2024.12.k3}{\kappa_{3}}}\epsilon^{\hyperlink{2024.12.k3}{\kappa_{3}}}\]
                   whenever $t\geq \hyperlink{2024.12.k2}{\kappa_{2}}\log \hyperlink{2023.11.C1}{C_{1}}\ell(x)^{-5/4}$.
                   This establishes (\ref{closing2024.3.61}).
      \end{proof}
         
       In view of Corollary \ref{closing2024.3.63}, let $\epsilon_{0}>0$ be so that 
       \begin{equation}\label{effective2024.6.02}
       |\{r\in I:\ell(a_{t}u_{r}\tilde{x})<\epsilon_{0}\}|\leq \textstyle{\frac{1}{100}}|I|
       \end{equation} 
for any $I\subset[-10,10]$ with $|I|\geq \eta$, and $t\geq \hyperlink{2024.12.k4}{\kappa_{4}}|\log \ell(x)|+|\log \eta|+\hyperlink{2024.3.C6}{C_{6}}$.
    \subsection{Avila-Gou\"{e}zel-Yoccoz norm}\label{closing2024.3.26} 
 We now introduce the  AGY norm, first defined in  \cite{avila2006exponential}, some properties of which were further developed in \cite{avila2013small}. 
 \begin{defn}[AGY norm]\label{closing2024.12.9}
     For $\tilde{x}\in  \mathcal{TH}(\alpha)$ and any $c\in H^{1}(M,\Sigma;\mathbb{C})$, we define
   \[ \|c\|_{\tilde{x}} \coloneqq \sup_{\gamma}\frac{|c(\gamma)|}{|\int_{\gamma}\omega|}\]
   where $\gamma$ is a saddle connection of $\tilde{x}$. We refer to $\|\cdot\|_{\tilde{x}}$ as the \textit{Avila-Gou\"{e}zel-Yoccoz norm} or \textit{AGY norm} for short.
 \end{defn}
 
 Note first that by definition, for $w=a+ib\in H^{1}(M,\Sigma;\mathbb{C})$, we have 
 \begin{equation}\label{effective2024.6.05}
   \max\{\|a\|_{x},\|b\|_{x}\}\leq \|w\|_{x}\leq \|a\|_{x}+\|b\|_{x}.
 \end{equation} 
 
By construction, the AGY norm is invariant under the action of the mapping class group $\Gamma$. Thus, it induces a norm on the moduli space $\mathcal{H}(\alpha)$.

     It was shown in  \cite[\S2.2.2]{avila2006exponential} that   this defines a norm and the corresponding Finsler metric is complete.  
    For $\tilde{x},\tilde{y}\in \mathcal{TH}(\alpha)$, we define a distance  
     \[d (\tilde{x},\tilde{y})\coloneqq\inf_{\gamma}\int_{0}^{1}\|\gamma^{\prime}(r)\|_{\gamma(r)}dr\] where $\gamma$ ranges over smooth paths $\gamma:[0,1]\rightarrow \mathcal{TH}(\alpha)$ with $\gamma(0)=\tilde{x}$ and $\gamma(1)=\tilde{y}$. It also induces a quotient metric on $\mathcal{H}(\alpha)$.
     
 Due to the splitting
   \[H^{1}(M,\Sigma;\mathbb{C})=H^{1}(M,\Sigma;\mathbb{R})\oplus iH^{1}(M,\Sigma;\mathbb{R}),\] 
   we often write an element of $H^{1}(M,\Sigma;\mathbb{C})$ as  $a+ib$ for $a,b\in H^{1}(M,\Sigma;\mathbb{R})$. 
    Let $\tilde{x}\in \mathcal{TH}(\alpha)$. For every $r>0$, define 
    \[R(\tilde{x},r)\coloneqq\{\phi(\tilde{x})+a+ib: a,b\in H^{1}(M,\Sigma;\mathbb{R}),\ \|a+ib\|_{\tilde{x}}\leq r\}.\]
   Let   $r>0$ be so that $\phi^{-1}$ is a homeomorphism on $R_{\tilde{x}}(r)$. Let 
    \[B(\tilde{x},r)\coloneqq \phi^{-1}(R(\tilde{x},r)).\]
    We call it a \textit{period box}\index{period box} of radius $r$ centered at $\tilde{x}$. Using   \cite[Proposition 5.3]{avila2013small}, $B(\tilde{x},r)$ is well defined for all $r\in(0,1/2]$ and all $\tilde{x}\in\mathcal{TH}(\alpha)$.
  Let $\inj(\tilde{x})$ be the injectivity radius of $\tilde{x}$ under the affine exponential map.

     We have the following estimate:
     \begin{lem}\label{closing2023.07.1}
    Let  $\tilde{x}\in \mathcal{TH}(\alpha)$. Then  for all $\tilde{y},\tilde{z}\in B(\tilde{x}, \inj(\tilde{x})/50)$, we have
    \[\frac{1}{2}\|\tilde{y}-\tilde{z}\|_{\tilde{y}}\leq \|\tilde{y}-\tilde{z}\|_{\tilde{z}}\leq 2\|\tilde{y}-\tilde{z}\|_{\tilde{y}},\]
    and further
    \[\frac{1}{4}\|\tilde{y}-\tilde{z}\|_{\tilde{x}}\leq d(\tilde{y},\tilde{z})\leq 4\|\tilde{y}-\tilde{z}\|_{\tilde{x}}.\] 
     \end{lem} 
 \begin{proof}
   It is actually a rephrasing of   \cite[Proposition 5.3]{avila2013small}.   See also \cite[Lemma 3.3]{Chaika2023ergodic}. 
 \end{proof}

      \begin{lem}[{\cite[Lemma 5.1]{avila2013small}}]\label{closing2023.12.3}
      For $x\in\mathcal{H}(\alpha)$, $g\in G$, we have 
         \[d(x,gx)\leq d_{G}(e,g). \]
         In particular, via the Cartan decomposition $g=kak^{\prime}$, we have 
               \[d(x,gx)\leq \log\|a\|+4\pi. \]
      \end{lem}

      The following crude estimates are well known, e.g. \cite[Corollary 2.6]{chaika2020tremors}.
      \begin{thm}\label{effective2023.11.3}
           For all $s,t\in\mathbb{R}$, $x\in\mathcal{H}(\alpha)$, we have 
         \[ \|u_{s}v\|_{u_{s}x}\leq \left(1+\frac{s^{2}+|s|\sqrt{s^{2}+4}}{2}\right)\|v\|_{x}\]
         and 
         \[ \|a_{t}v\|_{a_{t}x}\leq e^{2|t|}\|v\|_{x}.\]
      \end{thm}

         \begin{lem}[{\cite[Lemma 2.6]{eskin2022effective}}]\label{effective2023.11.4} \hypertarget{2024.12.k5}   There \hypertarget{2023.11.C2} exist     $\kappa_{5}=\kappa_{5}(\alpha)>0$ and  $C_{2}>1$   so that for all $x\in \mathcal{H}_{1}(\alpha)$, the following hold.
      For any $r\in(0,\hyperlink{2023.11.C2}{C_{2}}\ell(x)^{ \hyperlink{2024.12.k5}{\kappa_{5}}}]$, any lift $\tilde{x}$ of $x$, the restriction of the covering map $\pi:\mathcal{TH}(\alpha)\rightarrow \mathcal{H}(\alpha)$ to $B(\tilde{x},r)$ is injective. 
      \end{lem}   
     In what follows, we shall deduce the certain elementary   lemmas. It will be used in the proof of  Margulis functions (see Section \ref{margulis2024.3.04}). 
 
\begin{lem}\label{effective2024.6.07} For \hypertarget{2024.12.t3} any $C>0$ and $\gamma>0$, there exists $t_{3}=t_{3}(\ell(x),C)>0$ and $\lambda^{\prime}=\lambda^{\prime}(C)<1$ such that  
   \[ \int_{I}e^{\gamma \int_{0}^{t}\Lambda^{+}(a_{s}u_{r}x)ds}dr \leq  C  e^{\gamma t}+  e^{\gamma\lambda^{\prime}t} \]
   for any $t\geq 2  \hyperlink{2024.12.t3}{t_{3}}$, interval $I\subset[0,1]$ with $|I|\geq C$. Here $\Lambda^{+}:\mathcal{H}_{1}(\alpha)\rightarrow[0,1)$ is the function given by Theorem \ref{closing2024.12.6}.
\end{lem}
\begin{proof} By Corollary \ref{closing2024.3.63}, there exists $\epsilon_{1}=\epsilon_{1}(C)>0$ such that 
\begin{equation}\label{effective2024.6.06}
 |\{r\in I:\ell(a_{t}u_{r}\tilde{x})<\epsilon_{1}\}|\leq C|I|
\end{equation} 
for any interval $I\subset[0,1]$ with $|I|\geq C$.
 We let
\[\lambda=\max_{x\in\mathcal{H}^{(\epsilon_{1})}_{1}(2)}\Lambda^{+}(x),\ \ \ \lambda^{\prime}=\frac{1+\lambda}{2}.\]
Then $\lambda,\lambda^{\prime}<1$.  Let   $I\subset[0,1]$  be an interval with $|I|\geq C$. By (\ref{effective2024.6.06}), let $\hyperlink{2024.12.t3}{t_{3}}\coloneqq \hyperlink{2024.12.k4}{\kappa_{4}}|\log \ell(x)|+|\log C|+\hyperlink{2024.3.C6}{C_{6}}$. Then for  $t\geq 2\hyperlink{2024.12.t3}{t_{3}}$, one estimate 
\[\int_{I}e^{\gamma \int_{\hyperlink{2024.12.t3}{t_{3}}}^{t}\Lambda^{+}(a_{s}u_{r}x)ds}dr\leq C e^{(t-\hyperlink{2024.12.t3}{t_{3}})\gamma}+ (1-C)  e^{(t-\hyperlink{2024.12.t3}{t_{3}})\gamma\lambda_{0}} . \]
 Then by Jensen's inequality, one calculates
\begin{align}
  \int_{I}e^{\gamma \int_{0}^{t}\Lambda^{+}(a_{s}u_{r}x)ds}dr  & = \int_{I}e^{\gamma\int_{0}^{\hyperlink{2024.12.t3}{t_{3}}}\Lambda^{+}(a_{s}u_{r}x)ds }e^{\gamma\int_{\hyperlink{2024.12.t3}{t_{3}}}^{t}\Lambda^{+}(a_{s}u_{r}x)ds }dr    \;\nonumber\\ 
  &\leq e^{\gamma \hyperlink{2024.12.t3}{t_{3}}}\int_{I}e^{\gamma\int_{\hyperlink{2024.12.t3}{t_{3}}}^{t}\Lambda^{+}(a_{s}u_{r}x)ds }dr    \;\nonumber\\ 
  &\leq \frac{e^{\gamma \hyperlink{2024.12.t3}{t_{3}}}}{t-\hyperlink{2024.12.t3}{t_{3}}}\int_{\hyperlink{2024.12.t3}{t_{3}}}^{t}\int_{I}e^{(t-\hyperlink{2024.12.t3}{t_{3}})\gamma\Lambda^{+}(a_{s}u_{r}x) }dr   ds  \;\nonumber\\  
    & \leq \frac{e^{\gamma \hyperlink{2024.12.t3}{t_{3}}}}{t-\hyperlink{2024.12.t3}{t_{3}}} \int_{\hyperlink{2024.12.t3}{t_{3}}}^{t} C e^{(t-\hyperlink{2024.12.t3}{t_{3}})\gamma}+ (1-C)  e^{(t-\hyperlink{2024.12.t3}{t_{3}})\gamma\lambda_{0}} ds \;\nonumber\\ 
        & \leq e^{\gamma \hyperlink{2024.12.t3}{t_{3}}}( C e^{(t-\hyperlink{2024.12.t3}{t_{3}})\gamma}+   e^{(t-\hyperlink{2024.12.t3}{t_{3}})\gamma\lambda_{0}} ) \;\nonumber\\ 
    &  =   C e^{t\gamma}+  e^{\gamma \hyperlink{2024.12.t3}{t_{3}}+(t-\hyperlink{2024.12.t3}{t_{3}})\gamma\lambda_{0}}  \leq  C e^{\gamma t}+   e^{\gamma\lambda^{\prime}t} .\;  \nonumber
\end{align}
This establishes the claim.
\end{proof}

\begin{lem}\label{margulis2024.3.03}  \hypertarget{2024.4.C7}  
  Let $\gamma\in(0,1)$, $x\in\mathcal{H}_{1}(\alpha)$, $0\neq w\in H^{1}(M,\Sigma;\mathbb{C})$. Then  \hypertarget{2024.12.t4} there exists a time $t_{4}=t_{4}(\ell(x),\gamma)>0$  such that for any $t\geq \hyperlink{2024.12.t4}{t_{4}}$, we have
  \[\int_{0}^{1}\|a_{t}u_{r}w\|_{a_{t}u_{r}x}^{-\gamma}dr\leq   e^{-2}\|w\|_{x}^{-\gamma} .\]
\end{lem}
\begin{proof}  Write $w=a+ib$ for $a,b\in H^{1}(x)$. Without loss of generality, we may assume $\|w\|_{x}=1$.  Then one calculates $a_{t}u_{r}w=e^{t}(a+rb)+ie^{-t}b$ 
and so  
\[\|a_{t}u_{r}w\|_{a_{t}u_{r}x}=\|e^{t}(a+rb)+ie^{-t}b\|_{a_{t}u_{r}x}.\]
Then by (\ref{effective2024.6.05}), we have
\[ \|a_{t}u_{r}w\|_{a_{t}u_{r}x}\leq \|e^{t}(a+rb)\|_{a_{t}u_{r}x}+\|e^{-t}b\|_{a_{t}u_{r}x}\leq 2\|a_{t}u_{r}w\|_{a_{t}u_{r}x}.\]

   If $\|b\|_{x}<1/3$, then $\|a\|_{x}-\|b\|_{x}>1/3$. Let $C>0$. Then by Lemma \ref{effective2024.6.07}, for $t>2\hyperlink{2024.12.t3}{t_{3}}=2\hyperlink{2024.12.t3}{t_{3}}(\ell(x),C)$, we have
\begin{align}
   \int_{0}^{1}\|a_{t}u_{r}w\|_{a_{t}u_{r}x}^{-\gamma}dr  & \leq  \int_{0}^{1}(e^{t}\|a+rb\|_{a_{t}u_{r}x})^{-\gamma}dr  \;\nonumber\\
   & \leq  \int_{0}^{1}(e^{t-\int_{0}^{t}\Lambda^{+}(a_{s}u_{r}x)ds}\|a+rb\|_{u_{r}x})^{-\gamma}dr  \;\nonumber\\ 
    & \leq  3^{\gamma} \int_{0}^{1}e^{-\gamma (t-\int_{0}^{t}\Lambda^{+}(a_{s}u_{r}x)ds)}(\|a+rb\|_{x})^{-\gamma}dr  \;\nonumber\\
     & \leq  3^{\gamma}e^{-\gamma t} (\|a\|_{x}-\|b\|_{x})^{-\gamma} \int_{0}^{1}e^{ \gamma\int_{0}^{t}\Lambda^{+}(a_{s}u_{r}x)ds}dr  \;\nonumber\\
      & \leq  3^{2\gamma}e^{-\gamma t}  (C e^{\gamma t}+  e^{\gamma\lambda_{0}^{\prime}t}) \;\nonumber\\
    &  \leq  3^{2\gamma}  C+  3^{2\gamma} e^{-\gamma(1-\lambda_{0}^{\prime})t}.\;  \label{margulis2024.4.02}
\end{align}
Then for sufficiently small $C$ and large $t$, the right hand side of (\ref{margulis2024.4.02}) can be smaller than $e^{-2}$.

Now assume that $\|b\|_{x}\geq 1/3$.    
For $R>0$, consider
\[I(R)\coloneqq\{r\in[0,1]:\|a+rb\|_{x}\leq R\}.\]
Then there exists an absolute constant $c_{1}>0$ such that $|I(R)|\leq c_{1}\|b\|_{x}^{-1}R$. Let 
\[J_{k}\coloneqq\{r\in[0,1]:e^{-(k+1)}\leq \|a+rb\|_{x}\leq e^{-k}\}.\] 
Then $|J_{k}|\leq c_{1}\|b\|_{x}^{-1}\cdot e^{-k}\leq 3c_{1}  e^{-k}$. Further, one calculates
\begin{align}
\int_{I_{1}^{c}}\|a_{t}u_{r}w\|_{a_{t}u_{r}x}^{-\gamma}dr &    \leq  \sum_{k=0}^{\infty}\int_{J_{k}}(e^{t}\|a+rb\|_{a_{t}u_{r}x})^{-\gamma}dr  \;\nonumber\\
  & \leq  \sum_{k=0}^{\infty}\int_{J_{k}}(e^{t-\int_{0}^{t}\Lambda^{+}(a_{s}u_{r}x)ds}\|a+rb\|_{u_{r}x})^{-\gamma}dr  \;\nonumber\\
    & \leq  3^{\gamma}  \sum_{k=0}^{\infty}\int_{J_{k}}(e^{t-\int_{0}^{t}\Lambda^{+}(a_{s}u_{r}x)ds}\|a+rb\|_{x})^{-\gamma}dr  \;\nonumber\\
    & \leq  3^{\gamma}  \sum_{k=0}^{\infty}\int_{J_{k}}(e^{t-\int_{0}^{t}\Lambda^{+}(a_{s}u_{r}x)ds}e^{-(k+1)})^{-\gamma}dr  \;\nonumber\\
      & =  3^{\gamma} e^{\gamma(1- t)} \sum_{k=0}^{\infty}e^{k\gamma}\int_{J_{k}}(e^{\gamma\int_{0}^{t}\Lambda^{+}(a_{s}u_{r}x)ds})dr  .\;  \nonumber
\end{align}  

Then we split the sum into two parts. Let $m_{1}=m_{1}(\gamma)\in \mathbb{N}$ so that 
\[\sum_{k=m_{1}+1}^{\infty}e^{k(\gamma-1)}<\frac{1}{1000}.\]
Then by Lemma \ref{effective2024.6.07}, for $t>2\hyperlink{2024.12.t3}{t_{3}}=2\hyperlink{2024.12.t3}{t_{3}}(\ell(x),C,\gamma)$, we have  
\begin{align}
 &   3^{\gamma} e^{\gamma(1- t)} \sum_{k=0}^{\infty}e^{k\gamma}\int_{J_{k}}(e^{\gamma\int_{0}^{t}\Lambda^{+}(a_{s}u_{r}x)ds})dr  \;\nonumber\\
  \leq & 3^{\gamma} e^{\gamma(1- t)} \sum_{k=0}^{m_{1}}e^{k\gamma}\int_{J_{k}}(e^{\gamma\int_{0}^{t}\Lambda^{+}(a_{s}u_{r}x)ds})dr  +3^{\gamma} e^{\gamma} \sum_{k=m_{1}+1}^{\infty}e^{k(\gamma-1)} \;\nonumber\\  
    \leq & 3^{\gamma} e^{\gamma(1- t)} \sum_{k=0}^{m_{1}}e^{k\gamma}\int_{J_{k}}(e^{\gamma\int_{0}^{t}\Lambda^{+}(a_{s}u_{r}x)ds})dr  + \frac{3^{\gamma} e^{\gamma}}{1000}\;\nonumber\\  
     \leq & 3^{\gamma} e^{\gamma(1- t)} \sum_{k=0}^{m_{1}}e^{k(\gamma-1)} (Ce^{\gamma t}+   e^{\gamma\lambda_{0}^{\prime}t} )  + \frac{3^{\gamma} e^{\gamma}}{1000}\;\nonumber\\  
   \leq &  \frac{3^{\gamma} e^{\gamma}}{1-e^{\gamma-1}}(C+ e^{-\gamma(1-\lambda_{0}^{\prime})t} )  +     \frac{3^{\gamma} e^{\gamma}}{1000}.\;  \label{effective2024.6.08}
\end{align}  
Finally, let $C=C(\gamma)=\frac{1}{1000}(\frac{3^{\gamma} e^{\gamma}}{1-e^{\gamma-1}})^{-1}$. Then for sufficiently large $t\geq \hyperlink{2024.12.t4}{t_{4}}(\ell(x),\gamma)$, we can make the right hand side of (\ref{effective2024.6.08}) smaller than $e^{-2}$.
\end{proof}

      \subsection{Triangulation}\label{closing2024.3.17}
      
      In application, we usually choose a triangulation for the period coordinates, i.e. fix  a triangulation $\tau$ of the surface and choose a sequence of saddle connections from $\tau$ which form a basis for $H_{1}(M,\Sigma;\mathbb{Z})$.   In this section, we follow the idea in \cite{masur1991hausdorff} to  discuss the period coordinates, and find a lower bound of the non-degenerate deformations of a triangulation. In particular, this gives a  lower bound of the injectivity radius of the period map $\phi:\mathcal{TH}(\alpha)\rightarrow H^{1}(M,\Sigma;\mathbb{C})$. We adopt the notation introduced in \cite{chaika2020tremors}. 
      \begin{defn}[geodesic triangulation] \label{closing2024.1.3}
      We say $\tau$ is a \textit{geodesic triangulation}\index{geodesic triangulation} of $x$ if it is a decomposition of the surface into triangles whose sides are saddle connections,  and whose vertices are singular points, which need not be distinct.
      \end{defn} 
In \cite[$\S$4]{masur1991hausdorff}, Masur and Smillie showed that every translation surface $x\in\mathcal{H}(\alpha)$ admits a \textit{Delaunay triangulation}\index{Delaunay triangulation} $\tau_{x}$, which is  a typical geodesic triangulation. By the construction, each triangle $\Delta\in\tau_{x}$ can be inscribed in a disk of radius not greater than the diameter $d(M)$ of $M$ (cf. {\cite[Theorem 4.4]{masur1991hausdorff}}).

  Let $\tilde{x}\in\mathcal{TH}(\alpha)$ be a marked translation surface with the marking map $\varphi:(S,\Sigma)\rightarrow(M,\Sigma)$, and $x=\pi(\tilde{x})=(M,\omega)\in\mathcal{H}(\alpha)$.
 Let $\tau_{\tilde{x}}$ denote the pullback of the Delaunay triangulation with vertices in $\Sigma$, from $(M,\Sigma)$ to  $(S,\Sigma)$. 
       
Note that the period map $\hol_{\tilde{x}}(\gamma)$  can be thought of as giving a map from the triangles of $\tau_{\tilde{x}}$ to triangles in $\mathbb{C}\cong\mathbb{R}^{2}$ (well-defined up to translation). Moreover, we can define a local inverse of the period map as follows.
 
 Let $U_{\tilde{x}}\subset H^{1}(S,\Sigma;\mathbb{C})$ be the collection of all cohomology classes which map each triangle of $\tau_{\tilde{x}}$ into a positively oriented non-degenerate triangle in $\mathbb{C}$. 
    Each $\nu\in U_{\tilde{x}}$ gives a translation surface $M_{\tilde{x},\nu}$ built by gluing together the corresponding triangles in $\mathbb{C}$ along parallel edges, and a   marking map $\varphi_{\tilde{x},\nu}:(S,\Sigma)\rightarrow (M_{\tilde{x},\nu},\Sigma)$, by taking each triangle of the triangulation $\tau_{\tilde{x}}$ of $S$ to the corresponding triangle of the triangulation of $M_{\tilde{x},\nu}$. 
    Let $\tilde{y}_{\tilde{x},\nu}\in\mathcal{TH}(\alpha)$ denote the marked translation surface corresponding to the marking map $\varphi_{\tilde{x},\nu}:(S,\Sigma)\rightarrow (M_{\tilde{x},\nu},\Sigma)$. 
    Let 
    \[V_{\tilde{x}}\coloneqq\{\tilde{y}_{\tilde{x},\nu}:\nu\in U_{\tilde{x}}\}\subset \mathcal{TH}(\alpha)\]
     and $\psi_{\tilde{x}}:U_{\tilde{x}}\rightarrow V_{\tilde{x}}$ be defined by
       \[\psi_{\tilde{x}}:\nu\mapsto\tilde{y}_{\tilde{x},\nu}.\] 
  Let $\phi:V_{\tilde{x}}\rightarrow U_{\tilde{x}}$ be the period map. 
 By construction, $\nu\in U_{\tilde{x}}$ agrees with $\phi(\tilde{y}_{\tilde{x},\nu})$  on edges of $\tau_{\tilde{x}}$, and these edges generate $H_{1}(S,\Sigma;\mathbb{Z})$. Thus, the map $\psi_{\tilde{x}}$ is an inverse to $\phi$. Thus, we obtain the following:
   \begin{lem}[{\cite[Lemma 1.1]{masur1991hausdorff}}]\label{effective2023.9.20} The map $\phi$ is injective and locally onto when restricted to $V_{\tilde{x}}$.
   \end{lem}

 Now let $x=(M,\omega)\in \mathcal{H}_{1}^{(\epsilon)}(\alpha)$; in other words, the shortest length of   saddle connections in $x$ is not smaller than $\epsilon$.  Let $\tau_{x}$ be  the Delaunay triangulation of $x$. Then by the construction, each triangle $\Delta\in\tau_{x}$ can be inscribed in a disk of radius not greater than the diameter $d(M)$ of $M$ (cf. {\cite[Theorem 4.4]{masur1991hausdorff}}). In addition, we can further control these quantities as follows.
 \begin{thm}[{\cite[Theorem 5.3, Proposition 5.4]{masur1991hausdorff}}]\label{effective2023.9.19}
 \hypertarget{2024.3.C3}   Let   $x=(M,\omega)\in \mathcal{H}_{1}(\alpha)$.  Then there exists a constant  $C_{3}>0$, such that for any $p\in\Sigma$, $M\setminus B(p,\hyperlink{2024.3.C3}{C_{3}})$ is contained in a union of disjoint metric cylinders.
   
   Moreover, the Delaunay triangulation $\tau_{x}$ of $x$ consists of edges which either have length $\leq \hyperlink{2024.3.C3}{C_{3}}$ or which cross a cylinder $C\subset M$ whose height $h$ is greater than its circumference $c$. If an edge crosses $C$, then its length $l$ satisfies $h\leq l\leq \sqrt{h^{2}+c^{2}}$. 
 \end{thm}
 After calculating the area, we immediately obtain:
 \begin{cor}\label{effective2023.9.22}
    Let $\epsilon>0$, $x=(M,\omega)\in \mathcal{H}_{1}^{(\epsilon)}(\alpha)$. Then the length $l$ of any edge of the Delaunay triangulation $\tau_{x}$ of $x$ is bounded above by $l\leq 2\epsilon^{-1}$. Also, the diameter $d(M)\leq 2 \epsilon^{-1}$.
 \end{cor}
 \begin{proof} Note that for any $x=(M,\omega)\in \mathcal{H}_{1}^{(\epsilon)}(\alpha)$, the circumference of a cylinder is not less than $\epsilon$.  Let $C\subset M$  be a cylinder  with height $h$ and circumference $c\geq\epsilon$. Then one can calculate the area $ch=\Area(C)\leq \Area(M)=1$.  The consequence follows from Theorem \ref{effective2023.9.19} immediately.
 \end{proof}

     We shall also need certain elementary analysis of inscribed triangles. Let $T_{1}$ be the space of ordered triples of points in $\mathbb{C}\cong\mathbb{R}^{2}$ modulo the action of the group of translations. Let $T_{2} \subset T_{1}$ be the set of triples with positive determinant. Let $T_{3}(\epsilon,d)\subset T_{2}$ be the set of isometry classes of triangles, with all edges of length not less than $\epsilon$, which can be inscribed in circles of radius not greater than $d$.  
      \begin{lem}[{\cite[Lemma 6.7]{masur1991hausdorff}}]\label{effective2023.9.21} \hypertarget{2024.3.C4}  There exists a constant $C_{4}>0$ satisfying the following property. Let $\epsilon,d>0$, $\Delta\in T_{3}(\epsilon,d)$. Let $\Delta^{\prime}\in T_{3}(\epsilon,d)$ be a triangle such that each vertex of $\Delta^{\prime}$ differs from the corresponding  vertex of $\Delta$ by at most $\hyperlink{2024.3.C4}{C_{4}}\epsilon^{2}/d$. Then  $\Delta^{\prime}$ is a non-degenerate triangle with the same orientation as $\Delta$.
      \end{lem}
      
      \begin{cor}\label{closing2024.1.4} Let $\epsilon>0$ be small enough, and let $x\in\mathcal{H}_{1}^{(\epsilon)}(\alpha)$.  Let $\tilde{x}\in\mathcal{TH}(\alpha)$ satisfy $\pi_{1}(\tilde{x})=x$. Let $\tilde{y}\in B(\tilde{x},\epsilon^{5})$. Suppose that $\Delta\in\tau_{\tilde{x}}$ is a triangle, and  $\delta_{1},\delta_{2}$ are two directed edges of     $\Delta$.  Then $\tilde{y}(\delta_{1}),\tilde{y}(\delta_{2})$  are not parallel.
      \end{cor}
      \begin{proof} Recall that $\tau_{\tilde{x}}$ is the Delaunay triangulation of $\tilde{x}$.  Then by Corollary \ref{effective2023.9.22}, the  lengths of edges  of the triangle $\Delta$ satisfy $\epsilon\leq |\tilde{x}(\gamma)|\leq 4\epsilon^{-1}$.  Then 
      \[|\tilde{y}(\gamma)-\tilde{x}(\gamma)|\leq \|\tilde{y}-\tilde{x}\|_{\tilde{x}}\cdot |\tilde{x}(\gamma)|\leq 4 \epsilon^{4} \ll \epsilon^{2}/d(x).\]  
      Then by Lemma \ref{effective2023.9.21}, we get that $\Delta^{\prime}\in\tau_{\tilde{y}}$ generated by $\tilde{y}(\delta_{1}),\tilde{y}(\delta_{2})$ is  a non-degenerate triangle with the same orientation as $\Delta$. 
      \end{proof}
      
      \begin{cor}\label{closing2024.3.62} \hypertarget{2024.12.k6} Let $\kappa_{6}=\hyperlink{2024.12.k5}{\kappa_{5}}+5>0$. Then for $x\in\mathcal{H}_{1}(\alpha)$, the composition of the affine exponential map and the covering map is injective on $R(\tilde{x},\ell(x)^{\hyperlink{2024.12.k6}{\kappa_{6}}})$.
      \end{cor}
         For   $\eta>0$, $z\in \mathcal{H}_{1}(\alpha)$, we define 
\begin{equation}\label{margulis2024.4.14}
 L(z)\coloneqq \ell(z)^{\hyperlink{2024.12.k6}{\kappa_{6}}},\ \ \ \ L(\eta)\coloneqq \eta^{\hyperlink{2024.12.k6}{\kappa_{6}}}.
\end{equation}  
Then clearly, $L(\eta)\leq L(z)$ whenever $z\in \mathcal{H}_{1}^{(\eta)}(\alpha)$. By Lemma \ref{effective2023.11.4},  for $w\in H^{1}(M,\Sigma;\mathbb{C})$, we have 
\[\|w\|_{z}\leq L(z)\ \ \ \Leftrightarrow\ \ \ \|w\|_{z} \leq \ell(z)^{\hyperlink{2024.12.k6}{\kappa_{6}}}.\]
Thus, by Corollary \ref{closing2024.3.62}, we conclude that $w\mapsto z+w$ is well-defined and injective on $\|w\|_{z} \leq  L(z)$. In particular, the map 
\begin{equation}\label{margulis2024.4.52}
  (g,w)\mapsto g(z+w)
\end{equation} 
 is injective on   $ B_{G}(\frac{L(z)}{3})\times B_{H^{\perp}_{\mathbb{C}}(z)}(\frac{L(z)}{3})$.

   \section{Spectral gap and effective lattice point counting}\label{closing2024.3.27}
    In this section, we shall state some facts related  to effective volume estimate on the homogeneous space $G/\Gamma$ via the theory of spectral gap.  The results in this section are known to the experts. We adopt the standard notation in \cite{einsiedler2009effective,venkatesh2010sparse}.
     \subsection{Basic properties of spectral gap}
     In this section, we recall some basic properties of spectral gap of $G=\SL_{2}(\mathbb{R})$.  The origin of ideas comes from \cite{cowling1988almost,ratner1987rate}.
     \begin{defn}[Spectral gap] We say that a unitary representation $(\pi,V)$ of $G$   possesses a \textit{spectral gap}\index{spectral gap} if there is a compactly supported probability measure $\nu$ on   $G$, and $\delta>0$, so that
 \[\|\pi(\nu)v\|<(1-\delta)\|v\|\]
 for all $v\in V$, where $\pi(\nu)\coloneqq\int_{G}\pi(g)d\nu(g)$ denotes the convolution operator.
 \end{defn}
 
 We shall use the following uniform spectral gap result:
     \begin{prop}\label{closing2024.1.5}
  Let $\Gamma$ be a Fuchsian group in $G$. If $\Gamma$ is not Zariski dense in $G$, then  the action of $G$ on $L^{2}(G/\Gamma)$ has a uniform spectral gap $\delta>0$ for all such $\Gamma$. 
     \end{prop}
     
     This is a direct consequence of the following lemma: 
   \begin{lem}[{\cite[Lemma 6.5]{einsiedler2009effective}}]\label{effective22}
   Let $\mathbf{L}< \SL_{2}$ be an algebraic subgroup of strictly lower dimension. Let  $L=\mathbf{L}(\mathbb{R})$ be the real points of $\mathbf{L}$. Then the left action of $H$ on $L^{2}(G/L)$ has a uniform spectral gap for all such $L$.
\end{lem}
Lemma  \ref{effective22} follows from the fact that representations of the real points of an algebraic group, on the real points of an algebraic homogeneous spaces, have spectral gaps. See {\cite[\S6.4]{einsiedler2009effective}} for more details. 
\begin{proof}[Proof of Proposition \ref{closing2024.1.5}]
Let $\mathbf{L}=\overline{\Gamma}^{Z}$ be the Zariski closure of $\Gamma$, and let $L=\mathbf{L}(\mathbb{R})$. Then by the assumption, $L$ satisfies $\dim L<\dim G$. On the other hand, the representation $L^{2}(G/\Gamma)$ may be regarded as the induced representation $\Ind_{L}^{G}V$ from $L$ to $G$, of $V=L^{2}(L/\Gamma)$. Note that if  $\nu$ a probability measure on $G$, then the convolution operator satisfies
\[\|\nu\|_{\op,\Ind_{L}^{G}V}\leq \|\nu\|_{\op,L^{2}(G/L)}\] 
   where $\|\cdot\|_{\op,W}$ refers to the operator norm of a given unitary representation $W$.  It follows that, if $L^{2}(G/L)$ has a spectral gap, so also does $\Ind_{L}^{G}V$. Then Proposition \ref{closing2024.1.5} follows from Lemma \ref{effective22}. 
\end{proof}
     Once we have a spectral gap, it is known that the decay of matrix coefficients has an effective control. More precisely, let $\Gamma\subset G$ be as in Proposition \ref{closing2024.1.5}.
      \begin{defn}[Sobolev norm]   Fix for all time a basis $\mathcal{B}$ for $\mathfrak{g}$. We can define the Sobolev norms $\mathcal{S}_{d}$ on the smooth subspace of $L^{2}(G/\Gamma)$ via 
     \[\mathcal{S}_{d}(f)\coloneqq\big(\sum_{\ord(\mathcal{D})\leq d} \|\mathcal{D}f\|^{2}\big)^{1/2}\]
     where the sum is taken over all monomials  $\mathcal{D}$ in $\mathcal{B}$ of order $\leq d$. 
   \end{defn}
   
   Next, we define   the \textit{Harish-Chandra spherical function}\index{Harish-Chandra spherical function} on $G$. Let $G=KAN$ be an \textit{Iwasawa decomposition}. Then there exists a projection map $a:G\rightarrow A$. 
   \begin{defn}[Harish-Chandra spherical function]
      The \textit{Harish-Chandra spherical function}\index{Harish-Chandra spherical function} $\varphi$ on $G$ is defined by
         \[\varphi(g)\coloneqq\int_{K}\rho(a(gk))dk\]
         where $\rho:A\rightarrow\mathbb{R}^{+}$ is the half-sum of the positive root of $A$ acting on $N$.
   \end{defn}
   $G$ also admits  a \textit{Cartan decomposition}\index{Cartan decomposition} $G=KAK$. The function $\varphi_{0}$ is bi-$K$-invariant and belongs to $L^{2+\epsilon}(G)$, for every $\epsilon>0$.  Moreover,
     given the spectral gap $\delta>0$ as in Proposition \ref{closing2024.1.5},  it is known that \cite{cowling1988almost} that
     there exists a corresponding spectral coefficient \hypertarget{2024.12.k7}   $\kappa_{7}=\kappa_{7}(\delta)>0$ so that the following property holds.  
     Let $g\in G$, and let $f_{1},f_{2}\in C^{\infty}(G/\Gamma)$.  
     Then we have the effective mixing estimate:
     \begin{equation}\label{closing2024.2.2}
      \langle g.f_{1},f_{2}\rangle\ll \varphi(g)^{  \hyperlink{2024.12.k7}{\kappa_{7}}}\cdot S_{1}(f_{1}) \cdot S_{1}(f_{2}) .
     \end{equation} 
   See e.g. \cite[\S9]{venkatesh2010sparse}, \cite[Appendix B]{avila2013small} for more details. 
    \begin{lem}[{\cite[Lemma 6.1]{einsiedler2009effective}}]\label{closing2024.2.6} Let $p\geq 1$ and  let $F\subset G$ be bi-$K$-invariant, i.e. $KFK=F$. Then \hypertarget{2024.3.C5}  there exists a constant $C_{5}>0$ such that 
    \begin{equation}\label{effective16}
     \frac{1}{(\Vol(F))^{2}} \iint_{g_{1},g_{2}\in F}\varphi(g_{1}g_{2}^{-1})^{\frac{1}{p}} dg_{1}dg_{2}\leq \hyperlink{2024.3.C5}{C_{5}} \Vol(F)^{-\frac{2}{3p}}. 
    \end{equation} 
  \end{lem}
  \begin{proof}
     First, by the  H\"{o}lder's inequality for $L^{3}$ and $L^{3/2}$, we have 
     \begin{equation}\label{closing2024.2.1}
      \frac{1}{\Vol(F)}\int_{F}\varphi (g)dg\ll \Vol(F)^{-\frac{1}{3}}.
     \end{equation}  
    Next, let  $q$ satisfy $1/p+1/q=1$. Then by  the H\"{o}lder's inequality for $L^{p}$ and $L^{q}$, we have
    \[\frac{1}{(\Vol(F))^{2}}\int_{g_{1},g_{2}\in F} \varphi(g_{1}g_{2}^{-1})^{\frac{1}{p}} dg_{1}dg_{2}\leq\left( \frac{1}{(\Vol(F))^{2}}\int_{g_{1},g_{2}\in F}\varphi(g_{1}g_{2}^{-1}) dg_{1}dg_{2}\right)^{\frac{1}{p}}.\] 
      Noting the identity $\int_{K}\varphi(g_{1}kg_{2})dk=\varphi(g_{1})\varphi(g_{2})$, the consequence follows.
  \end{proof}
  \subsection{Effective lattice point counting}
In this section, we shall derive upper bounds for the number of points of $\Gamma\subset G$ inside a big ball.  The following proposition is well-known.  See \cite{duke1993density,eskin1993mixing,margulis2004some,einsiedler2009effective} for examples of this technique. We give a
short elementary proof.
   \begin{prop}\label{effective73} Let $\varphi$ be the Harish-Chandra spherical function of $G$, and $\alpha=\alpha(\delta)>0$    the spectral coefficient presented as in (\ref{closing2024.2.2}).
  Let  $\Gamma\subset G$ be a Fuchsian group that is not Zariski dense, $B\subset G$ an open set. Then  \hypertarget{2024.3.C6}  there exists a constant $C_{6}>0$ so that the cardinality of any $1$-separated subset $\Delta\subset B\cap\Gamma$ is bounded by
    \begin{equation}\label{effective63}
     |\Delta|^{2}\leq \hyperlink{2024.3.C6}{C_{6}}  \iint_{\tilde{B}\times \tilde{B}}\varphi (g_{1}g_{2}^{-1})^{  \hyperlink{2024.12.k7}{\kappa_{7}}}dg_{1}dg_{2}
    \end{equation} 
  where     $\tilde{B}\coloneqq B_{G}(1)BB_{G}(1)$. 
\end{prop}
   \begin{proof}
      Let $h$ be a nonnegative smooth function supported in the neighborhood $W\coloneqq B_{G}(1/10)$ with $\int h=1$. 
 Let $\pi:G\rightarrow G/\Gamma$ be the projection and let $\pi_{\ast}:C_{c}(G)\rightarrow C_{c}(G/\Gamma)$ be  the natural projection map:
  \[\pi_{\ast}(f)( g\Gamma)\coloneqq\sum_{\gamma\in\Gamma}f( g\gamma).  \]
  Let $\pi^{\ast}$ be the pullback $C(G/\Gamma)\rightarrow C(G)$:
   \[\pi^{\ast}(f)( g)\coloneqq f( g\Gamma).  \]
    Then $f=1_{WBW}\ast h\in C_{c}(G)$ is a nonnegative smooth bump function satisfying $f\geq 1_{WB}$ and $\supp(f)\subset WB$. Then one calculates
    \begin{multline}  \label{closing2024.2.4}
      \langle\pi_{\ast}f,\pi_{\ast} f\rangle_{G/\Gamma}=   \langle f,\pi^{\ast}\pi_{\ast}f\rangle_{G} \\
     \geq   \langle 1_{WB},\pi^{\ast}\pi_{\ast}1_{WB}\rangle_{G}= \int_{WB}| g\Gamma\cap WB|dg=\int_{WB}|\Gamma \cap g^{-1}WB|dg.
    \end{multline} 
For each $\delta\in\Delta$, $w\in W$, we have
\[\delta^{-1}\Delta \subset \Gamma\cap \delta^{-1}B \subset\Gamma\cap \delta^{-1}w^{-1} WB.\]
In addition,
\[\bigcup_{\delta\in\Delta}W\delta \subset WB.\]
It follows that  
\begin{equation}\label{closing2024.2.3}
 \Vol(W)\cdot|\Delta|^{2}=\sum_{\delta\in\Delta} \Vol( W\delta)\cdot |\delta^{-1}\Delta |\leq \int_{WB}|\Gamma \cap g^{-1}WB|dg.
\end{equation} 
Combining (\ref{closing2024.2.4}) and (\ref{closing2024.2.3}), we get
     \begin{equation}\label{closing2024.2.5}
       \Vol(W)\cdot|\Delta|^{2}\leq  \langle\pi_{\ast}f,\pi_{\ast} f\rangle_{G/\Gamma}.
     \end{equation}
     
     On the other hand, by the effective mixing estimate (\ref{closing2024.2.2}), we get 
\[|\langle g. \pi_{\ast}h,\pi_{\ast}h\rangle|\ll\varphi(g)^{\hyperlink{2024.12.k7}{\kappa_{7}}}\]
     for $g\in G$. It follows that 
     \[ \langle\pi_{\ast}f,\pi_{\ast} f\rangle_{G/\Gamma}\ll \iint_{\tilde{B}\times \tilde{B}}\varphi (g_{1}g_{2}^{-1})^{\hyperlink{2024.12.k7}{\kappa_{7}}}dg_{1}dg_{2}.\]
     Using (\ref{closing2024.2.5}), the consequence follows. 
   \end{proof}
   Combining Lemma \ref{closing2024.2.6} and Proposition \ref{effective73}, we get 
   \begin{cor}\label{closing2024.3.18}  \hypertarget{2024.3.C7} Let  $\Gamma\subset G$ be a Fuchsian group that is not Zariski dense, $B_{G}(T)\subset G$ an open ball with   radius $T\geq 1$. Then the cardinality of any $1$-separated subset $\Delta\subset B\cap\Gamma$ is bounded by
      \[     |\Delta| \leq \hyperlink{2024.3.C7}{C_{7}} \Vol(B)^{1-\frac{1}{3}  \hyperlink{2024.12.k7}{\kappa_{7}}}\]
      for some absolute constants $C_{7}>0$,  $  \hyperlink{2024.12.k7}{\kappa_{7}}>0$.   
   \end{cor} 
     
  \section{Dynamics over $\mathcal{H}(2)$}\label{closing2024.3.28}
   \subsection{McMullen's  classification}\label{closing2024.12.8} In this section, we manage to combine  the quantitative discreteness developed in Section \ref{closing2024.3.5}, with the McMullen's classification of Teichm\"{u}ller curves over  $\mathcal{H}(2)$. We shall first
    recall the dynamics of $G=\SL_{2}(\mathbb{R})$ over $\mathcal{H}(2)$.  
    We  refer to $\mathcal{H}(0)=\Omega\mathcal{M}_{1}$ as  the moduli space of holomorphic $1$-forms of genus $1$ with a marked point. We usually identify elements of $\mathcal{H}(0)$   with   lattices $\Lambda\subset\mathbb{C}$, via the correspondence $(M,\omega)=(\mathbb{C}/\Lambda,dz)$; in other words, $\Lambda$ is the image of the absolute periods $\omega(H_{1}(M;\mathbb{Z}))$. Thus, $\mathcal{H}(0)\cong \GL_{2}^{+}(\mathbb{R})/\Gamma=\mathbb{R}^{+}\times G/\Gamma$.  (Here $G/\Gamma$ is identified with the space of tori of area $1$.)
       
        In the sequel, we shall study tori with different areas: $\Lambda_{1}\in  \mathcal{H}_{A}(0)$ and   $\Lambda_{2}\in  \mathcal{H}_{1-A}(0)$. We consider $(\Lambda_{1},\Lambda_{2})\in G/\Gamma\times G/\Gamma$ as two corresponding tori with area $1$, after rescaling.  
        
    For $(X,\omega)\in\mathcal{H}(2)$, we will be interested in presenting forms of genus $2$ as connected sums of forms of genus $1$, 
    \begin{equation}\label{closing2024.3.8}
      (X,\omega)=(E_{1},\omega_{1})\stackrel[I]{}{\#} (E_{2},\omega_{2}).
    \end{equation} 
Here $(E_{1},\omega_{1}),(E_{2},\omega_{2}) \in\mathcal{H}(0)$ , $v\in\mathbb{C}^{\ast}$ and $I=[0,v]\coloneqq [0,1]\cdot v$. We also say   that $(E_{1},\omega_{1})\stackrel[I]{}{\#} (E_{2},\omega_{2})$ is a splitting of $(X,\omega)$.

It is straightforward to check that the connected sum operation commutes with the action of $\GL_{2}^{+}(\mathbb{R})$: we have 
\begin{equation}\label{two2023.6.14}
   g.((Y_{1},\omega_{1})\stackrel[I]{}{\#}(Y_{2},\omega_{2}))=   g.(Y_{1},\omega_{1})\stackrel[g\cdot I]{}{\#} g. (Y_{2},\omega_{2})
\end{equation}
for all $g\in\GL_{2}^{+}(\mathbb{R})$.

Let   $S(2)$ denote the set of triples $(\Lambda_{1},\Lambda_{2},v)\in\mathcal{H}(0)\times\mathcal{H}(0)\times\mathbb{C}^{\ast}$  
satisfying 
\begin{equation}\label{two2023.6.16}
   [0,v]\cap\Lambda_{1}=\{0\},\ \ \  [0,v]\cap\Lambda_{2}=\{0,v\}
\end{equation}
or vice versa.  The group $\GL^{+}_{2}(\mathbb{R})$ acts on the space of  triples $(\Lambda_{1},\Lambda_{2},v)$, leaving $S(2)$ invariant.  Clearly,  given a triple $(\Lambda_{1},\Lambda_{2},v)$, one defines
\[ x=\Lambda_{1}\stackrel[{[0,v]}]{}{\#} \Lambda_{2}\in \mathcal{H}(2)\]
and we obtain a natural map  $\Phi:S(2)\rightarrow \mathcal{H}(2)$. By {\cite[Theorem 7.2]{mcmullen2007dynamics}}, the connected sum mapping  $\Phi$ is a surjective, $\GL^{+}_{2}(\mathbb{R})$-equivariant local covering map.

    In \cite{mcmullen2007dynamics}, McMullen    classified the $\SL_{2}(\mathbb{R})$-orbit  closures of $\mathcal{H}_{1}(2)$:
  \begin{thm}[{\cite[Theorem 10.1]{mcmullen2007dynamics}}]\label{closing2024.12.7}
     Let $Z=\overline{G.x}$ be a $G$-orbit closure of some $x\in\mathcal{H}_{1}(2)$. Then either:
     \begin{itemize}
       \item $Z$ is a Teichm\"{u}ller curve, or
       \item $Z=\mathcal{H}_{1}(2)$.
     \end{itemize}
  \end{thm}
   We also have a simple criterion for Teichm\"{u}ller curves in $\mathcal{H}(2)$:
   \begin{thm}[{\cite[Theorems 5.8, 5.10]{mcmullen2007dynamics}}]\label{closing2024.3.19} A form $(X,\omega)\in\mathcal{H}(2)$ generates a Teichm\"{u}ller curve if and only if the Veech group $\SL(X,\omega)$ contains a hyperbolic element.
   \end{thm}
  
In addition, in \cite{mcmullen2005teichmullerDiscriminant}, McMullen provided a complete list of  Teichm\"{u}ller curves in $\mathcal{H}(2)$.  
 We say $W_{D}\subset\Omega\mathcal{M}_{2}$ is a \textit{Weierstrass curve}\index{Weierstrass curve} if it is the locus of Riemann surfaces $M\in\mathcal{M}_{2}$ such that 
  \begin{enumerate}[\ \ \ (i)]
    \item  $\Jac(M)$ admits real multiplication by $\mathcal{O}_{D}$, where $\mathcal{O}_{D}\cong\mathbb{Z}[x]/(x^{2}+bx+c)$ is a quadratic order with  $b,c\in\mathbb{Z}$ and the \textit{discriminant}\index{discriminant} $D=b^{2}-4c>0$ (cf. \cite{mcmullen2003billiards});
    \item $M$ carries an eigenform $\omega$ with a double zero at one of the six Weierstrass points of $M$.
  \end{enumerate}
  
  Every \textbf{irreducible} component of $W_{D}$ is a Teichm\"{u}ller curve. When $D\equiv 1\bmod 8$, one can also define a \textit{spin invariant}\index{spin invariant} $\epsilon(M,\omega)\in\mathbb{Z}/2\mathbb{Z}$ which is constant along the components of $W_{D}$.  In \cite{mcmullen2005teichmullerDiscriminant}, McMullen showed that each Teichm\"{u}ller curve is uniquely determined by these two invariants:
  \begin{thm}[{\cite[Theorem 1.1]{mcmullen2005teichmullerDiscriminant}}]\label{effective2023.10.1}
     For any integer $D\geq 5$ with $D\equiv0$ or $1\bmod 4$, either:
     \begin{itemize}
       \item The Weierstrass curve $W_{D}$ is irreducible, or
       \item We have $D\equiv 1\bmod 8$ and $D\neq 9$, in which case $W_{D}=W_{D}^{0}\sqcup W_{D}^{1}$ has exactly two components, distinguished by their spin invariants.
     \end{itemize}
  \end{thm}
  Intuitively, a Teichm\"{u}ller curve in $\mathcal{H}(2)$ of low discriminant in the sense of Definition \ref{closing2024.3.6} has a low discriminant introduced by McMullen. In this section, we shall show it is indeed the case. 
  
  We are interested in the information provided by the absolute periods.
  \begin{thm}[{\cite[Theorem 3.1]{mcmullen2005teichmullerDiscriminant}}]\label{effective2023.9.12}
     Let $(M,\omega)=(E_{1},\omega_{1})\stackrel[I]{}{\#} (E_{2},\omega_{2})$. Then the following are equivalent:
     \begin{enumerate}[\ \ \ (i)]
       \item  $\omega$ is an eigenform for real multiplication by $\mathcal{O}_{D}$ on $\Jac(M)$;
       \item $\omega_{1}+\omega_{2}$ is an eigenform for real multiplication by $\mathcal{O}_{D}$ on $E_{1}\times E_{2}$.
     \end{enumerate}
  \end{thm}
  In particular, the discriminant $D$ of a Weierstrass curve $W_{D}$ is purely determined by the absolute period map $I_{\omega}:H_{1}(M;\mathbb{Z})\rightarrow\mathbb{C}$. In addition, we have a complete list of all possible  absolute period maps. More precisely, we define the locus
  \begin{multline} 
    \Omega Q_{D}\coloneqq\{(E_{1}\times E_{2},\omega)\in\Omega\mathcal{M}_{1}\times\Omega\mathcal{M}_{1}: \\
    \omega \text{ is an eigenform for real muliplication by }\mathcal{O}_{D}\}.\nonumber
  \end{multline} 
 Besides, we define the splitting space
  \[\Omega W_{D}^{s}\coloneqq\{(X,\omega,I):(X,\omega)\in\Omega W_{D}\text{ splits along }I\}. \]
  Then by Theorem \ref{effective2023.9.12}, there is a covering map 
  \begin{equation}\label{closing2024.3.10}
    \Pi:\Omega W_{D}^{s}\rightarrow\Omega Q_{D}
  \end{equation} 
   which records the summands $(E_{i},\omega_{i})$ in (\ref{closing2024.3.8}).
  \subsection{Prototypes of splittings}
Let us say a triple of integers $(e,\ell,m)$ is a \textit{prototype}\index{prototype} for real multiplication, with discriminant $D$, if 
\begin{equation}\label{effective2023.9.14}
D=e^{2}+4\ell^{2}m,\ \ \ \ell,m>0,\ \ \  \gcd(e,\ell)=1.
\end{equation} 
 We can associate a prototype  $(e,\ell,m)$ to each eigenform $(E_{1}\times E_{2},\omega)\in  \Omega Q_{D}$. Moreover, we have
\begin{thm}[{\cite[Theorem 2.1]{mcmullen2005teichmullerDiscriminant}}]\label{effective2023.9.24}
   The space $\Omega Q_{D}$ decomposes into a finite union
   \[\Omega Q_{D}=\bigcup \Omega Q_{D}(e,\ell,m)\]
   of closed $\GL_{2}^{+}(\mathbb{R})$-orbits, one for each prototype $(e,\ell,m)$. Besides, we have
   \[\Omega Q_{D}(e,\ell,m)\cong\GL_{2}^{+}(\mathbb{R})/\Gamma_{0}(m)\]
   where $\Gamma_{0}(m)$ is the \textit{Hecke congruence subgroup}\index{Hecke congruence subgroup} of level $m$:
   \[\Gamma_{0}(m)\coloneqq \left\{\left[
            \begin{array}{cccc}
   a & b  \\
   c & d   \\
            \end{array}
          \right]\in\Gamma:c\equiv 0\bmod m\right\}.\]
\end{thm}  
Let $\lambda=(e+\sqrt{D})/2$. Define a pair of lattices in $\mathbb{C}$ by 
\begin{equation}\label{effective2023.9.13}
 \Lambda_{1}=\mathbb{Z}(\ell m,0)\oplus\mathbb{Z}(0,\ell),\ \ \         \Lambda_{2}=\mathbb{Z}(\lambda,0)\oplus\mathbb{Z}(0,\lambda).
\end{equation} 
Let $(E_{i},\omega_{i})=(\mathbb{C}/\Lambda_{i},dz)$ be the corresponding forms of genus $1$, and let 
\[(A,\omega)=(E_{1}\times E_{2},\omega_{1}+\omega_{2}).\]
Then $(A,\omega)$ is an eigenform with invariant $(e,\ell,m)$, and we refer to it as the \textit{prototypical example}\index{prototypical example} of type $(e,\ell,m)$.
\begin{cor}\label{effective2023.9.11} 
   Every eigenform $(E_{1}\times E_{2},\omega)\in\Omega Q_{D}$ is equivalent, under the action of $\GL_{2}^{+}(\mathbb{R})$, to a unique prototypical example.
\end{cor}
Moreover, we can assign a prototypical splitting to a quadruple of integers $(a,b,c,e)$.
First, we 
 say a quadruple of integers $(a,b,c,e)$ is a \textit{prototype}\index{prototype} of discriminant $D$, if   
\begin{alignat}{5}
 &  D=e^{2}+4bc,\ \ \ & & 0\leq a<\gcd(b,c),\ \ \  & &   c+e<b, \nonumber\\
& b>0, & & c>0, &  & \gcd(a,b,c,e)=1. \label{closing2024.3.9}
\end{alignat} 
We then assign a prototypical splitting to a prototype quadruple $(a,b,c,e)$ as follows. Let 
\[
 \Lambda_{1}=\mathbb{Z}(b,0)\oplus\mathbb{Z}(a,c),\ \ \         \Lambda_{2}=\mathbb{Z}(\lambda,0)\oplus\mathbb{Z}(0,\lambda)\]
 and $\lambda=(e+\sqrt{D})/2$ and $D=e^{2}+4bc$ (see Figure \ref{closing2024.3.7}).
\begin{figure}[H]
\centering

\tikzset{every picture/.style={line width=0.75pt}} 

\begin{tikzpicture}[x=0.75pt,y=0.75pt,yscale=-1,xscale=1]

\draw    (165,71) -- (301.11,71) ;
\draw    (165,71) -- (121.11,140.67) ;
\draw [shift={(121.11,140.67)}, rotate = 122.21] [color={rgb, 255:red, 0; green, 0; blue, 0 }  ][fill={rgb, 255:red, 0; green, 0; blue, 0 }  ][line width=0.75]      (0, 0) circle [x radius= 3.35, y radius= 3.35]   ;
\draw    (301.11,71) -- (257.22,140.67) ;
\draw [shift={(257.22,140.67)}, rotate = 122.21] [color={rgb, 255:red, 0; green, 0; blue, 0 }  ][fill={rgb, 255:red, 0; green, 0; blue, 0 }  ][line width=0.75]      (0, 0) circle [x radius= 3.35, y radius= 3.35]   ;
\draw [shift={(301.11,71)}, rotate = 122.21] [color={rgb, 255:red, 0; green, 0; blue, 0 }  ][fill={rgb, 255:red, 0; green, 0; blue, 0 }  ][line width=0.75]      (0, 0) circle [x radius= 3.35, y radius= 3.35]   ;
\draw    (121.11,140.67) -- (257.22,140.67) ;
\draw    (121.11,140.67) -- (205.44,140.67) ;
\draw    (121.11,225) -- (121.11,140.67) ;
\draw    (121.11,225) -- (205.44,225) ;
\draw [shift={(121.11,225)}, rotate = 0] [color={rgb, 255:red, 0; green, 0; blue, 0 }  ][fill={rgb, 255:red, 0; green, 0; blue, 0 }  ][line width=0.75]      (0, 0) circle [x radius= 3.35, y radius= 3.35]   ;
\draw    (205.44,225) -- (205.44,140.67) ;
\draw [shift={(205.44,140.67)}, rotate = 270] [color={rgb, 255:red, 0; green, 0; blue, 0 }  ][fill={rgb, 255:red, 0; green, 0; blue, 0 }  ][line width=0.75]      (0, 0) circle [x radius= 3.35, y radius= 3.35]   ;
\draw [shift={(205.44,225)}, rotate = 270] [color={rgb, 255:red, 0; green, 0; blue, 0 }  ][fill={rgb, 255:red, 0; green, 0; blue, 0 }  ][line width=0.75]      (0, 0) circle [x radius= 3.35, y radius= 3.35]   ;
\draw    (165,71) -- (249.33,71) ;
\draw [shift={(249.33,71)}, rotate = 0] [color={rgb, 255:red, 0; green, 0; blue, 0 }  ][fill={rgb, 255:red, 0; green, 0; blue, 0 }  ][line width=0.75]      (0, 0) circle [x radius= 3.35, y radius= 3.35]   ;
\draw [shift={(165,71)}, rotate = 0] [color={rgb, 255:red, 0; green, 0; blue, 0 }  ][fill={rgb, 255:red, 0; green, 0; blue, 0 }  ][line width=0.75]      (0, 0) circle [x radius= 3.35, y radius= 3.35]   ;

\draw (145,45) node [anchor=north west][inner sep=0.75pt]    {$( a,c)$};
\draw (267,135) node [anchor=north west][inner sep=0.75pt]    {$( b,0)$};
\draw (105,180) node [anchor=north west][inner sep=0.75pt]    {$\lambda$};
\draw (158,205) node [anchor=north west][inner sep=0.75pt]    {$\lambda$};

\end{tikzpicture}

  \caption{Prototypical splitting of type $(a,b,c,e)$.}
\label{closing2024.3.7}
\end{figure}
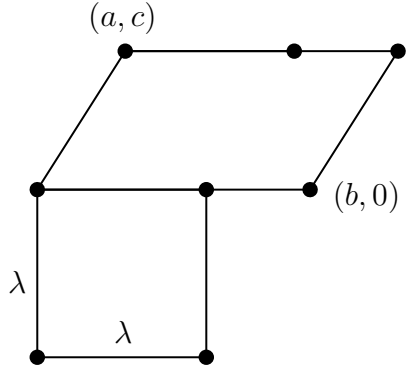
\noindent
Then the splittings $(X,\omega)=\Lambda_{1}\stackrel[{[0,v]}]{}{\#} \Lambda_{2}$ is said to be a \textit{prototypical splitting of type $(a,b,c,e)$}\index{prototypical splitting}. 
   In {\cite[\S3]{mcmullen2005teichmullerDiscriminant}}, McMullen  showed that all prototypical splittings are well-defined, and all other  splittings of $\Omega W^{s}_{D}$ can be generated by these prototypes via $\GL_{2}^{+}(\mathbb{R})$-action. In particular, one may check (\ref{closing2024.3.9}) that  
   \begin{itemize}
     \item  For $D\equiv 0\bmod 4$,  the Weierstrass curve $W_{D}$  contains a   prototypical splitting
         \begin{equation}\label{closing2024.3.11}
           x_{D}=\Lambda_{1}(D) \stackrel[I(D)]{}{\#} \Lambda_{2}(D)\in \Omega W_{D}
         \end{equation} 
          of type $(0,D/4,1,0)$. Under the projection map (\ref{closing2024.3.10}), $\Pi(x_{D})$ is of type $(0,1,D)$.
     \item For  $D\equiv 1\bmod 4$, $\epsilon=0,1$, there are    splittings 
  \[x_{D}^{(\epsilon)}\coloneqq \Lambda_{1}^{(\epsilon)}(D) \stackrel[I^{(\epsilon)}(D)]{}{\#} \Lambda_{2}^{(\epsilon)}(D)\in \Omega W_{D}^{\epsilon}\]
  of type $(0,1,(D-1)/4,(-1)^{\epsilon})$. Under the projection map,    (\ref{closing2024.3.10}), $\Pi(x_{D}^{(\epsilon)})$ is of type $( (-1)^{\epsilon},1,(D-1)/4)$.
   \end{itemize} 

\subsection{Density of Teichm\"{u}ller curves}
 After  constructing certain splittings, we are able to observe the density of Teichm\"{u}ller curves.  First, we recall the density of periodic $G$-orbit on the homogeneous space $G/\Gamma\times G/\Gamma$. Let 
  \[K(\delta)\coloneqq\{x\in G/\Gamma\times G/\Gamma:\inj(x)\geq \delta\} \]
  where $\inj(x)$ denotes the injectivity radius of $x$. We also abuse notation and refer to $d_{G}(\cdot,\cdot)$  as the metric on $G/\Gamma\times G/\Gamma$ induced by the right-invariant metric on $G\times G$.
 \begin{thm}[{\cite[Theorem 1.3]{lindenstrauss2023polynomial}}]\label{effective2022.2.12} \hypertarget{2024.12.k8} 
  Let $Y\subset G/\Gamma\times G/\Gamma$ be a periodic $G$-orbit in $G/\Gamma\times G/\Gamma$.  
  Then there exist $\kappa_{8}>0$ and $C_{11}>0$ such that 
  for every $z^{\ast}\in K(\vol(Y)^{-\hyperlink{2024.12.k8}{\kappa_{8}}})$, we have \hypertarget{2024.3.C11}
  \[d_{G}(z^{\ast},Y)\leq \hyperlink{2024.3.C11}{C_{11}}\vol(Y)^{-\hyperlink{2024.12.k8}{\kappa_{8}}}\] 
  where  $\vol(Y)$ is the volume of $Y$.
\end{thm}

Now we are able to show the connection between the density and the discriminant of  Teichm\"{u}ller curves in $\mathcal{H}(2)$.

 \begin{prop}\label{closing2024.3.22}\hypertarget{2024.12.k9} 
 There exist  $\kappa_{9}>0$, $C_{12}>0$ and  $D_{0}>0$ such that   
    \[\disc(\Omega W_{D})\geq  \hyperlink{2024.3.C12}{C_{12}}D^{\hyperlink{2024.12.k9}{\kappa_{9}}}\]
  \hypertarget{2024.3.C12} for $D\geq D_{0}$, where $\disc(\cdot)$ is given in Definition \ref{closing2024.3.6}.
 \end{prop}
 \begin{proof}    Suppose that $D\equiv 0\bmod 4$. Then by  Theorem \ref{effective2023.9.24}, the splitting (\ref{closing2024.3.11}) (after ignoring the areas of tori) generates a periodic $G$-orbit
\[Y_{D}\coloneqq G.(\Lambda_{1}(D) ,\Lambda_{2}(D))\subset G/\Gamma \times G/\Gamma,\ \ \ \text{ and }\ \ \ \vol(Y_{D})\geq D.\]
By Theorem \ref{effective2022.2.12}, we conclude that for every $z^{\ast}\in  K(D^{-\hyperlink{2024.12.k8}{\kappa_{8}}})$, we have 
\begin{equation}\label{closing2024.3.12}
 d_{G}(z^{\ast},Y_{D})\leq \hyperlink{2024.3.C11}{C_{11}}\vol(Y)^{-\hyperlink{2024.12.k8}{\kappa_{8}}}.
\end{equation} 
 
 On the other hand, let $K_{2}^{\prime}\subset\mathcal{H}_{1}(2)$ be as in Theorem \ref{effective2023.11.2}. Then there exists a  large $D_{1}> 1$ such that $K_{2}^{\prime}\subset\mathcal{H}_{1}^{(D_{1}^{-1})}(2)$ and  $\Omega W_{D}\cap K_{2}^{\prime}\neq \emptyset$ for all $D\geq D_{1}$.
 Fix $x\in \Omega W_{D}\cap K_{2}^{\prime}$. Then there exists $g\in G$ such that 
 \[x=gx_{D}=g\Lambda_{1}(D) \stackrel[gI(D)]{}{\#} g\Lambda_{2}(D).\]

 Recall that a splitting consists of two tori and a slit (see e.g. (\ref{two2023.6.16})). Let $^{\ast}:S(2)\rightarrow G/\Gamma\times G/\Gamma$ be the forgetting map from the space of splittings to the space of tori defined by
 \[^{\ast}:(\Lambda_{1},\Lambda_{2},v)\mapsto(\Lambda_{1},\Lambda_{2}).\]
 Let $x^{\ast}\coloneqq(g\Lambda_{1}(D),g\Lambda_{2}(D))\in G/\Gamma\times G/\Gamma$, and let  $\|\cdot\|^{\prime}$ be a norm on  $H^{1}(M,\Sigma;\mathbb{C})$. Then $\|\cdot\|^{\prime}$ is comparable to $d_{G}(\cdot,\cdot)$, in the sense that  
 \[ \|y-z\|^{\prime}\asymp_{D_{2}} d_{G}(y^{\ast},z^{\ast})\]
 for $y,z\in B(x,D_{2}^{- 1})$ for some $D_{2}\geq D_{1}$. In addition, by  Lemma \ref{effective2023.11.4} and the fact that $x\in\mathcal{H}_{1}^{(D_{1}^{-1})}(2)$, we get that  
 \begin{equation}\label{closing2024.3.14}
  \|y-z\|_{x}\asymp_{D_{2}} d_{G}(y^{\ast},z^{\ast})
 \end{equation}
 for $y,z\in B(x,D_{2}^{- 1})$.
 
 Therefore,  by (\ref{closing2024.3.14}) and $x\in\mathcal{H}_{1}^{(D_{1}^{-1})}(2)$ again, there exists $D_{3}\geq D_{2}$ so that for  $D\geq D_{3}$, we have 
  \[x^{\ast}=(g\Lambda_{1}(D),g\Lambda_{2}(D))\in     K(2\hyperlink{2024.3.C11}{C_{11}} D^{-\hyperlink{2024.12.k8}{\kappa_{8}}}).\]
  Then by
 (\ref{closing2024.3.12}), the periodic $G$-orbit $Y_{D}$ must intersect $B_{G}(x^{\ast},2\hyperlink{2024.3.C11}{C_{11}} D^{-\hyperlink{2024.12.k8}{\kappa_{8}}})$ at least twice. Thus, there exist $h_{1},h_{2}\in G$  such that $h_{1}x^{\ast}$ and $h_{2}x^{\ast}$ are in different connected components of the intersection of $Y_{D}$ and $B_{G}(x^{\ast},2\hyperlink{2024.3.C11}{C_{11}} D^{-\hyperlink{2024.12.k8}{\kappa_{8}}})$, and that
 \[d_{G}(h_{1}x^{\ast},h_{2}x^{\ast})\leq 4\hyperlink{2024.3.C11}{C_{11}} D^{-\hyperlink{2024.12.k8}{\kappa_{8}}}.\]
 
 Going back to $\mathcal{H}(2)$, we see that there exist  $h_{1}x,h_{2}x\in G.x_{D}$ such that in period coordinates,
 \[\|x-h_{i}x\|_{x} \ll_{D_{2}} 2\hyperlink{2024.3.C11}{C_{11}} D^{-\hyperlink{2024.12.k8}{\kappa_{8}}},\ \ \ \|h_{1}x-h_{2}x\|_{x} \ll_{D_{2}} 4\hyperlink{2024.3.C11}{C_{11}} D^{-\hyperlink{2024.12.k8}{\kappa_{8}}}.\]
 By (\ref{closing2024.3.10}), there exist $C=C(D_{3})>0$ and $x_{1}^{\prime},x_{2}^{\prime}\in B(x,C D^{-\hyperlink{2024.12.k8}{\kappa_{8}}})\cap \Omega W_{D}$ with the same period coordinates as $h_{1}x,h_{2}x$, such that they are in different connected components of the intersection of $\Omega W_{D}$ and $B(x,C D^{-\hyperlink{2024.12.k8}{\kappa_{8}}})\subset B(x,D_{2}^{- 1})$. Therefore, by Corollary \ref{closing2024.3.15}, we conclude that $\disc(\Omega W_{D})\gg  D^{\hyperlink{2024.12.k8}{\kappa_{8}}}$. By a similar argument, one may  obtain the same inequalities    for odd $D$.  
 \end{proof}

\begin{cor}\label{closing2024.12.5} Let   \hypertarget {2024.12.C13} $T>0$, $y\in\mathcal{H}_{1}(2)$. Suppose that  $y$ generates \hypertarget {2024.12.k10} a Teichm\"{u}ller curve $\Omega W_{D}(y)$, and $\{g^{\ast}_{0},\ldots,g^{\ast}_{l}\}\subset\SL(y)\cap B_{G}(T) $
 generates the Veech group $\SL(y)$.  Then there exist $\kappa_{10}>0$ and  $C_{13}>1$ such that 
the discriminant 
\[D\leq  \hyperlink{2024.12.C13}{C_{13}} T^{\hyperlink{2024.12.k10}{\kappa_{10}}} .\]   
\end{cor}
\begin{proof}
If $\{g^{\ast}_{0},\ldots,g^{\ast}_{l}\}\subset\SL(y)\cap B_{G}(T) $
 generates the Veech group $\SL(y)$, then the corresponding elements in the mapping class group generate a subgroup $\Sp(p(\bar{\iota}(y)))$  so that  
the discriminant 
\[\disc(\bar{\iota}(y))\leq  C T^{ \hyperlink{2024.3.k8}{\kappa_{8}}}  \]
for some $\kappa_{8}>0$ and  $C>1$.

 Then by Proposition \ref{closing2024.3.22}, we conclude that the discriminant $D$ of the Teichm\"{u}ller curve satisfies  $D\leq ( \hyperlink{2024.3.C12}{C_{12}}^{-1}C)^{1/\hyperlink{2024.12.k9}{\kappa_{9}}} T^{ \hyperlink{2024.3.k8}{\kappa_{8}}/\hyperlink{2024.12.k9}{\kappa_{9}}}$. The consequence follows from letting $\hyperlink{2024.12.C13}{C_{13}} =( \hyperlink{2024.3.C12}{C_{12}}^{-1}C)^{1/\hyperlink{2024.12.k9}{\kappa_{9}}} $ and $\hyperlink{2024.12.k10}{\kappa_{10}}=\hyperlink{2024.3.k8}{\kappa_{8}}/\hyperlink{2024.12.k9}{\kappa_{9}}$.
\end{proof}

\section{Discreteness of Teichm\"{u}ller curves}\label{closing2024.3.5}
\subsection{Heights}
The theory of heights of Lie groups have been used extensively. In particular, it measures the arithmetic complexity of Lie groups. See e.g. \cite{einsiedler2009distribution,einsiedler2009effective,einsiedler2020effective} for more details. 

Let $W= \mathfrak{sl}_{2g+n-1}(\mathbb{R})$ be the Lie algebra of $\SL_{2g+n-1}(\mathbb{R})$, $B$ the Killing form of $W$. Then $W$ is a normed vector space defined over $\mathbb{Q}$ and has an integral lattice $W_{\mathbb{Z}}=\mathfrak{sl}_{2g+n-1}(\mathbb{Z})$.
\begin{defn}[Height] Given a subspace $V\subset W$, the \textit{height}\index{height} of $V$ is given by the norm
\[\htt(V)\coloneqq\|e_{1}\wedge\cdots\wedge e_{r}\|_{\wedge^{r} W}\]
where $e_{1},\ldots, e_{r}$ is a basis for $V\cap W_{\mathbb{Z}}$, and $\|\cdot\|$ is derived from the Killing form $B$.
\end{defn}

  Let $H\subset \SL_{2g+|\Sigma|-1}(\mathbb{R})$ be defined over $\mathbb{Q}$,   $\mathfrak{h}=\Lie(H)\subset W$   the Lie algebra of $H$, and $r=\dim\mathfrak{h}$.  
   Let $V\coloneqq\wedge^{r} W$, $V_{\mathbb{Z}}\coloneqq\wedge^{r} W_{\mathbb{Z}}$. Then there exists a natural $\SL_{2g+|\Sigma|-1}$-action on $V$ given  by
\begin{equation}\label{closing2024.1.1}
 (v_{1}\wedge \cdots\wedge v_{r}).\gamma=v_{1}\gamma\wedge \cdots\wedge v_{r}\gamma
\end{equation} 
for $v_{1},v_{2}\in H^{1}(M,\Sigma;\mathbb{R})$, $\gamma\in\SL_{2g+|\Sigma|-1}(\mathbb{R})$. It therefore leads to a $\Gamma$-action on $V$. Besides, $V_{\mathbb{Z}}$ is $\Gamma$-stable.
  
    Suppose that $\htt(\mathfrak{h})\leq D$. Then one may choose a basis $e_{1},\ldots,e_{r}$ for $\mathfrak{h}_{\mathbb{Z}}$ such that for all $1\leq i\leq r$,
    \begin{equation}\label{closing2024.2.7}
      \|e_{i}\|\leq D.
    \end{equation} 
    In fact, (\ref{closing2024.2.7}) follows from the lattice reduction theory, together with the fact that the lengths of elements of $W_{\mathbb{Z}}$ are bounded below.

 Now we set 
\[v_{\mathfrak{h}}\coloneqq\frac{e_{1}\wedge \cdots\wedge   e_{r}}{\|e_{1}\wedge \cdots\wedge e_{r}\|}\in V.\]
Suppose that two different $\mathfrak{h}_{1}$, $\mathfrak{h}_{2}$ with $\htt(\mathfrak{h}_{1}),\htt(\mathfrak{h}_{2})\leq D$. Then one can show that 
\begin{equation}\label{closing2024.2.8}
  \|v_{\mathfrak{h}_{1}}-v_{\mathfrak{h}_{2}}\|\gg  D^{-1}.
\end{equation}

\subsection{Grassmannian}
Let $p:H^{1}(M,\Sigma;\mathbb{R})\rightarrow H^{1}(M;\mathbb{R})$ be the forgetful map from relative to absolute cohomology.
Let $\mathcal{N}\subset\mathcal{H}(\alpha)$ be an affine invariant submanifold. Let $T\mathcal{N}\subset H^{1}(M,\Sigma;\mathbb{R})$ be the sublocal system of $H^{1}_{\rel}$ given by the period coordinates on $\mathcal{N}$. We have the short exact sequences
\[0\rightarrow (\ker p)\cap T\mathcal{N}\rightarrow T\mathcal{N}\rightarrow p(T\mathcal{N})\rightarrow0.\]
Let $n_{0}\in\mathcal{N}$ and let $\Gr^{\circ}(2,T\mathcal{N}_{n_{0}})\subset \wedge^{2} T\mathcal{N}_{n_{0}}/\mathbb{R}^{\times}$ be the Grassmannian  of real $2$-planes in $T\mathcal{N}_{n_{0}}$ whose projection to $H^{1}$ is a symplectically non-degenerate $2$-plane. Note that $\Gr^{\circ}(2,T\mathcal{N}_{n_{0}})$  is an open subset of the full Grassmannian of $2$-planes.

The reason we consider the Grassmannian $\Gr^{\circ}(2,T\mathcal{N}_{n_{0}})$ is that the set of $\GL_{2}^{+}(\mathbb{R})$ orbits near $n_{0}\in\mathcal{N}$ is locally modeled on $\Gr^{\circ}(2,T\mathcal{N}_{n_{0}})$.
Let $U$ a simply connected neighborhood of  $n_{0}$.
Let $\bar{\iota}:U\rightarrow\Gr^{\circ}(2,T\mathcal{N}_{n_{0}})$ be the map given by
\begin{equation}\label{closing2024.3.16}
  \bar{\iota}:(X,\omega)\mapsto\Span(\Ree(\omega),\Imm(\omega)).
\end{equation} 
For $T\in\Gr^{\circ}(2,T\mathcal{N}_{n_{0}})$, the fibers  $\bar{\iota}^{-1}(T)$ are connected components of the intersection of $\GL_{2}^{+}(\mathbb{R})$ orbits with $U$. This helps explain the motivation. 

Let $G_{T\mathcal{N}}\subset\GL(T\mathcal{N})$   be the subgroup which acts as the identity on $(\ker p)\cap T\mathcal{N}$ and by symplectic transformations on $p(T\mathcal{N})$, i.e.
\[G_{T\mathcal{N}}=\Sp(p(T\mathcal{N}))\ltimes\Hom(p(T\mathcal{N}),(\ker p)\cap T\mathcal{N}).\] 
In \cite{eskin2018algebraic}, Eskin, Filip and Wright showed that  $G_{T\mathcal{N}}$ is the algebraic hull of the Kontsevich-Zorich cocycle of $T\mathcal{N}$. 
The group $G_{T\mathcal{N}}$ acts transitively on $\Gr^{\circ}(2,T\mathcal{N}_{n_{0}})$, hence 
\begin{equation}\label{algebraic.hull2024.1.6}
  \Gr^{\circ}(2,T\mathcal{N}_{n_{0}})=\stab_{T}\backslash  G_{T\mathcal{N}}
\end{equation}
is a homogenous space, where $T\in\Gr^{\circ}(2,T\mathcal{N}_{n_{0}})$ and $\stab_{T}$ satisfies the short exact sequence
 \[0\rightarrow U\rightarrow\stab_{T}\rightarrow \Sp(p(T))\times\Sp(p(T)^{\perp})\rightarrow0\]
 for  the symplectic-orthogonal decomposition $(p(T\mathcal{N}))_{n_{0}}=p(T)\oplus p(T)^{\perp}$ and some unipotent subgroup $U$. 
 
 For arithmetic applications, we introduce the notion of discriminant of a plane in $T\mathcal{N}$.
 \begin{defn}[Discriminant]\label{closing2024.3.6} Let $T\in  \Gr^{\circ}(2,T\mathcal{N}_{n_{0}})$ be a plane of $T\mathcal{N}_{n_{0}}$. Then the \textit{discriminant}\index{discriminant} $\disc(T)$ of $T$ is defined by
 \[\disc(T)\coloneqq\htt(\Lie( \Sp(p(T)))).\]
 Similarly, we define the   \textit{discriminant} of a Teichm\"{u}ller curve $\mathcal{O}$ as the discriminant of the corresponding plane. That is, if $\mathcal{O}=G.x$ for some $x\in\mathcal{H}(\alpha)$, then 
 \begin{equation}\label{closing2024.3.1}
   \disc(\mathcal{O})\coloneqq\disc(\bar{\iota}(x)).
 \end{equation} 
 \end{defn}
 \begin{rem}
    Note that the mapping class group $\Gamma$ preserves $W_{\mathbb{Z}}$. Thus, the definition (\ref{closing2024.3.1}) is well-defined, i.e. independent of the choice of $x$.
 \end{rem}

Let $d_{G_{T\mathcal{N}}}(\cdot,\cdot)$ be a left-invariant metric on $G_{T\mathcal{N}}$. It induces a metric $\bar{d}_{G_{T\mathcal{N}}}(\cdot,\cdot)$ on $\stab_{T}\backslash  G_{T\mathcal{N}}$ and so on $\Gr^{\circ}(2,T\mathcal{N}_{n_{0}})$.
\begin{lem}\label{closing2024.3.2} \hypertarget{2024.3.C8}
   Let $D>0$, and $\Omega\subset G_{T\mathcal{N}}$  a fixed compact subset of $G_{T\mathcal{N}}$. Let $H\subset G_{T\mathcal{N}}$ is a subgroup with $\Lie(H)=\mathfrak{h}$. Suppose $g_{1},g_{2}\in\Omega$ are so that \[\htt(\Ad(g_{1})\mathfrak{h})\leq D,\ \ \   \htt(\Ad(g_{2})\mathfrak{h})\leq D,\ \ \ \Ad(g_{1})\mathfrak{h}\neq \Ad(g_{2})\mathfrak{h}.\]
    Then $d_{G_{T\mathcal{N}}}(g_{1},g_{2})\geq \hyperlink{2024.3.C8}{C_{8}(\Omega)}  D^{-1}$ for some $C_{8}(\Omega)>0$.
\end{lem}
\begin{proof}
  Let $\mathfrak{h}$, $g_{1}$, $g_{2}$ be as stated. Then by (\ref{closing2024.2.8}), we have
\[\|v_{\Ad(g_{1})\mathfrak{h}}-v_{\Ad(g_{2})\mathfrak{h}}\|\geq  D^{-1}.\]
Note that the map $g\mapsto \Ad(g)\mathfrak{h}$ is a smooth map from $\Omega$ to $V=\wedge^{r} W$. Thus, it cannot increase distances by more than a constant factor depending on $\Omega$. In particular, we get that 
\[d_{G_{T\mathcal{N}}}(g_{1},g_{2})\gg_{\Omega}D^{-1}\]
as claimed.
\end{proof}
\begin{prop}[Discreteness of Teichm\"{u}ller curves]\label{closing2024.3.3} \hypertarget{2024.3.C9} Suppose $x_{1},x_{2}\in\mathcal{N}$ generate   Teichm\"{u}ller curves $\mathcal{O}_{1}=\GL_{2}^{+}(\mathbb{R}).x_{1}$, $\mathcal{O}_{2}=\GL_{2}^{+}(\mathbb{R}).x_{2}$ of discriminants 
\[\disc(\mathcal{O}_{1}), \disc(\mathcal{O}_{2})\leq D.\] Let $U$ be a simply connected neighborhood of $\mathcal{N}$.  Suppose  $x_{1},x_{2}\in U$ are lying in the different connected components of the  intersection of $\GL_{2}^{+}(\mathbb{R})$-orbits  with $U$.  Then 
\begin{equation}\label{closing2024.3.4}
  d(x_{1},x_{2})\geq  \hyperlink{2024.3.C9}{C_{9}(U)} D^{-1}
\end{equation}
for some $C_{9}(U)>0$.
\end{prop}
\begin{rem}
  When $\mathcal{O}_{1}=\mathcal{O}_{2}$, we see that a Teichm\"{u}ller curve of low discriminant implies that the distance of any two connected components  of $\GL_{2}^{+}(\mathbb{R})$-orbits  in a neighborhood has a lower bound. In particular, we conclude that there are only finitely many connected components of  a given Teichm\"{u}ller curve in a neighborhood.
\end{rem}
\begin{proof}[Proof of Proposition \ref{closing2024.3.3}]
Let $T_{1}=\bar{\iota}(x_{1}),T_{2}=\bar{\iota}(x_{2})\in\Gr^{\circ}(2,T\mathcal{N}_{n_{0}})$. Then there exists $g\in G_{T\mathcal{N}}$ such that $T_{2}=T_{1}g$ and
\begin{equation}\label{closing2024.2.9}
  \bar{d}_{G_{T\mathcal{N}}}(T_{1},T_{2})=d_{G_{T\mathcal{N}}}(e,g).
\end{equation} 
Let $\mathfrak{h}=\Lie(\Sp(p(T_{1})))$. Then the assumption indicates that 
\[\htt(\mathfrak{h})\leq D,\ \ \   \htt(\Ad(g)\mathfrak{h})\leq D,\ \ \ \mathfrak{h}\neq \Ad(g)\mathfrak{h}.\]
Then by Lemma \ref{closing2024.3.2}, we get that 
\[d_{G_{T\mathcal{N}}}(e,g)\geq \hyperlink{2024.3.C8}{C_{8}(U)}  D^{-1}.\]
 Thus, by  (\ref{closing2024.2.9}), we conclude that $d(x_{1},x_{2})\gg_{U}D^{-1}$.
\end{proof}

 Enlighten by Theorem \ref{effective2023.11.2}, when we focus on a compact subset of $\mathcal{H}_{1}(\alpha)$, we are able to make the implicit constant in (\ref{closing2024.3.4}) absolute.
 \begin{cor}\label{closing2024.3.15} \hypertarget{2024.3.C10}
    Let $K_{\alpha}^{\prime}\subset\mathcal{H}_{1}(\alpha)$ be as in Theorem \ref{effective2023.11.2}. Then there exists a constant $C>0$ such that the following holds. For any $x\in K_{\alpha}^{\prime}$, any Teichm\"{u}ller  curves  $\mathcal{O}_{1},\mathcal{O}_{2}$ of discriminants $\disc(\mathcal{O}_{1}), \disc(\mathcal{O}_{2})\leq D$, if $E_{1}\subset B(x,C)\cap\mathcal{O}_{1}$, $E_{2}\subset B(x,C)\cap\mathcal{O}_{2}$ are two distinct  connected components of $\SL_{2}(\mathbb{R})$-orbits, then  
    \[d(E_{1},E_{2})\geq  \hyperlink{2024.3.C10}{C_{10}}  D^{-1}\]
    for some $C_{10}>0$.
 \end{cor}

     \section{Effective closing}\label{closing2024.3.29}

      \begin{thm}[Effective closing lemma] \label{closing2024.12.2} \hypertarget{2024.12.k11} Let   $\alpha=(2g-2)$.
 There exists    $\kappa_{11}>0$ such that for  $N\geq \hyperlink{2024.12.k11}{\kappa_{11}}$, $x\in\mathcal{H}_{1}(\alpha)$, there exists  $T_{0}=T_{0}(\ell(x))>0$, with the following property.
         Let $T\geq T_{0}$. Suppose that $\{g_{1},\ldots,g_{l}\}\subset B_{G}(T)$ is $1$-separated,   and that  
          \[d(g_{i}x,g_{j}x)< T^{-N}.\]  
       for any $1\leq i,j\leq l$.   Then there is   a point $y\in B(x,T^{\hyperlink{2024.12.k11}{\kappa_{11}}-N})$, and  a  collection of  $\frac{1}{2}$-separated points 
 \begin{equation}\label{closing2024.12.1}
   \{g^{\ast}_{i}:1\leq i\leq l\}\subset B_{G}(2T),\ \ \ d_{G}(g^{\ast}_{i},g_{i})\leq T^{\hyperlink{2024.12.k11}{\kappa_{11}}-N}
 \end{equation} 
 such that 
 \begin{equation}\label{closing2024.12.4}
  g^{\ast}_{i}y=g^{\ast}_{j}y
 \end{equation} 
  for any $1\leq i,j\leq l$.
\end{thm}

        Let $\varepsilon>0$, $N,T>1$ and let $x\in\mathcal{H}_{1}(\alpha)$.   Suppose that $\{g_{1},\ldots,g_{l}\}\subset B_{G}(T)$ is $1$-separated,  and that  
          \[d(g_{i}x,g_{j}x)< T^{-N}.\] 
Write $\tilde{x}\in\mathcal{TH}(\alpha)$ with $\pi(\tilde{x})=x$. Then by the assumption, there exists $\gamma_{ij}\in\Gamma$ so that 
\begin{equation}\label{closing2023.12.6}
 d( g_{i}\tilde{x}\gamma_{ij} ,g_{j}\tilde{x})\leq T^{-N} .
\end{equation} 
i.e.
          \[\Gamma(g_{j}^{-1}g_{i},\tilde{x})=\Gamma(\gamma_{ij}).\]
Now recall that  
\[ \log\|\Gamma(g,x)\|\leq C_{14}\log\|g\|+C_{15}\]
for some $C_{14}, C_{15}>0$.  This was first proved by Forni in \cite{forni2002deviation}, see also \cite[Lemma 2.3]{forni2014lyapunov} and \cite[Corollary 30]{forni2014introduction}. Then
 \begin{equation}\label{closing2024.05.1}
   \log\|\Gamma(g_{j}^{-1}g_{i},\tilde{x})\|\leq C_{14}\log\|g_{j}^{-1}g_{i}\|+C_{15}\ll_{N,x} \log T.
 \end{equation}

On the other hand, note  that   
\begin{multline}\label{closing2024.1.2}
 d(g_{i}\tilde{x}\gamma_{ij} \gamma_{jk},g_{i}\tilde{x}\gamma_{ik}) \\ 
 \leq  d(g_{i}\tilde{x}\gamma_{ij} \gamma_{jk},g_{i}\tilde{x}\gamma_{ik})+d(g_{j}\tilde{x}\gamma_{jk},g_{k}\tilde{x})+d(g_{k}\tilde{x},g_{i}\tilde{x}\gamma_{ik})\leq 3T^{-N}.
\end{multline}
In addition, note that 
 by Lemma \ref{effective2023.11.4}, if $T\geq T_{0}^{\prime}$ and $N$ are large enough so that $T^{N} \geq\hyperlink{2023.11.C2}{C_{2}} ^{-1} \ell(x)^{ - \hyperlink{2024.12.k5}{\kappa_{5}}}T^{\hyperlink{2024.12.k5}{\kappa_{5}} }\geq\hyperlink{2023.11.C2}{C_{2}} ^{-1} \ell(g_{k}x)^{ -\hyperlink{2024.12.k5}{\kappa_{5}}}$, the points of (\ref{closing2024.1.2}) are in the same fundamental domain. Then we get
\begin{equation}\label{closing2023.12.4}
 \gamma_{ij} \gamma_{jk}=\gamma_{ik}.
\end{equation}

Next, we study the effective closing via the $\SL_{2}(\mathbb{R})$-invariant subbundles. 
Let $x\in \mathcal{H}(\alpha)$, and $E$ be a subbundle of $H^{1}$ over $\mathcal{H}(\alpha)$ with dimension $d$. As in Section \ref{closing2024.3.5}, let 
\[V(x)\coloneqq\wedge^{d} H^{1}(x)\]
 be the
$d$-exterior product of $H^{1}(x)$. In particular, $E$ can be considered as an element of $V$. Let $n=\dim V$. Let $v_{1}^{E},\ldots,v_{d}^{E}\in E$ generate $E$ so that  $\|\iota\|=1$ where
\[\iota=v_{1}^{E}\wedge\cdots\wedge v_{d}^{E}.\]
 
\begin{prop}\label{closing2024.11.2}
   Let $E$ be an $\SL_{2}(\mathbb{R})$-invariant subbundle of $H^{1}$ over $\mathcal{H}(\alpha)$ with dimension $d$ and $\alpha=(2g-2)$. Suppose that 
   \begin{equation}\label{closing2023.12.6}
d(\iota(g_{j}x),\iota(g_{i}x)) \leq T^{-N}.
\end{equation}  
Then there exists $y\in\mathcal{H}(\alpha)$ with $d(y, x)<T^{-1}$ such that 
\[\iota(g_{j} y) =\iota(g_{i}y).\]
and so 
\[E(g_{i}y)=E(g_{j} y).\]
\end{prop}

We consider the bundle $E$ (and so $\iota$) are over $\mathcal{TH}(\alpha)$.  Then by (\ref{closing2023.12.6}), under local affine structure, we have
\begin{equation}\label{closing2023.12.7}
 d(\iota(\tilde{x})\Gamma(g_{j}^{-1}g_{i},\tilde{x}),\iota(\tilde{x})) \leq T^{-N}.
\end{equation}

For a subset $F\subset\{\gamma_{ij}\}_{i,j}$, let 
\[W(F)\coloneqq\{v\in \mathbb{R}^{n}:v \Gamma(\gamma)= v,\ \gamma\in F\}.\]
We shall show that (\ref{closing2023.12.7}) reduces to a system of finitely many equations.
\begin{lem}[Effective Noetherian]\label{closing2023.12.8}\hypertarget{2024.3.k12}  There exists an integer $\kappa_{12}>0$ depending only on the genus $g$ so that there exists a finite subset $F\subset\{\gamma_{ij}\}_{i,j}$, $|F|\leq\hyperlink{2024.3.k12}{\kappa_{12}}$, and
\[W(F)=W(\{\gamma_{ij}\}_{i,j}).\]
\end{lem}
\begin{proof} For any $F\subset\{\gamma_{ij}\}_{i,j}$ and $\gamma\in  \{\gamma_{ij}\}_{i,j}\setminus F$, we clearly have 
 \[\dim W(F)-1\leq \dim W(F\cup\{\gamma\}).\]
An induction then shows that there exists a finite set $F\subset\{\gamma_{ij}\}_{i,j}$ with $|F|\leq \dim V$ so that $W(F)=W(\{\gamma_{ij}\}_{i,j})$.
\end{proof}

Now let $F=\{\gamma_{1},\ldots,\gamma_{m}\}\subset\{\gamma_{ij}\}_{i,j}$  be the finite set given by Lemma \ref{closing2023.12.8}. Define $\Phi:V\rightarrow V^{m}$ by
\[\Phi:v\mapsto \bigoplus_{k=1}^{m}v(\Gamma(\gamma_{k})-I).\]
Then  (\ref{closing2024.05.1}) indicates that  
\[\|\Phi(\iota(\tilde{x}))\|\leq  C_{16}  T^{-N}\]
for some $C_{16}>0$.
Moreover, since $\Gamma(\gamma_{ij})$ is an integral matrix, by (\ref{closing2024.05.1}), all the entries of $\Phi:\mathbb{R}^{n}\rightarrow (\mathbb{R}^{n})^{m}$ are integers of size $\ll_{N,x} T$. Then there exists a  vector $\tilde{v}\in \ker(\Phi)$ near $\iota(\tilde{x})$.
\begin{lem}\hypertarget{2024.12.k13}
   There exists $\kappa_{13}=\kappa_{13}(g)>0$, $C_{17}=C_{17}(N,x)>0$, and  $\tilde{v}\in \mathbb{R}^{n}$ such that
    \begin{equation}\label{closing2023.12.9}
  \Phi(\tilde{v})=0,\ \ \   \|\tilde{v}-\iota(\tilde{x})\|\leq C_{17} T^{\hyperlink{2024.12.k13}{\kappa_{13}}-N}.
    \end{equation}  
\end{lem}
\begin{proof} This is purely the linear algebra. Note that $\Phi$ is an integer matrix with bounded coefficients. If for some $v$, $\Phi(v)$ is sufficiently close to $0$, then $\ker(\Phi)$   have a large contribution.
  See \cite[Lemma 13.1]{einsiedler2009effective} for the  detailed proof.
\end{proof}
Furthermore, we shall show that there exists $\tilde{y}\in \mathcal{TH}(\alpha)$ near $\tilde{x}$, such that
\[ \Phi(\iota(\tilde{y}))=0.\]
Roughly speaking, this can be done by projecting $\tilde{v}$ down to $\iota(\tilde{x})+H_{\mathbb{C}}^{1}(\tilde{x})$.

 First,  define $\iota_{\tilde{x}}:H_{\mathbb{C}}^{1}(\tilde{x})\rightarrow  \mathbb{R}^{d}\cong V(\tilde{x})$ by
    \[\iota_{\tilde{x}}:a+ib\mapsto \iota(\tilde{x}+a+ib).\]
    Consider its differential  on tangent spaces 
    \[ D _{0}\iota_{\tilde{x}}:T_{0}(H_{\mathbb{C}}^{1}(\tilde{x}))\rightarrow T_{\iota(\tilde{x})}(\mathbb{R}^{n})\cong  V(\tilde{x}).\] 
    
    Then by Theorem \ref{closing2024.11.1}, we know that any $\SL_{2}(\mathbb{R})$-invariant subbundle has an invariant complement:
    \[V=\iota(H_{\mathbb{C}}^{1}) \oplus W.\]
    Consider the map $H_{\mathbb{C}}^{1}(\tilde{x})\oplus W\rightarrow V$
    defined by
    \[(c,w)\mapsto \iota(\tilde{x}+c)+w.\]
    Its differential  is subjective at $(0,0)$ so that we obtain a map from $T_{(0,0)}(H_{\mathbb{C}}^{1}(\tilde{x}))\oplus 0$ onto $W^{\perp}$.
    Thus, it defines a smooth map from a neighborhood $\mathcal{U}_{1}\subset   H_{\mathbb{C}}^{1}(\tilde{x})\oplus W=V$ at $(0,0)$, to a neighborhood $\mathcal{U}_{2}$ at $\iota(\tilde{x})$. Now consider the projection $\Pi:\mathcal{U}_{1}\rightarrow \mathcal{U}_{2}$ given by
    \[\Pi:\iota+w\mapsto \iota.\]
    \begin{lem} For any $\gamma_{ij}$, we have 
    \begin{equation}\label{closing2023.12.10}
      \Pi(\tilde{v})\gamma_{ij}=\Pi(\tilde{v}).
    \end{equation} 
    \end{lem}
    \begin{proof}
       Choose $(c,w)\in\mathcal{U}_{1}$ so that 
       \[\iota_{\tilde{x}}(c)+w=\tilde{v}.\]
      Then, note that 
      \[\|\iota(\tilde{x})-\iota_{\tilde{x}}(c)\|,\|w\|,\|w\gamma_{ij}\|\ll\|\Gamma(\gamma_{ij})\|\|\iota(\tilde{x})-\tilde{v}\|.\]
       Moreover, since $\tilde{v}\gamma_{ij}-\tilde{v}=(\iota_{\tilde{x}}(c)\gamma_{ij}-\iota(\tilde{x}))+(\iota(\tilde{x})-\iota_{\tilde{x}}(c))+(w\gamma_{ij}-w)$, we have 
       \begin{align}
\|\Pi(\tilde{v})\gamma_{ij}-\iota(\tilde{x})\|=& \|\iota_{\tilde{x}}(c)\gamma_{ij}-\iota(\tilde{x})\|\;\nonumber\\
\leq & \|\iota(\tilde{x})-\iota_{\tilde{x}}(c)\|+\|w\gamma_{ij}-w\|  \;\nonumber\\
\ll & \|\Gamma(\gamma_{ij})\|\|\iota(\tilde{x})-\tilde{v}\|.\;  \nonumber
\end{align}

Next, we observe
 \[\iota_{\tilde{x}}(c)+w=\tilde{v}=\tilde{v}\gamma_{ij}=\iota_{\tilde{x}}(c)\gamma_{ij}+w\gamma_{ij}=\iota_{\tilde{x}}((\tilde{x}+c)\gamma_{ij}-\tilde{x})+w\gamma_{ij}.\] 
By (\ref{closing2024.05.1}) and (\ref{closing2023.12.7}), for sufficiently large $T\geq T_{0}^{\prime\prime}$ and $N$, we get
\[\iota_{\tilde{x}}(c)=\iota_{\tilde{x}}((\tilde{x}+c)\gamma_{ij}-\tilde{x}),\ \ \ w=w\gamma_{ij}\]
and so
 \[  \Pi(\tilde{v})\gamma_{ij}=\iota_{\tilde{x}}(c)\gamma_{ij}=\iota_{\tilde{x}}(c)=  \Pi(\tilde{v}).\]
We obtain (\ref{closing2023.12.10}).
    \end{proof}
    \begin{proof}[Proof of Proposition \ref{closing2024.11.2}]
        Therefore, there exists $c\in H^{1}(M,\Sigma;\mathbb{C})$ close to $0$ so that  $\iota_{\tilde{x}}(c)\gamma_{ij}=\iota_{\tilde{x}}(c)$. 
     In other words, letting $\tilde{y}\coloneqq\tilde{x}+c$,  we get
     \begin{equation}\label{closing2023.12.12}
     \iota(g_{i}\tilde{y}\gamma_{ij})=      \iota(\tilde{y}\gamma_{ij})= \iota(\tilde{y}\gamma_{ij})=\iota(\tilde{y})=  \iota(g_{j}\tilde{y})
     \end{equation} 
   for any $\gamma_{ij}$.
    \end{proof}

  Now let $T(w)=\Ree w\oplus \Imm w$ be the tautological bundle, and 
  \[\iota=v_{1}^{T}\wedge v_{2}^{T}.\]
   Then  It is a $\SL_{2}(\mathbb{R})$-invariant subbundle. Also, by (\ref{closing2024.1.2}),  we have
\[d(\iota(g_{j}\tilde{x}),\iota(g_{i}\tilde{x})\gamma_{ij}) \leq T^{-N}.\]  
Then Proposition \ref{closing2024.11.2} implies that  there exists $\tilde{y}\in\mathcal{TH}(\alpha)$ with $d(\tilde{y}, \tilde{x})<T^{-1}$ such that 
\[\iota(g_{j} \tilde{y}) =\iota(g_{i}\tilde{y})\gamma_{ij}.\]
 Moreover, we have: 
  \begin{prop}\label{closing2024.3.40}  For sufficiently large $T\geq T_{0}^{\prime\prime\prime}$ and $N$, there is $\tilde{y}\in B(\tilde{x},T^{\hyperlink{2024.12.k13}{\kappa_{13}}-N})$ and a collection of $\frac{1}{2}$-separated points 
   \[\{h_{i}:1\leq i\leq l\}\subset B_{G}(2T),\ \ \ d_{G}(h_{i},g_{i})<T^{\hyperlink{2024.12.k13}{\kappa_{13}}+2-N}\]
   so that 
      \[h_{i}\tilde{y}=h_{1}\tilde{y}\gamma_{1i}\ \ \ \text{ or }\ \ \ h_{i}\pi_{1}(\tilde{y})=h_{1}\pi_{1}(\tilde{y})\]
      where $\pi_{1}:\mathcal{TH}(\alpha)\rightarrow \mathcal{H}_{1}(\alpha)$ is the natural projection.
  \end{prop} 
 \begin{proof}
    Let $\phi:\mathcal{TH}(\alpha)\rightarrow H^{1}(M,\Sigma;\mathbb{C})$ be the period map. Let 
    \begin{equation}\label{closing2023.12.13}
      \phi(g_{j}\tilde{y})=\begin{bmatrix}
a_{1} & \cdots & a_{2g+|\Sigma|-1}\\
b_{1} & \cdots & b_{2g+|\Sigma|-1}
\end{bmatrix},\ \ \ \phi( g_{i}\tilde{y} \gamma_{ij})=\begin{bmatrix}
a_{1}^{\prime} & \cdots & a_{2g+|\Sigma|-1}^{\prime}\\
b_{1}^{\prime} & \cdots & b_{2g+|\Sigma|-1}^{\prime}
\end{bmatrix},
    \end{equation} 
and let 
\[  h^{ij}_{1}=\begin{bmatrix}
a_{1} & a_{2}\\
b_{1} & b_{2}
\end{bmatrix}\begin{bmatrix}
a_{1}^{\prime} & a_{2}^{\prime}\\
b_{1}^{\prime} & b_{2}^{\prime}
\end{bmatrix}^{-1}.\] 
Then  for sufficiently large $T\geq T_{0}^{\prime\prime\prime}$ and $N$, by Corollary \ref{closing2024.1.4}, $h^{ij}_{1}\in\GL^{+}_{2}(\mathbb{R})$. One writes
  \begin{equation}\label{closing2023.12.11}
  \phi(g_{j}\tilde{y})=\left[
            \begin{array}{ccccccc}   \multicolumn{2}{c}{\multirow{2}*{$h^{ij}_{1}\begin{bmatrix}
a_{1}^{\prime} & a_{2}^{\prime}\\
b_{1}^{\prime} & b_{2}^{\prime}
\end{bmatrix}$,}} & a_{3} & a_{4} & \cdots & a_{2g+|\Sigma|-1}\\
                    & & b_{3}& b_{4} & \cdots & b_{2g+|\Sigma|-1}\\  
            \end{array}
          \right].
  \end{equation}  
  Recall from (\ref{closing2023.12.6}) that 
  \begin{equation}\label{closing2023.12.14}
  d(g_{j}\tilde{y},g_{i}\tilde{y}\gamma_{ij})\leq  d(g_{j}\tilde{y},g_{j}\tilde{x})+  d(g_{j}\tilde{x},g_{i}\tilde{x}  \gamma_{ij} )+ d(g_{i}\tilde{x}  \gamma_{ij} ,g_{i}\tilde{y}\gamma_{ij})\ll  T^{\hyperlink{2024.12.k13}{\kappa_{13}}+1-N}.
  \end{equation} 
As before,
 by Lemma \ref{effective2023.11.4}, if $T$ and $N$ are large enough, the points of (\ref{closing2023.12.14}) are in the same fundamental domain.  
 In particular, we get that 
  \begin{equation}\label{closing2023.12.16}
    d_{G}(h_{1}^{ij},e)\ll  T^{\hyperlink{2024.12.k13}{\kappa_{13}}+2-N}
  \end{equation} 
  for any $i,j$.
  
On the other hand, by (\ref{closing2023.12.12}), we get 
\[\iota(g_{j}\tilde{y})=\iota(g_{i}\tilde{y}\gamma_{ij})=|\det h^{ij}_{1}|^{-1}\iota(h^{ij}_{1}g_{i}\tilde{y}\gamma_{ij}).\]
Thus, as a $2$-dimensional plane, we have 
\[ \Ree(g_{j}\tilde{y})\oplus \Imm(g_{j}\tilde{y})=\bar{\iota}(g_{j}\tilde{y})=\bar{\iota}(h^{ij}_{1}g_{i}\tilde{y}\gamma_{ij})=\Ree(h^{ij}_{1}g_{i}\tilde{y}\gamma_{ij})\oplus \Imm(h^{ij}_{1}g_{i}\tilde{y}\gamma_{ij})\]
where $\overline{\iota}:(X,\omega)\mapsto\Span(\Ree(\omega),\Imm(\omega))$ is given in (\ref{closing2024.3.16}).
In particular, we have 
\[\Ree(g_{j}\tilde{y})-\Ree(h^{ij}_{1}g_{i}\tilde{y}\gamma_{ij}), \ \Imm(g_{j}\tilde{y})-\Imm(h^{ij}_{1}g_{i}\tilde{y}\gamma_{ij})\in \Ree(g_{j}\tilde{y})\oplus \Imm(g_{j}\tilde{y}).\]
In other words, we can write
\[\Ree(g_{j}\tilde{y})-\Ree(h^{ij}_{1}g_{i}\tilde{y}\gamma_{ij})=r_{1}\Ree(g_{j}\tilde{y})+r_{2}\Imm(g_{j}\tilde{y}).\]
for some $r_{1},r_{2}\in\mathbb{R}$.
Comparing (\ref{closing2023.12.13}) and (\ref{closing2023.12.11}), we see the first two entries of $\Ree(g_{j}\tilde{y})-\Ree(h^{ij}_{1}g_{i}\tilde{y}\gamma_{ij})\in H^{1}(X,\Sigma;\mathbb{R})$ in period coordinates are $0$.
This forces $r_{1}=r_{2}=0$.  Similarly, we have 
   $\Imm(g_{j}\tilde{y})=\Imm(h^{ij}_{1}g_{i}\tilde{y}\gamma_{ij})$.
     Thus, we get
      \begin{equation}\label{closing2023.12.15}
        \phi(g_{j}\tilde{y})=\phi(h^{ij}_{1}g_{i}\tilde{y}\gamma_{ij}).
      \end{equation}   
     Combining   (\ref{closing2023.12.14}) and (\ref{closing2023.12.15}), we conclude that 
     \[g_{j}\tilde{y}=h^{ij}_{1}g_{i}\tilde{y}\gamma_{ij} \ \ \ \text{ or }\ \ \ g_{j} \pi(\tilde{y})=h^{ij}_{1}g_{i}\pi(\tilde{y}). \]
      In particular, we have $\Area(g_{j}\pi(\tilde{y}))=\Area(h^{ij}_{1}g_{i}\pi(\tilde{y}))$ and so $h^{ij}_{1}\in G=\SL_{2}(\mathbb{R})$.
      
    Now writing $h_{i}\coloneqq (h^{1i}_{1})^{-1}g_{i}\in G$, we have  
     \[h_{i}\tilde{y}=h_{1}\tilde{y}\gamma_{1i} \ \ \ \text{ or }\ \ \ h_{i} \pi_{1}(\tilde{y})=h_{1}\pi_{1}(\tilde{y})\]
     for any $i$. By (\ref{closing2023.12.16}), $h_{i}$ and $h_{j}$ are $\frac{1}{2}$-separated for $i\neq j$.  
 \end{proof}  
 
By letting $g^{\ast}_{i}=h_{i}$, and $\hyperlink{2024.12.k11}{\kappa_{11}}=\hyperlink{2024.12.k13}{\kappa_{13}}+2$, we obtain Theorem \ref{closing2024.12.2}.
 
 Finally, we deduce   Theorem \ref{closing2023.12.1}  where a similar homogeneous version has been settled in \cite[Proposition  13.1]{einsiedler2009effective}.
 Theorem \ref{closing2023.12.1}  follows from Theorem \ref{closing2024.12.2} and the following lemma:
 \begin{lem}\label{closing2024.12.3} For $l\geq (\Vol B_{G}(T))^{1-\varepsilon}$, there is  $y\in B(x,T^{\hyperlink{2024.12.k11}{\kappa_{11}}-N})\subset\mathcal{H}_{1}(\alpha)$ such that  the Veech group $\SL(y)\subset G$ is a non-elementary Fuchsian group generated by elements in $B_{G}(2T)$.
 \end{lem}
 \begin{proof}
     Letting $y=\pi_{1}(\tilde{y})\in \mathcal{H}_{1}(\alpha)$, we have a collection of $\frac{1}{2}$-separated group elements $\{h_{1},\ldots,h_{l}\}\subset B_{G}(2T)$  such that 
 \[h_{i} y=h_{1}y\]
 for all $1\leq i\leq l$. In other words, we have 
 \begin{equation}\label{closing2024.3.21}
   \{h_{1}h_{1}^{-1},\ldots,h_{l}h_{1}^{-1}\}\subset \SL(h_{1}y)\cap B_{G}(2T). 
 \end{equation}  
 Then  for sufficiently small $\varepsilon$, we have 
 \[l\geq (\Vol B_{G}(2T))^{1-\varepsilon}\geq \hyperlink{2024.3.C7}{C_{7}} \Vol(B_{G}(2T))^{1-\frac{1}{3}  \hyperlink{2024.12.k7}{\kappa_{7}}}   \hypertarget{2024.3.k8}\]
Thus, by Corollary \ref{closing2024.3.18}, we conclude that 
$\SL(h_{1}y)$, and so $\SL(y)$, are Zariski dense in $G=\SL_{2}(\mathbb{R})$. 
In particular, $\SL(y)$ is a non-elementary Fuchsian group. 
 \end{proof}
\begin{proof}[Proof of Theorem \ref{closing2023.12.1}]
         Let $\varepsilon>0$, $N,T>1$ and let $x\in\mathcal{H}_{1}(2)$.   Suppose that $\{g_{1},\ldots,g_{l}\}\subset B_{G}(T)$ is $1$-separated, that $l\geq (\Vol B_{G}(T))^{1-\varepsilon}$, and that  
          \[d(g_{i}x,g_{j}x)< T^{-N}.\]  
          
          Now by  Theorem \ref{closing2024.12.2}, there is   a point $y\in B(x,T^{\hyperlink{2024.12.k11}{\kappa_{11}}-N})$, and  a  collection of  $\frac{1}{2}$-separated points 
\[  \{g^{\ast}_{i}:1\leq i\leq l\}\subset B_{G}(2T),\ \ \ d_{G}(g^{\ast}_{i},g_{i})\leq T^{\hyperlink{2024.12.k11}{\kappa_{11}}-N}
 \]
 such that 
 \[g^{\ast}_{i}y=g^{\ast}_{j}y\]  
  for any $1\leq i,j\leq l$.
          
          Since $l\geq (\Vol B_{G}(T))^{1-\varepsilon}$, by Lemma \ref{closing2024.12.3},  $\SL(y)$ contains a hyperbolic element. Then by  Theorem \ref{closing2024.3.19}, when we consider $y\in\mathcal{H}(2)$, we see that $y$ generates a Teichm\"{u}ller curve. 

Moreover, recall from (\ref{closing2024.3.20}) that $\SL(y)$  has an alternate definition in terms of affine automorphisms $\Aff(y)$. Each affine automorphism determines a mapping class in $\Gamma=\Mod(M,\Sigma)$. Thus, we can view the Veech group $\SL(y)$ as a subgroup of $\Gamma$. In addition, by (\ref{closing2024.3.21}), one may deduce that  $\SL(h_{1}y)\cap B_{G}(2T)$ is Zariski dense in $\Sp(p(\bar{\iota}(y)))$. Then by Corollary \ref{closing2024.12.5}, the discriminant of the Teichm\"{u}ller curve satisfies
\[D\leq    \hyperlink{2024.12.C13}{C_{13}} (2T)^{\hyperlink{2024.12.k10}{\kappa_{10}}}\leq     T^{2\hyperlink{2024.12.k10}{\kappa_{10}}} .\]   
Let 
     $T\geq T_{0}\geq\max\{ T_{0}^{\prime}, T_{0}^{\prime\prime}, T_{0}^{\prime\prime\prime}\}$ and $\hyperlink{2024.04.N6}{N_{1}}=\max\{2\hyperlink{2024.12.k10}{\kappa_{10}},\hyperlink{2024.12.k11}{\kappa_{11}}\}$ so that the constants $\hyperlink{2024.12.C13}{C_{13}} 2^{\hyperlink{2024.12.k10}{\kappa_{10}}}$ are absorbed. We establish Theorem \ref{closing2023.12.1}.
\end{proof}

\section{Observable transversal increment}\label{margulis2024.12.3}
In this section,   we define
a Margulis function on a long horocycle orbit, measuring the discrete
dimension  of  the additional direction transverse to
the $\SL_{2}(\mathbb{R})$-orbit. Then
we establish an effective closing lemma with respect to the $P$-action on $\mathcal{H}_{1}(2)$ similar to \cite{rached2024separation}. We fix the notation:
\begin{itemize}
  \item  Let $\kappa>0$. For $t>0$, let 
  \begin{equation}\label{closing2024.4.01}
   \mathsf{E}_{t}=\mathsf{E}_{t,[0,1]}(e^{-\kappa t})\coloneqq B_{G}(e^{-\kappa t})\cdot a_{t}\cdot u_{[0,1]}\subset G.
  \end{equation} 
  \item Let $x\in \mathcal{H}_{1}(\alpha)$. For every $z\in \mathsf{E}_{t}.x$, let
  \begin{equation}\label{effective2022.2.36}
     F_{z}(t)\coloneqq\{w\in H^{\perp}_{\mathbb{C}}(z):0<\|w\|_{x}<\ell(z), \ z+w\in \mathsf{E}_{t}.x\}.
  \end{equation}
It collects the appearance of the long horocycle   $\mathsf{E}_{t}.x$ in the transversal direction.  Note  that this is a finite subset of $H^{\perp}_{\mathbb{C}}(x)$.   
  \item Let $\gamma\in(0,1)$. Define the local density function $f_{t,\gamma}:\mathsf{E}_{t}.x\rightarrow[2,\infty)$ by
  \[f_{t,\gamma}(z)=\left\{
    \begin{array}{ll} \displaystyle{\sum_{w\in  F_{z}(t)}\|w\|^{-\gamma}_{x}} &,\text{ if } F_{z}(t)\neq\emptyset\\
   \ell(z)^{-\gamma} &,\text{ otherwise }
                     \end{array}\right. .\]
  The function $f_{t,\gamma}(z)$ record the local density at $z$ of the horocycle along the transversal direction. Informally, $f_{t,\gamma}(z)$ being small means $w\in F_{z}(t)$ are \textbf{not too close} to $0$.  
\end{itemize}

\begin{thm}[Observable transversal increment]\label{effective2024.6.01} There exist  $N_{0}=N_{0}(\alpha)$, \hypertarget{2024.04.N6}   $\varpi=\varpi(\alpha)>0$ satisfying the following:
\begin{enumerate}[\ \ \ ]
  \item  Let $N\geq 2N_{0}$, $x_{0}\in  \mathcal{H}_{1}(\alpha)$. \hypertarget{2024.12.t5} Then there     \hypertarget{2024.12.C18} exists $t_{5}=t_{5}(\ell(x_{0}))>0$ such \hypertarget{2024.12.N63} that for all large enough $t\geq \hyperlink{2024.12.t5}{t_{5}}$, at least one of the following holds:
      \begin{enumerate}[\ \ \ (1)]
        \item  There is some $x\in \mathcal{H}_{1}^{(N_{0}^{-1})}(\alpha)\cap\mathsf{E}_{(5\varpi +1)t}.x_{0}$ such that
        \begin{enumerate}[\ \ \ (a)]
          \item  $h\mapsto hx$ is injective on $\mathsf{E}_{t}$.
          \item We have
          \[f_{t,\gamma}(z)\leq e^{Nt}\]  
          for all $\gamma\in(0,1)$, $z\in \mathsf{E}_{t}.x$.
        \end{enumerate}
        \item There is $y\in \mathcal{H}_{1}(\alpha)$, $C_{18}>0$ such that   
          \begin{itemize}
            \item  $d(y,x)\leq \hyperlink{2024.12.C18}{C_{18}} e^{(N_{0}-N)t}$,
            \item The Veech group $\SL(y)\subset G$ is a non-elementary Fuchsian group.
          \end{itemize}  
          If we restrict our attention to $\alpha=(2)$, then we further  have
          \begin{itemize}
            \item   $y$ generates a Teichm\"{u}ller curve of discriminant $\leq  \hyperlink{2024.12.C18}{C_{18}} e^{N_{0}t}$.
          \end{itemize}
      \end{enumerate} 
\end{enumerate}
\end{thm}  

Let $\epsilon_{0}>0$ be as defined in (\ref{effective2024.6.02}), $N_{0}\geq \epsilon_{0}^{-1}$ that will be determined later.
\begin{lem}\label{closing2024.3.31}
   Let $x\in \mathcal{H}_{1}^{(\epsilon_{0})}(\alpha)$. Then for every $z\in \mathsf{E}_{t}.x$, we have
   \[|F_{z}(t)|\ll e^{(3\hyperlink{2024.12.k6}{\kappa_{6}}+1)t}.\]
\end{lem}
\begin{proof}  
   Note that since $x\in\mathcal{H}_{1}^{(\epsilon_{0})}(\alpha)$, we have  
   \begin{equation}\label{closing2024.3.47}
     \ell(\mathsf{h}x)>\epsilon_{0} e^{-t}  
   \end{equation}
   for all $\mathsf{h}\in\mathsf{E}_{t}$.
   
   Let $z\in\mathsf{E}_{t}.x$, $w\in F_{z}(t)$, and let  $\epsilon_{t}\coloneqq \textstyle{\frac{1}{100}}\epsilon_{0}^{\hyperlink{2024.12.k6}{\kappa_{6}}}e^{-\hyperlink{2024.12.k6}{\kappa_{6}}t}$. Then since $z+w\in\mathsf{E}_{t}.x$, we obtain
   \[B_{G}(\epsilon_{t}).(z+w)\subset B_{G}(\epsilon_{t}+\delta)\cdot a_{t}\cdot u_{[0,1]}.x.\]  
 By (\ref{closing2024.3.47}), Corollary \ref{closing2024.3.62} and Lemma \ref{closing2023.07.1},   the map $(\mathsf{h},w)\mapsto  \mathsf{h}(z+w)$ is injective over $B_{G}(\epsilon_{t})\times  B_{H^{\perp}_{\mathbb{C}}(z)}(\epsilon_{t})$. Hence we have 
   \[B_{G}(\epsilon_{t}).(z+w_{1})\cap B_{G}(\epsilon_{t}).(z+w_{2})=\emptyset \]
   for all distinct $w_{1},w_{2}\in F_{z}(t)$. Now note that 
   \[\Leb_{G}(B_{G}(\epsilon_{t}+\delta)\cdot a_{t}\cdot u_{[0,1]})\ll e^{t}\ \ \ \text{and}\ \ \ \Leb_{G}(B_{G}(\epsilon_{t}))\gg e^{-3\hyperlink{2024.12.k6}{\kappa_{6}}t}.\]
   This establishes the claim.
\end{proof}
 In the following, we fix the coefficients: 
   \begin{itemize} 
  \item  By (\ref{effective2024.6.02}), let $t\geq   t^{\prime}= \hyperlink{2024.12.k4}{\kappa_{4}}|\log \ell(x_{0})|+|\log 10^{-3}|+\hyperlink{2024.3.C6}{C_{6}}$ be large enough so that
  \begin{equation}\label{closing2024.3.37}
    |\{r\in J:a_{t}u_{r}x_{0}\not\in \mathcal{H}_{1}^{(\epsilon_{0})}(\alpha)\}|\leq \textstyle{\frac{1}{100}}|J|
  \end{equation} 
  for all $J\subset[0,1]$ with $|J|\geq 10^{-3}$.
  \item Let $\varpi\geq 1$ be a constant that will be determined later. See (\ref{closing2024.3.59}).
  \item Let $x_{1}=a_{t}u_{r_{0}}x_{0}$ for $r_{0}\in[0,1/2]$  be so that 
      \[x_{1}\in \mathcal{H}_{1}^{(\epsilon_{0})}(\alpha)\ \ \ \text{and}\ \ \ a_{5\varpi t}x_{1}\in \mathcal{H}_{1}^{(\epsilon_{0})}(\alpha).\]
  \item For $r\in[0,1]$, let $h_{r}\coloneqq  a_{5\varpi t}u_{r}$. Note that for all $r\in[0,1]$,
  \[h_{r}x_{1}\in a_{(5\varpi+1)t}u_{[0,1]}x_{0}.\] 
\end{itemize} 
 \begin{lem}
    Assume that Theorem \ref{effective2024.6.01}(1) does not hold for $x=h_{r}x_{1}$ for any $r\in[0,1]$. Then  for $N_{0}> 3\hyperlink{2024.12.k6}{\kappa_{6}}+2$ and $t\geq t^{\prime\prime}$ is large enough, there exist $1\neq s_{r}\in G$ with $\|s_{r}\|\ll e^{2t}$ and $e\neq \gamma_{r}\in\Gamma$ such that  
    \begin{equation}\label{effective2024.6.03}
      d(h_{r}^{-1}s_{r}h_{r}\tilde{x}_{1},\tilde{x}_{1}\gamma_{r})\leq e^{(20\varpi-N_{0})t}.
    \end{equation} 
 \end{lem}
 \begin{proof}
By the assumption, for all $r\in[0,1]$ with $h_{r}x_{1}\in \mathcal{H}_{1}^{(\epsilon_{0})}(\alpha)$, we have 
      \begin{enumerate}[\ \ \ (i)]
        \item  either there exists $z\in \mathsf{E}_{t}\cdot h_{r}x_{1}$ so that $f_{t,\gamma}(z)>e^{Nt}$,
        \item or the map $\mathsf{h}\mapsto\mathsf{h}h_{r}x_{1}$ is not injective on $\mathsf{E}_{t}$.
      \end{enumerate} 
      
  For the case (i), since $h_{r}x_{1}\in \mathcal{H}_{1}^{(\epsilon_{0})}(\alpha)$, we have 
\begin{equation}\label{closing2024.3.30}
  \ell(\mathsf{h}h_{r}x_{1})\gg e^{-t}
\end{equation}
for all $\mathsf{h}\in\mathsf{E}_{t}$.
\begin{itemize}
  \item Suppose for some  $z=\mathsf{h}_{1}h_{r}x_{1}\in \mathsf{E}_{t}h_{r}x_{1}$, we have $f_{t,\gamma}(z)>e^{Nt}>e^{2N_{0}t}$.
\end{itemize} 
By the definition of $f_{t,\gamma}$, if $F_{z}(t)=\emptyset$, then $f_{t,\gamma}(z)\ll e^{t}$. Hence, we have $F_{z}(t)\neq\emptyset$ for $t\gg 1$. Recall that from Lemma \ref{closing2024.3.31} that $| F_{z}(t)|\ll e^{(3\hyperlink{2024.12.k6}{\kappa_{6}}+1)t}$.

Thus, if $N_{0}> 3\hyperlink{2024.12.k6}{\kappa_{6}}+2$ and $t\geq t^{\prime\prime}$ is large enough, there exists some $w\in F_{z}(t)$ such that 
\[0<\|w\|_{z}\leq e^{(3\hyperlink{2024.12.k6}{\kappa_{6}}+2-N_{0})t}.\]
It follows that for some  $\mathsf{h}_{1},\mathsf{h}_{2}\in \mathsf{E}_{t}$ and $w\in H^{\perp}_{\mathbb{C}}(\mathsf{h}_{1}h_{r}x_{1})$, we have $\mathsf{h}_{1}h_{r}x_{1}+w=\mathsf{h}_{2}h_{r}x_{1}$, i.e. 
\[\bar{u}_{s_{1}}a_{t}u_{r_{1}}h_{r}x+w=\bar{u}_{s_{2}}a_{t}u_{r_{2}}h_{r}x.\] Then we have 
\begin{equation}\label{closing2024.3.32}
 d(h_{r}^{-1}s_{r}h_{r}x_{1},x_{1})\leq e^{(20\varpi-N_{0})t}
\end{equation}
where $s_{r}=\mathsf{h}_{2}^{-1}\mathsf{h}_{1}$.  
Since $x_{1}\in  \mathcal{H}_{1}^{(\epsilon_{0})}(\alpha)$, we have  
\[ d(h_{r}^{-1}s_{r}h_{r}\tilde{x}_{1},\tilde{x}_{1}\gamma_{r})\leq e^{(20\varpi-N_{0})t}\]
where $1\neq s_{r}\in G$ with $\|s_{r}\|\ll e^{2t}$ and $e\neq \gamma_{r}\in\Gamma$. Thus, we establish (\ref{effective2024.6.03}) in this case.

Similarly,   for the case (ii), if $\mathsf{h}\mapsto \mathsf{h} h_{r}x_{1}$ is not injective, we conclude that 
\[h_{r}^{-1}s_{r}h_{r}\tilde{x}_{1}=\tilde{x}_{1}\gamma_{r}.\] 
Again we establish (\ref{effective2024.6.03}).
 \end{proof}

In view of (\ref{effective2024.6.03}), we calculate
\begin{equation}\label{closing2024.3.48}
   h_{r}^{-1}s_{r}h_{r} = h_{r}^{-1}\begin{bmatrix}
S_{11}^{(r)} & S_{12}^{(r)}\\
S_{21}^{(r)} & S_{22}^{(r)}
\end{bmatrix}h_{r}=\begin{bmatrix}
P_{11}^{t,r}(r) & P_{12}^{t,r}(r)\\
P_{21}^{t,r} & P_{22}^{t,r}(r)
\end{bmatrix} 
\end{equation} 
where
\begin{alignat*}{7}
 P_{11}^{t,r}(R)&=e^{5\varpi t}S_{21}^{(r)} R+S_{11}^{(r)}& , & \ \ \   & P_{12}^{t,r}(R)= &-e^{5\varpi t}S_{21}^{(r)} R^{2} +(S_{22}^{(r)}-S_{11}^{(r)}) R + e^{-5\varpi t}S_{12}^{(r)} & , \\ 
 P_{21}^{t,r}&=e^{5\varpi t}S_{21}^{(r)} &  &, \ \ \ & P_{22}^{t,r}(R)=& -e^{5\varpi t}S_{21}^{(r)} R+S_{22}^{(r)}  &. 
\end{alignat*}
Note that the leading coefficients of  the polynomials $P_{11}^{t,r}(\cdot)$, $P_{12}^{t,r}(\cdot)$, $P_{22}^{t,r}(\cdot)$ are $P_{21}^{t,r}$.

 \begin{lem}\label{closing2024.3.42}
 Given $T>0$,  suppose that $\|h_{r}^{-1}s_{r}h_{r}-I\|\geq  T$. Then we have 
\begin{equation}\label{closing2024.3.36}
   \max\{e^{5\varpi}|S_{21}^{(r)}|,|S_{11}^{(r)}-S_{22}^{(r)}|\}\gg T.
\end{equation}
   Moreover, if there is a unipotent element $\hat{u}\in G$ so that 
   \begin{equation}\label{closing2024.3.52}
  h_{r}^{-1}s_{r}h_{r}\stackrel[]{e^{(\varpi-N_{0})t}}{\sim}  \hat{u} \text{\ \ \ or \ \ \ }h_{r}^{-1}s_{r}h_{r}\stackrel[]{e^{(\varpi-N_{0})t}}{\sim}  -\hat{u}, 
   \end{equation}
   then we must have
    \begin{equation}\label{closing2024.3.41}
      |P_{21}^{t,r}|=e^{5\varpi}|S_{21}^{(r)}|\geq T.
    \end{equation} 
 \end{lem}
 \begin{proof} 
    By (\ref{closing2024.3.48}), we have   
\begin{equation}\label{closing2024.3.35}
\max\{e^{5\varpi}|S_{21}^{(r)}|,|S_{11}^{(r)}-1|,|S_{22}^{(r)}-1|\}\geq T.
\end{equation}
Assume  that $e^{5\varpi t}|S_{21}^{(r)}|<T$. Then 
\begin{equation}\label{closing2024.3.43}
  \max\{|S_{11}^{(r)}-1|,|S_{22}^{(r)}-1|\}\geq T\ \ \ \text{ and }\ \ \ |S_{12}^{(r)}S_{21}^{(r)}|\ll T e^{(4-5\varpi) t}.
\end{equation}   
 Thus, 
\begin{equation}\label{closing2024.3.44}
   |S_{11}^{(r)}S_{22}^{(r)}-1|\ll T e^{(4-5\varpi) t}.
\end{equation}
Then by (\ref{closing2024.3.35}), if $|S_{11}^{(r)}-1|\geq T$, then 
\[|S_{11}^{(r)}-S_{22}^{(r)}|=\left|S_{11}^{(r)}-\frac{1+O(T e^{(4-5\varpi) t})}{S_{11}^{(r)}}\right|\gg T. \]
Hence, we get (\ref{closing2024.3.36}).

Now we  assume that (\ref{closing2024.3.52}) holds and $e^{5\varpi t}|S_{21}^{(r)}|<T$. Then   $h_{r}^{-1}s_{r}h_{r}$ is very near  either a unipotent element $u^{\star}$ or its minus, we conclude that 
\begin{equation}\label{closing2024.3.45}
   \min\{|S_{11}^{(r)}+S_{22}^{(r)}-2|,|S_{11}^{(r)}+S_{22}^{(r)}+2|\}\ll e^{(11\varpi-N)t}.
\end{equation}
Equations (\ref{closing2024.3.44}) and (\ref{closing2024.3.45}) contradict (\ref{closing2024.3.43}) if $N,t$ are large enough, hence necessarily $e^{5\varpi t}|S_{21}^{(r)}|<T$. 
 \end{proof}

Now we define:
\begin{itemize}
  \item  Let $J^{(\epsilon_{0})} =\{r\in[1/2,1]:h_{r}x_{1}\in \mathcal{H}_{1}^{(\epsilon_{0})}(\alpha)\}$.
  \item Let $\delta_{\Gamma}=\min_{\gamma\in\Gamma\setminus\{e\}} d_{\Gamma}(\gamma,e)$. 
\end{itemize}
It follows that 
\[\delta_{\Gamma}\leq  d_{\Gamma}(e,\gamma_{r})= d(\tilde{x}_{1},\tilde{x}_{1}\gamma_{r})\leq d(\tilde{x}_{1},h_{r}^{-1}s_{r}h_{r}\tilde{x}_{1})+ d(h_{r}^{-1}s_{r}h_{r}\tilde{x}_{1},\tilde{x}_{1}\gamma_{r}).\]
So by Lemma \ref{closing2023.12.3} and (\ref{effective2024.6.03}), we get 
\begin{equation}\label{closing2024.3.53}
  1\ll\delta_{\Gamma}\ll d(\tilde{x}_{1},h_{r}^{-1}s_{r}h_{r}\tilde{x}_{1})\leq d_{G}(e,h_{r}^{-1}s_{r}h_{r}).
\end{equation} 

\begin{lem}\label{closing2024.3.54} 
  There are $\gg e^{2\varpi t}$ distinct elements in $\{\gamma_{r}\in\Gamma :r\in J^{(\epsilon_{0})}\}$.
\end{lem} 
\begin{proof}
   By (\ref{closing2024.3.37}), we have $|J^{(\epsilon_{0})}|\geq 1/4$ for $t\gg 1$. Fix $r\in J^{(\epsilon_{0})}$ as above, and consider the set of $r^{\prime}\in J^{(\epsilon_{0})}$ so that $\gamma_{r}=\gamma_{r^{\prime}}$. Then for each such $r^{\prime}$, 
\[d(h_{r}^{-1}s_{r}h_{r}\tilde{x}_{1} ,\tilde{x}_{1}\gamma_{r})\ll e^{(4\varpi-N_{0})t},\ \ \  d(h_{r^{\prime}}^{-1}s_{r^{\prime}}h_{r^{\prime}}\tilde{x}_{1} ,\tilde{x}_{1}\gamma_{r^{\prime}})\ll e^{(4\varpi-N_{0})t}  .\]
It follows that 
\[d(h_{r}^{-1}s_{r}h_{r}\tilde{x}_{1} ,h_{r^{\prime}}^{-1}s_{r^{\prime}}h_{r^{\prime}}\tilde{x}_{1} )\ll e^{(4\varpi-N_{0})t}. \]
In other words, we have  
\begin{equation}\label{closing2024.3.50}
  d_{G}(h_{r^{\prime}}h_{r}^{-1}s_{r}h_{r}h_{r^{\prime}}^{-1},s_{r^{\prime}})\ll e^{(20\varpi-N_{0})t}.
\end{equation} 

On the other hand, in view of, one calculates 
\begin{equation}\label{closing2024.3.51}
h_{r^{\prime}}h_{r}^{-1}s_{r}h_{r}h_{r^{\prime}}^{-1} = \begin{bmatrix}
P_{11}^{t,r}(r^{\prime}-r) & e^{5\varpi t}P_{12}^{t,r}(r^{\prime}-r)\\
e^{-5\varpi t}P_{21}^{t,r} & P_{22}^{t,r}(r^{\prime}-r)
\end{bmatrix} .
\end{equation} 
For every $r\in J^{(\epsilon_{0})}$, let 
\[J(t,r)\coloneqq\{ r^{\prime}\in J^{(\epsilon_{0})} :|P_{12}^{t,r}(r^{\prime}-r)|\leq  e^{-4\varpi t}\}\]
In view of (\ref{closing2024.3.36}), for every $r\in J^{(\epsilon_{0})}$,
 we have
\begin{equation}\label{closing2024.3.39}
\left|J(t,r)\right|\ll e^{-2\varpi t}
\end{equation}
 since by (\ref{closing2024.3.53}) and (\ref{closing2024.3.36}), at least one of the coefficients of the quadratic polynomial $P_{12}^{t,r}$ is of size $\gg 1$. 

Now if $r^{\prime}\in J^{(\epsilon_{0})}\setminus J(t,r)$, then $|e^{5\varpi t}P_{12}^{t,r}(r^{\prime}-r)|> e^{\varpi t}$. Thus, we have  
\[\|h_{r^{\prime}}h_{r}^{-1}s_{r}h_{r}h_{r^{\prime}}^{-1} \|> e^{\varpi t}>\|s_{r^{\prime}}\|,\]
  for all $r^{\prime}\in J^{(\epsilon_{0})}\setminus J(t,r)$, in contradiction to (\ref{closing2024.3.50}).

Therefore, for every $r\in J^{(\epsilon_{0})}$, we have $|\{r^{\prime}\in J^{(\epsilon_{0})}:\gamma_{r^{\prime}}=\gamma_{r}\}|\ll e^{-2\varpi t}$, and   so the set $\{\gamma_{r}\in\Gamma :r\in J^{(\epsilon_{0})}\}$ has at least $\gg e^{2\varpi t}$ distinct elements.
\end{proof}
Next, we apply the  closing lemma  (Theorem \ref{closing2024.12.2})   to the $P$-action on $\mathcal{H}_{1}(\alpha)$.
\begin{itemize}
  \item \hypertarget{2024.3.k7} 
   Let  $\kappa_{7}>0$ be as defined in Theorem \ref{closing2024.12.2}. It depends only on the underlying space $\mathcal{H}_{1}(\alpha)$.
  \item  Let $\varpi\geq  \hyperlink{2024.12.k11}{\kappa_{11}}$.
  \item  Let $\{r_{i}:1\leq i\leq l\}\subset J^{(\epsilon_{0})}$ be a sequence so that 
\[h_{r_{i}}^{-1}s_{r_{i}}h_{r_{i}}\neq h_{r_{j}}^{-1}s_{r_{j}}h_{r_{j}}\]
for $i\neq j$. By Lemma \ref{closing2024.3.54}, we may choose $l\gg  e^{2\varpi t}$.  
\item Let $g_{i}=h_{r_{i}}^{-1}s_{r_{i}}h_{r_{i}}$. Then $\{g_{i}\}\subset B_{G}( e^{20\varpi t})$ is a  collection of  $1$-separated points. Also,  by (\ref{effective2024.6.03}), for sufficiently large $N_{0}$ and $t\geq t^{\prime\prime}$, we have
    \begin{equation}\label{closing2024.3.34}
 d(g_{i}x_{1},g_{j}x_{1})< e^{(21\varpi -N_{0})t}\ \ \ \text{ and }\ \ \  \|\gamma_{r}\|\leq e^{21\varpi t}.
\end{equation}     
\end{itemize}

 Then by Proposition \ref{closing2024.3.40}, we get a point $\tilde{y}\in B(\tilde{x},e^{(\hyperlink{2024.12.k11}{\kappa_{11}}-N_{0})t})$, and  a  collection of  $\frac{1}{2}$-separated points 
 \begin{equation}\label{closing2024.3.55}
   \{g^{\ast}_{i}:0\leq i\leq l\}\subset B_{G}(2 e^{20\varpi t}),\ \ \ d_{G}(g^{\ast}_{i},g_{i})\leq e^{(3\varpi-N_{0})t}
 \end{equation} 
 such that 
 \[g^{\ast}_{i}\tilde{y}=\tilde{y}\gamma_{r_{i}}.\]  
 It follows that 
 \[ \{g^{\ast}_{0},\ldots,g^{\ast}_{l}\}\subset\SL(y)\cap B_{G}(2 e^{20\varpi t}). \]
  Moreover, recall from (\ref{closing2024.3.20}) that $\SL(y)$  has an alternate definition in terms of affine automorphisms $\Aff(y)$. Each affine automorphism determines a mapping class in $\Gamma=\Mod(M,\Sigma)$. Thus, we can view the Veech group $\SL(y)$ is isomorphic to a subgroup of $\Gamma$. In addition,  one may deduce that  $\SL(y)$ is isomorphic   to a subgroup $\Gamma_{y}$ of $\Sp(p(y))$ where $p(y)$ denotes the $2$-plane $p(y)=\mathbb{R}\Ree y\oplus \mathbb{R}\Imm y\subset H^{1}(M;\mathbb{C})$. 
  
  In what follows, we shall show that  $\SL(y)\subset G$ is non-elementary.
  \begin{prop}\label{closing2023.12.17}
     Let the notation and assumptions be as above. Then $\SL(y)\subset G$ is a non-elementary Fuchsian group.
  \end{prop}
  \begin{proof}
Suppose that $\SL(y)\subset G$ is elementary. Then by the classification of Fuchsian groups, we see that one of the following hold (a proof can be found in \cite[Proposition 3.1.2]{hubbard2006teichmuller1}):
 \begin{itemize}
   \item   $\SL(y)$ is a finite cyclic group;
   \item  $\SL(y)$ is an infinite  cyclic group generated by a single parabolic element;
   \item  $\SL(y)$ contains an infinite  cyclic group generated by a single hyperbolic element of index at most two.
 \end{itemize}
 
 First, assume  that $\SL(y)$ is finite. Then recall that Minkowski proved that $\GL_{2g+|\Sigma|-1}(\mathbb{Z})$ has only finitely many isomorphism classes of finite subgroups (see e.g. \cite{kuzmanovich2002finite}). In particular, it contradicts the fact that $|\SL(y)|=|\Gamma_{y}|\gg e^{2\varpi t}$ for sufficiently large $t\geq \hyperlink{2024.12.t5}{t_{5}}=\max\{t^{\prime},t^{\prime\prime}\}$.
 
 Now assume that $\SL(y)$ (and so $\Gamma_{y}$) contains a infinite cyclic subgroup generated by a single hyperbolic element of index at most two. Then by considering the trace of hyperbolic elements in $\GL_{2g+|\Sigma|-1}(\mathbb{Z})$, we see that there exist absolute \hypertarget{2024.12.k14} constants $C=C(g,|\Sigma|)>0$ and $\kappa_{14}=\aleph_{14}(g,|\Sigma|)>0$ such that
 \[|\Gamma_{y}\cap B_{\GL}(T)|\leq C (\log T)^{\hyperlink{2024.12.k14}{\kappa_{14}}}\]
for any $T\gg 1$. Again, it contradicts the fact that $|\Gamma_{y}\cap B_{\GL}(e^{20\varpi t})|\gg e^{2\varpi t}$ for sufficiently large $t\geq \hyperlink{2024.12.t5}{t_{5}}=\max\{t^{\prime},t^{\prime\prime}\}$.

 Thus, in the following, we assume that $\SL(y)$ is a unipotent subgroup of $G$. 
  Since $\SL(y)$ is unipotent, there exists some $k$ so that $\SL(y)\subset kUk^{-1}$. Via the Iwasawa decomposition,  $k$ can be chosen to be in $\SO(2)$ - for our purposes, we only need to know that the size of $k$ can be bounded by an absolute constant.
 
 Then since unipotent elements have a polynomial growth, we have 
 \begin{equation}\label{closing2024.3.59}
 |\{\gamma_{r}\in\Gamma_{y}:\|\gamma_{r}\|\leq e^{3 t}\}|\ll e^{\varpi t} 
 \end{equation} 
 for some $\varpi>0$ depending only on the size of $\GL$, i.e. genus $g$ and $|\Sigma|$.

 It follows that for any $1\leq i\leq l$, 
 \begin{equation}\label{closing2024.3.46}
   d_{G}( h_{r_{i}}^{-1}s_{r_{i}}h_{r_{i}}, g_{i}^{\ast})\leq   e^{(3\varpi-N_{0})t} 
 \end{equation}
 for  $g_{i}^{\ast}\in kUk^{-1}$. Recall from the beginning of the proof that $h_{0}x_{1}\in \mathcal{H}_{1}^{(\epsilon_{0})}(\alpha)$. We write $r_{0}=0$.
 
In (\ref{closing2024.3.46}), we obtain a unipotent group $kUk^{-1}$ so that the elements look like $h_{r}^{-1}s_{r}h_{r}$. However, under a detailed (but elementary) calculation of matrices, we shall show that a unipotent subgroup cannot look like that. Let us write $k=\begin{bmatrix}
k_{11} & k_{12}\\
k_{21} & k_{22}
\end{bmatrix}$, then for all $g_{i}^{\ast}$, we have 
\[g_{i}^{\ast}=k\begin{bmatrix}
1 & w_{i}\\
0 & 1
\end{bmatrix}k^{-1}=\begin{bmatrix}
1-k_{11}k_{21}w_{i} & k_{11}^{2}w_{i}\\
-k_{21}^{2}w_{i} & 1+k_{11}k_{21}w_{i}
\end{bmatrix}\]
for $w_{i}\in\mathbb{R}$.    We shall show that 
\begin{cla}\label{closing2024.3.58}
   For $r_{i}\in J^{(\epsilon_{0})}$ with $\|\gamma_{r_{i}}\|> e^{3t}$, 
   \begin{equation}\label{closing2024.3.56}
      d_{G}( h_{0}^{-1}s_{0}h_{0}, g_{0}^{\ast})\leq   e^{(3\varpi-N_{0})t},\ \ \  d_{G}( h_{r_{i}}^{-1}s_{r_{i}}h_{r_{i}}, g_{i}^{\ast})\leq   e^{(3\varpi-N_{0})t}
   \end{equation} 
   cannot hold at the same time.
\end{cla}
\begin{proof} Assume for contradiction that (\ref{closing2024.3.56}) holds. Then one calculates
\[\begin{bmatrix}
P_{11}^{t,0}(0) & P_{12}^{t,0}(0)\\
P_{21}^{t,0} & P_{22}^{t,0}(0)
\end{bmatrix} =h_{0}^{-1}s_{0}h_{0}\stackrel[]{e^{(3\varpi-N_{0})t}}{\sim} g_{0}^{\ast}=\begin{bmatrix}
1-k_{11}k_{21}w_{0} & k_{11}^{2}w_{0}\\
-k_{21}^{2}w_{0} & 1+k_{11}k_{21}w_{0}
\end{bmatrix}.\]
By (\ref{closing2024.3.41}) applied with $T=\delta_{\Gamma}$, $|P_{21}^{t,0}|\gg  \delta_{\Gamma}$. Since $\|g\|\ll 1$, comparing the $(2,1)$-entries of the matrices we get $|w_{0}|\gg 1$. On the other hand, since $|P_{12}^{t,0}(0)|=|e^{-5\varpi t}S_{12}^{(0)}|\leq  e^{(2-5\varpi) t}$, comparing the $(1,2)$-entries we conclude that $|k_{11}|\ll e^{-2\varpi t}$. Since $\det(k)=1$, it follows that $|k_{21}|\gg 1$.

Let now $r_{i}\in J^{(\epsilon_{0})}$ be so that $\|\gamma_{r_{i}}\|> e^{3t}$.  
By (\ref{closing2024.3.41}), applied this time with $T=e^{3t}$, we have that $|P_{21}^{t,r_{i}}|\geq e^{3t}$; note also that $|e^{-5\varpi t}S_{12}^{(r_{i})}|\ll e^{-4\varpi t}$. In view of (\ref{closing2024.3.56}), there exists $w_{i}\in \mathbb{R}$ so that 
\begin{equation}\label{closing2024.3.57}
  \begin{bmatrix}
P_{11}^{t,r_{i}}(r_{i}) & P_{12}^{t,r_{i}}(r_{i})\\
P_{21}^{t,r_{i}} & P_{22}^{t,r_{i}}(r_{i})
\end{bmatrix}  =h_{r_{i}}^{-1}s_{r_{i}}h_{r_{i}}\stackrel[]{e^{(3\varpi-N_{0})t}}{\sim} g_{i}^{\ast}=\begin{bmatrix}
1-k_{11}k_{21}w_{i} & k_{11}^{2}w_{i}\\
-k_{21}^{2}w_{i} & 1+k_{11}k_{21}w_{i}
\end{bmatrix}.
\end{equation} 
Now note that $|P_{21}^{t,r_{i}}|\geq e^{3t}$, $\|s_{r_{i}}\|\ll e^{2t}$, $|e^{-5\varpi t}S_{12}^{(r_{i})}|\ll e^{-4\varpi t}$, and $r_{i}\in[\frac{1}{2},1]$. Via (\ref{closing2024.3.48}), we have that 
\[  |P_{12}^{t,r_{i}}(r_{i})| \asymp  |P_{21}^{t,r_{i}}|.\]
In view of (\ref{closing2024.3.57}), it follows that 
 \[|k_{21}^{2}w_{i}|\asymp|k_{11}^{2}w_{i}|.\] 
 It contradicts the fact that $|k_{11}|\ll e^{-2\varpi t}$ and $|k_{21}|\gg 1$.
\end{proof}

Now by Lemma \ref{closing2024.3.54} and (\ref{closing2024.3.59}), there exists $r_{i}\in J^{(\epsilon_{0})}$ with $\|\gamma_{r_{i}}\|> e^{3t}$. Then we see that Claim \ref{closing2024.3.58} contradicts (\ref{closing2024.3.46}). Therefore, we conclude that $\SL(y)$ is not a unipotent subgroup of $G$, and so is not elementary. 
  \end{proof}

Now we restrict our attention to $\mathcal{H}_{1}(2)$. By Proposition \ref{closing2023.12.17}, we know that the Veech group $\SL(y)\subset G$ is a non-elementary Fuchsian group. Then by Theorem \ref{closing2024.3.19}, we conclude  that $y$ generates a Teichm\"{u}ller curve. 
 In addition, since 
 \[ \{g^{\ast}_{0},\ldots,g^{\ast}_{l}\}\subset\SL(y)\cap B_{G}(2 e^{20\varpi t}) \]
 generates $\SL(y)$,    
the discriminant of the Teichm\"{u}ller curve satisfies
\[\disc(\bar{\iota}(y))\leq  \hyperlink{2024.12.C13}{C_{13}} (2 e^{20\varpi t})^{\hyperlink{2024.12.k10}{\kappa_{10}}}.\]
Let 
     $ \hyperlink{2024.12.C18}{C_{18}} =  \hyperlink{2024.12.C13}{C_{13}} 2^{\hyperlink{2024.12.k10}{\kappa_{10}}} $ and $N_{0}\geq\max\{\epsilon_{0}^{-1}, 20\varpi\hyperlink{2024.12.k10}{\kappa_{10}}, \hyperlink{2024.12.k11}{\kappa_{11}}\}$. We establish Theorem \ref{effective2024.6.01}.

\section{Margulis functions}\label{margulis2024.3.04}
In the following, we  prove Theorem \ref{margulis2024.12.1}   by 
 improving the discrete
dimension  of  the additional direction transverse to
the $\SL_{2}(\mathbb{R})$-orbit via the technique of Margulis functions. More specifically, we show the following:   
\begin{thm}\label{margulis2024.4.30}  
    Let $\epsilon\in(0,\frac{1}{10})$, \hypertarget{2024.04.k123} $\gamma\in(0,1)$, \hypertarget{2024.04.t345} $\eta\in(0,N_{0}^{-1})$, \hypertarget{2024.04.NN789} $N\geq 2N_{0}$, $x_{0}\in \mathcal{H}_{1}(\alpha)$, where $N_{0}$ is given in Theorem    \ref{effective2024.6.01}.  Then there exist  
\begin{itemize} 
  \item $\varkappa=\varkappa(N,\gamma,\varpi)>0$,
  \item $\kappa_{1}=\kappa_{1}(N,\gamma,\alpha,\epsilon)>0$,
  \item $t_{1}=t_{1}(\gamma,\epsilon,\eta,\alpha,\ell(x_{0}))>0$,
  \end{itemize}
  such that for $\kappa\in(0,\hyperlink{2024.12.k1}{\kappa_{1}})$ and $t\geq \hyperlink{2024.12.t1}{t_{1}}$, at least one of the following holds:  
\begin{enumerate}[\ \ \ (1)]
  \item  There exists $x_{1}\in \mathcal{H}_{1}^{(\eta)}(\alpha)$,   and a finite subset $F\subset B_{H^{\perp}_{\mathbb{C}}(x_{1})}(e^{-\kappa t})$ containing $0$ with
      \[|F|\geq  e^{ \frac{1}{2}t}  \]
   such that:
      \begin{itemize}
        \item  $x_{1}+F \subset \mathsf{E}_{\hyperlink{2024.12.NN}{\varkappa}t,[0,2]}(e^{10-\kappa t}) .x_{0}$.
        \item  For all $w\in F$, we have
        \[\sum_{\substack{w^{\prime}\neq w\\ w^{\prime}\in F}} \|w-w^{\prime}\|^{-\gamma}\leq 2\cdot|F|^{1+\epsilon}.\]  
      \end{itemize}
  \item (Existence  \hypertarget{2024.3.C} of a small closed orbit near $x_{0}$)  There is $y\in \mathcal{H}_{1}(\alpha)$ such that   
          \begin{itemize}
            \item  $d(y,x_{0})\leq \hyperlink{2024.12.C18}{C_{18}} e^{(N_{0}-N)t}$,
            \item The Veech group $\SL(y)\subset G$ is a non-elementary Fuchsian group.
          \end{itemize}  
          If we restrict our attention to $\alpha=2$, then we further  have
          \begin{itemize}
            \item   $y$ generates a Teichm\"{u}ller curve of discriminant $\leq \hyperlink{2024.12.C18}{C_{18}} e^{N_{0}t}$.
          \end{itemize}
\end{enumerate}
\end{thm} 

\subsection{Preparation}
\begin{defn}[Skeleton] A \textit{skeleton}\index{skeleton} $\mathcal{E}=\mathcal{E}(x,F,  E)\subset \mathcal{H}_{1}(\alpha)$ is a Borel subset of  $\mathcal{H}_{1}(\alpha)$ equipped with the following:
\begin{itemize}
  \item  A \textit{base point}\index{base point} $x\in\mathcal{H}_{1}(\alpha)$.
  \item A finite subset $F\subset H^{\perp}_{\mathbb{C}}(x)$ containing $0$. We refer to it as the \textit{spine}\index{spine} of $\mathcal{E}$.
  \item A function $E:F\rightarrow\mathcal{B}_{B}(G)$, $w\mapsto E_{w}$, where $\mathcal{B}_{B}(G)$ denotes the collection of Bounded Borel subsets of $G$. We refer to $E$ as the \textit{bone function}\index{bone function} of $\mathcal{E}$.
\end{itemize}
such that
\begin{itemize}
  \item  The map $h\mapsto h.(x+w)$ is injective on $E_{w}$ for all $w\in F$.
  \item For $w_{1}\neq w_{2}\in F$, $E_{w_{1}}.(x+w_{1})\cap E_{w_{2}}.(x+w_{2})=\emptyset$.
  \item $\mathcal{E}$ is the disjoint union of the bones:
\[\mathcal{E}=\bigsqcup_{w\in F}E_{w}.(x+w).\] 
\end{itemize}
There naturally is a Lebesgue measure $\Leb_{\mathcal{E}}$ on  a skeleton $\mathcal{E}$ given by   
      \begin{equation}\label{effective2022.2.39}
           \Leb_{\mathcal{E}}\coloneqq\frac{1}{\sum_{w}\Leb_{G}(E_{w})}\sum_{w}\Leb_{E_{w}}
      \end{equation} 
where $\Leb_{E_{w}}$ denotes the   pushforward of the Haar measure $\Leb_{G}|_{E_{w}}$.
\end{defn}
Roughly speaking, a skeleton a Borel set which is a disjoint finite union of local $G$-orbits. This is motivated by a local observation of the $G$-thickening of a long horocycle.
\begin{ex}[{$G$-thickening of horocycles}]\label{margulis2024.3.12} Let $r_{u},r_{G}>0$, $x\in \mathcal{H}_{1}(\alpha)$, and $F\subset H^{\perp}_{\mathbb{C}}(x)$ a finite set. For every $w\in F$, let 
\[E_{w}=\mathsf{E}_{[-r_{u},r_{u}]}(r_{G})\coloneqq B_{G}(r_{G})\cdot u_{[-r_{u},r_{u}]}. \]
Then the skeleton $\mathcal{E}(x,F,E)$ consists of the $G$-thickening of horocycles. It is also  denoted by 
   $\mathcal{E}(x,F,r_{G},r_{u})$.
\end{ex}
\begin{ex}[{$g$-shifting of spines}]  
Let $g\in G$, $\mathcal{E}=\mathcal{E}(x,F,E)$ a skeleton. Then for every $(g,z)\in G\times\mathcal{E}$, we define
  \begin{equation}\label{margulis2024.3.05}
     F_{g,z}(\mathcal{E})\coloneqq\{w\in  H^{\perp}_{\mathbb{C}}(gz):0<\|w\|_{gz}<L(gz),\ gz+w\in g\mathcal{E}\}.
  \end{equation}    
  Since $E_{w}$ is bounded for every $w\in F$ and $F$ is finite, $F_{g,z}(\mathcal{E})$ is a finite set for all $(g,z)\in G\times\mathcal{E}$. Ideally, $F_{g,z}(\mathcal{E})$ is given by picking a subset of $gF$, and adding  points $g\mathcal{E}$ newly occurred in a small cross section  $gz+H^{\perp}_{\mathbb{C}}(gz)$. In particular, we have 
  $\mathcal{E}\cap B(z,L(z))\subset F_{e,z}(\mathcal{E})$ where $L(z)$ is defined in (\ref{margulis2024.4.14}). 
  If $E_{w}$ are small enough for all $w\in F$, say e.g. $E_{w}\subset B(x,L(x))$, then we  have
  \[\mathcal{E}\cap B(z,L(z))= F_{e,z}(\mathcal{E}).\]
\end{ex}

\begin{ex}[{Local observation}]\label{margulis2024.4.36} Let $y^{\ast}\in\mathcal{H}_{1}(\alpha)$, $w_{i}^{\ast}\in B_{H^{\perp}_{\mathbb{C}}(y)}(1)$, $\mathcal{I}$ a finite index set. For  $i\in\mathcal{I}$, let $\mathsf{V}_{i}\subset   \mathsf{V}\subset G$  be connected open  neighborhoods of the identity in $G$. Suppose that $\mathcal{E}^{\ast}\subset \mathcal{H}_{1}(\alpha)$ is a Borel subset  such that 
\[\mathcal{E}^{\ast}\subset \bigsqcup_{i\in\mathcal{I}}\mathsf{V}_{i}.(y^{\ast}+w_{i}^{\ast}).\]
 (In particular, $\mathcal{E}^{\ast}$ is contained in a neighborhood.) Then there exist $y_{i}\in \mathcal{E}^{\ast}$ and $\mathsf{h}_{i}\in \mathsf{V}_{i}$ such that
\[y_{i}=\mathsf{h}_{i}(y^{\ast}+w_{i}^{\ast}).\]
Moreover, write 
\[w_{ij}\coloneqq \mathsf{h}_{j}(w_{i}^{\ast}-w_{j}^{\ast})\in H^{\perp}_{\mathbb{C}}(y_{j}). \]
Then for every $i,j\in\mathcal{I}$, we have 
   \begin{equation}\label{margulis2024.4.37}
    y_{i}=\mathsf{h}_{i}(y^{\ast}+w_{i}^{\ast})= \mathsf{h}_{i}\mathsf{h}_{j}^{-1}[(\mathsf{h}_{j}(y^{\ast}+w_{j}^{\ast}))+\mathsf{h}_{j}(w_{i}^{\ast}-w_{j}^{\ast})]= \mathsf{h}_{i}\mathsf{h}_{j}^{-1}[y_{j}+w_{ij}].
   \end{equation}
   In particular, one calculates 
   \begin{equation}\label{margulis2024.4.38}
  w_{jk}-w_{ik}=\mathsf{h}_{k}[\mathsf{h}_{j}^{-1}y_{j}-\mathsf{h}_{i}^{-1}y_{i}]=\mathsf{h}_{k}[\mathsf{h}_{i}^{-1}(y_{i}+w_{ji})-\mathsf{h}_{i}^{-1}y_{i}]=\mathsf{h}_{k}\mathsf{h}_{i}^{-1}w_{ji}
   \end{equation}
   for any $i,j,k\in\mathcal{I}$.
   Now we consider 
   \[F\coloneqq \{w_{i1}\in   H^{\perp}_{\mathbb{C}}(y_{1}):i\in\mathcal{I}\}.\]
If $\mathsf{V}\subset \mathsf{E}_{[-r_{u},r_{u}]}(r_{G})$ for $r_{u},r_{G}>0$, then the skeleton $\mathcal{E}=\mathcal{E}(y_{1},F,r_{G},r_{u})$ includes all $\{y_{i}:i\in \mathcal{I}\}$. Moreover, if $\mu$ is some measure on $\mathcal{E}$, then we have
  \begin{equation}\label{margulis2024.4.40}
  |F|=|\mathcal{I}|\geq\frac{\mu(\mathcal{E}^{\ast})}{\mu(\mathsf{V}.(y^{\ast}+w_{i}^{\ast}))}.
  \end{equation} 
\end{ex}

In application, we consider  the $g$-shifting of certain skeleton $\mathcal{E}=\mathcal{E}(x,F,r_{G},r_{u})$. After choosing   $(g,z)\in G\times\mathcal{E}$, we obtain a new skeleton $\mathcal{E}_{1}=\mathcal{E}(gz,F_{g,z}(\mathcal{E}),r_{G},r_{u})$ by a local observation.  The new skeleton $\mathcal{E}_{1}$ may also be considered as a modification of $g\mathcal{E}$.

\begin{defn}[Local density function] Let $\mathcal{E}=\mathcal{E}(x,F,E)\subset\mathcal{H}_{1}(\alpha)$ be a skeleton. Define the \textit{local density function}\index{local density function} $f_{\mathcal{E}}=f_{\mathcal{E},\gamma}:G\times\mathcal{E}\rightarrow[1,\infty)$ by  
\begin{equation}\label{margulis2024.3.06}
        f_{\mathcal{E}}:(g,z)\mapsto\left\{
    \begin{array}{ll}   \displaystyle{\sum_{w\in  F_{g,z}(\mathcal{E})}}\|w\|_{gz}^{-\gamma}  &,\text{ if }  F_{g,z}(\mathcal{E})\neq\emptyset\\
   \ell(hz)^{-\gamma} &,\text{ otherwise }
                     \end{array}\right. .
      \end{equation}
\end{defn}
Roughly speaking, $f_{\mathcal{E}}(g,z)$ observes the local density of $F_{g,z}(\mathcal{E})$. A large $f_{\mathcal{E}}(g,z)$ indicates a high density of $F_{g,z}(\mathcal{E})$ around at $gz$.   In application, given some small constant $c^{\ast}>0$,  one calculates 
\begin{align}
f_{\mathcal{E}}(g,z)=  \sum_{w\in F_{g,z}(\mathcal{E})}\|w\|^{-\gamma}=& \sum_{\|w\|\leq c^{\ast}}\|w\|^{-\gamma}+\sum_{\|w\|>c^{\ast}}\|w\|^{-\gamma} \nonumber\\
  \leq & \sum_{\|w\|\leq c^{\ast}}\|w\|^{-\gamma}+(c^{\ast})^{-\gamma}\cdot|F_{g,z}(\mathcal{E})|. \label{margulis2024.4.44}
\end{align} 
Then we can pay attention to estimate the   summation $\sum_{\|w\|\leq c^{\ast}}\|w\|^{-\gamma}$, which explains the terminology.

\subsection{Positive dimension of additional invariance}
In the following, we shall deduce that   the first situation of the closing lemma 
(Theorem   \ref{effective2024.6.01}) provides us a skeleton $\mathcal{E}$ with bounded density for all $z\in\mathcal{E}$.

Similar to (\ref{closing2024.4.01}), we define 
\[  \mathsf{E}_{t,[r_{1},r_{2}]}(\delta)\coloneqq B_{G}(\delta)\cdot a_{t}\cdot u_{[r_{1},r_{2}]}\subset G.\]
and $\mathsf{E}_{t}=\mathsf{E}_{t,[0,1]}(e^{-\kappa t})$. Also recall  (\ref{margulis2024.4.34}). We let  
\begin{align} 
 \mathsf{Q}_{G}(\delta,\tau)  &  \coloneqq \bar{u}_{[- \frac{\delta}{\tau}, \frac{\delta}{\tau}]}\cdot a_{[-\delta,\delta]}\cdot u_{[-\delta,\delta]},\;\label{effective2022.3.24}\\
 \mathsf{Q}(y,\delta,\tau)&   \coloneqq\{h(y+w)\in\mathcal{H}_{1}(\alpha):w\in B_{H^{\perp}_{\mathbb{C}}(y)}(\delta),\ h\in \mathsf{Q}_{G}(\delta,\tau) \},\;  \label{effective2022.3.25}\\
 \mathsf{B}(y,\delta)& \coloneqq \mathsf{Q}(y,\delta,1).\;  \nonumber
\end{align}
 
\begin{prop}\label{margulis2024.3.07}
   Let the \hypertarget{2024.12.t2} notation   be as above. \hypertarget{2024.12.k15}
     Suppose that Theorem   \ref{effective2024.6.01}(1) holds. 
     Thus, suppose that  
        $\eta\in(0,N_{0}^{-1}/100)$,
    $N\geq 2N_{0}$,
        $x_{0}\in  \mathcal{H}_{1}(\alpha)$, $\kappa>0$.
  Then  there exists $t_{2}=t_{2}(\ell(x_{0}),\kappa,\eta,\alpha)>0$ such that for any   $t\geq \hyperlink{2024.12.t2}{t_{2}}$,     there exist 
                \begin{itemize}
                 \item  $\kappa_{15}=\kappa_{15}(\alpha)>0$,  
                 \item a point $y\in \mathcal{H}^{(2\eta)}_{1}(\alpha)\cap \left(\mathsf{E}_{(5\varpi +2)t,[0,1]}(e^{-\kappa t}).x_{0}\right)$,
                  \item  a finite subset $F\subset B_{H^{\perp}_{\mathbb{C}}(y)}(e^{-\kappa t})$ containing $0$ with
                      \begin{equation}\label{margulis2024.4.41}
                        e^{(1-\hyperlink{2024.12.k15}{\kappa_{15}}\kappa) t}\leq |F|\leq  e^{(\hyperlink{2024.12.k15}{\kappa_{15}}-\kappa)t},
                      \end{equation} 
                  \item a skeleton  $\mathcal{E}=\mathcal{E}(y,F,e^{-\kappa t},\textstyle{\frac{L(\eta)}{100}})$,
                \end{itemize}
    such that  
             \begin{equation}\label{margulis2024.4.19}
             \mathcal{E} \subset \mathsf{E}_{(5\varpi +2)t,[0,\frac{11}{10}]}(e^{3-\kappa t}).x_{0},\ \ \ \text{ and }\ \ \ \max_{z\in\mathcal{E}} f_{\mathcal{E}}(e,z)\leq  e^{(N+\gamma)t}.
             \end{equation}    
\end{prop} 
\begin{proof}
  For $N\geq 2N_{0}$, $x_{0}\in \mathcal{H}_{1}(\alpha)$, and $t\geq \hyperlink{2024.12.t5}{t_{5}}$, Theorem   \ref{effective2024.6.01}(1) says that there exists $x\in \mathcal{H}_{1}^{(N_{0}^{-1})}(\alpha)\cap\mathsf{E}_{(5\varpi +1)t}.x_{0}$ such that 
\begin{enumerate}[\ \ \ (a)]
          \item  $h\mapsto hx$ is injective on $\mathsf{E}_{t}$.
          \item We have
          \begin{equation}\label{margulis2024.4.05}
            f_{t,\gamma}(z)\leq e^{Nt}
          \end{equation} 
          for all $\gamma\in(0,1)$, $z\in \mathsf{E}_{t}.x$.
        \end{enumerate}
 Recall that $\mathsf{E}_{t}.x$ is a $G$-thickening of a long horocycle. Thus, one can locally observe a finite family of $G$-thickening of  horocycles.
 \begin{cla}\label{margulis2024.4.09} There exists $\hat{y}\in \mathcal{H}_{1}^{(4\eta)}(\alpha)$ such that  
 \begin{equation}\label{margulis2024.4.42}
    \Leb_{\mathsf{E}_{t}.x}\left[\mathsf{E}_{t,[0,1]}(e^{-\kappa t}-e^{10-2\kappa t}).x\cap \mathcal{H}_{1}^{(4\eta)}(\alpha)\cap \mathsf{B}(\hat{y},e^{-2\kappa t})\right]\geq  e^{-2(2g+|\Sigma|)\kappa t}
 \end{equation}
 for sufficiently large $t\geq \hyperlink{2024.12.t2}{t_{2}}(\ell(x_{0}),\kappa,\eta,\alpha)$.
 \end{cla}
 \begin{proof}[Proof of Claim \ref{margulis2024.4.09}]
   Let $\Leb_{\mathsf{E}_{t}.x}$ denote the pushforward of the Haar measure on $\mathsf{E}_{t}\subset G$. Then by Corollary \ref{closing2024.3.63}, for sufficiently large $t$ depending on $N_{0}$, we have 
       \begin{equation}\label{margulis2024.4.07}
         \Leb_{\mathsf{E}_{t}.x}\left[\mathsf{E}_{t,[0,1]}(e^{-\kappa t}-e^{10-2\kappa t}).x\cap \mathcal{H}_{1}^{(4\eta)}(\alpha)\right]\geq 1-O(\eta^{\frac{1}{2}}).
       \end{equation} 
   
     Let $t$ be large enough so that $e^{-\kappa t}<L(\eta)^{2}$. 
Let $\{\hat{y}_{k}\in \mathcal{H}_{1}^{(4\eta)}(\alpha):k\in \mathcal{K}\}$ be a maximal family  of $\textstyle{\frac{1}{10}e^{-2\kappa t}}$-separated points in $\mathcal{H}_{1}^{(4\eta)}(\alpha)$. Then by the definition, we have 
\[\mathsf{B}(\hat{y}_{i},\textstyle{\frac{1}{20}e^{-2\kappa t}})\neq \mathsf{B}(\hat{y}_{j},\textstyle{\frac{1}{20}e^{-2\kappa t}})\]
for any $i\neq j\in\mathcal{K}$. Moreover, one deduces that $\{\mathsf{B}(\hat{y}_{k},e^{-2\kappa t}):k\in \mathcal{K}\}$ is a covering of $\mathcal{H}_{1}^{(4\eta)}(\alpha)$. Since $\mu_{(1)}(\mathsf{B}(\hat{y}_{k},\textstyle{\frac{1}{20}e^{-2\kappa t}}))\asymp e^{-2(2g+|\Sigma|-2)\kappa t}$, we obtain 
\begin{equation}\label{margulis2024.4.06}
  |\mathcal{K}|\asymp e^{2(2g+|\Sigma|-2)\kappa t}.
\end{equation} 
Let 
$\mathcal{K}^{\prime}\subset   \mathcal{K}$ be the subset of indexes so that 
   \begin{equation}\label{margulis2024.3.08}
  \Leb_{\mathsf{E}_{t}.x}\left[\mathsf{E}_{t,[0,1]}(e^{-\kappa t}-e^{10-2\kappa t}).x\cap \mathcal{H}_{1}^{(4\eta)}(\alpha)\cap \mathsf{B}(\hat{y}_{k},e^{-2\kappa t})\right]\geq  e^{-2(2g+|\Sigma|)\kappa t}
   \end{equation}
   for any $k\in \mathcal{K}^{\prime}$.
   Then by (\ref{margulis2024.4.07}), (\ref{margulis2024.4.06}) and (\ref{margulis2024.3.08}), we also have  
   \[   \Leb_{\mathsf{E}_{t}.x}\left[\mathsf{E}_{t,[0,1]}(e^{-\kappa t}-e^{10-2\kappa t}).x \cap \bigcup_{k\in \mathcal{K}\setminus\mathcal{K}^{\prime}}\mathsf{B}(\hat{y}_{k},e^{-2\kappa t})\right]\leq e^{-4\kappa t}.\]    
  and so $\mathcal{K}^{\prime}\neq\emptyset$. 
 \end{proof}
 Let  $\hat{y}\in \mathcal{H}_{1}^{(4\eta)}(\alpha)$ be as in Claim \ref{margulis2024.4.09}. Then by the definition, there exist $w_{i}\in B_{H^{\perp}_{\mathbb{C}}(\hat{y})}(e^{-2\kappa t})$,   $h_{i}\in B_{G}(e^{-2\kappa t})$, and  $B_{i}\subset B_{G}(e^{3-2\kappa t})$ for $i\in\mathcal{I}$  such that  
   \begin{equation}\label{margulis2024.4.08}
   \mathcal{E}^{\ast}\coloneqq \mathsf{E}_{t,[0,1]}(e^{-\kappa t}-e^{10-2\kappa t}).x\cap \mathsf{B}(\hat{y},e^{-2\kappa t})=\bigsqcup_{i\in\mathcal{I}}B_{i}.(\hat{y}+w_{i}).
   \end{equation} 
  In particular, by repeatedly using  Lemma \ref{margulis2024.4.35}, we deduce that 
  \begin{equation}\label{margulis2024.4.10}
  \bigsqcup_{i\in\mathcal{I}}B_{G}(e^{7-2\kappa t}).(\hat{y}+w_{i}) \subset  \mathsf{E}_{t,[0,1]}(e^{-\kappa t}).x.
 \end{equation}

   We now define the skeleton $\mathcal{E}$ via a local observation.  Let $\mathcal{E}^{\ast}\subset  \mathcal{H}_{1}^{(3\eta)}(\alpha)$ be as in (\ref{margulis2024.4.08}), and let
   \[y^{\ast} =\hat{y},\ \ \ \ \ \ w_{i}^{\ast}=w_{i},\ \ \  \ \ \ \mathsf{V}=B_{G}(e^{3-2\kappa t}).\] 
  Applying Example \ref{margulis2024.4.36},  there exist  
   \begin{alignat*}{9}
 y_{i}&\in B_{i}.(\hat{y}+w_{i})\cap \mathcal{H}_{1}^{(2\eta)}(\alpha),& \ \ \ &  &y&\coloneqq y_{1}  \\
\mathsf{h}_{i}&\in B_{i},& \ \ \ &  & &    \\
F &\coloneqq \{w_{i1}\in   H^{\perp}_{\mathbb{C}}(y):i\in\mathcal{I}\},& \ \ \ &  &w_{ij}&  \in  H^{\perp}_{\mathbb{C}}(y_{j}), \\
 \mathcal{E} &\coloneqq \mathcal{E}(y,F,e^{-\kappa t},\textstyle{\frac{L(\eta)}{100}}),  & \ \ \ &  & &
\end{alignat*} 
 such that 
  \begin{equation}\label{margulis2024.3.10}
 y_{i}\in \mathcal{E}^{\ast} \ \ \ \text{ and }\ \ \  y_{i}= \mathsf{h}_{i}\mathsf{h}_{j}^{-1}(y_{j}+w_{ij})
   \end{equation}
 for any $i,j\in\mathcal{I}$.
   
   Next, we show that $\mathcal{E}$ satisfies the desired properties.  
   For $t$ large enough, by Lemma \ref{margulis2024.4.35}, we have $\mathsf{h}_{j}\mathsf{h}_{i}^{-1}\in B_{G}(e^{5-2\kappa t})$, for all $i,j\in\mathcal{I}$. Then by (\ref{margulis2024.4.10}), we have 
   \begin{equation}\label{margulis2024.3.11}
       y_{j}+w_{ij}=   \mathsf{h}_{j}\mathsf{h}_{i}^{-1} y_{i}\subset   B_{G}(e^{7-2\kappa t}).(\hat{y}+w_{i})\subset \mathsf{E}_{t,[0,1]}(e^{-\kappa t}).x
   \end{equation} 
   Then by (\ref{effective2022.2.36}), we conclude that
   \begin{equation}\label{margulis2024.4.04}
     w_{ij}\in F_{y_{j}}(t).
   \end{equation} 
Then note that by (\ref{margulis2024.4.04}) and Lemma \ref{closing2024.3.31}, we have 
\[|F|\leq |F_{y}(t)|\leq e^{-\kappa t} e^{(3\hyperlink{2024.12.k6}{\kappa_{6}}+1)t}\]
  where we use the factor $e^{-\aleph t}$ to absorb the implicit constant.

  For a lower bound of $|F|$, recall that $e^{-\kappa t}<L(\eta)^{2}$ and that 
   \[\Leb_{G}(\mathsf{E}_{t,[0,1]}(e^{-\kappa t}))\asymp e^{(1-2\kappa) t},\ \ \  \Leb_{G}(B_{G}(e^{3-2\kappa t}))\ll e^{-6\kappa t}.\] 
   Thus, for $i\in \mathcal{I}$,  we have  
   \[\Leb_{\mathsf{E}_{t}.x}(B_{i}.(\hat{y}+w_{i}))\ll e^{-6\kappa t}\cdot e^{-(1-2\kappa) t}=e^{-(1+4\kappa) t}.\]
   This, (\ref{margulis2024.4.08}) and (\ref{margulis2024.4.42}) (cf. (\ref{margulis2024.4.40})) imply that   for $t\geq \hyperlink{2024.12.t2}{t_{2}}$ sufficiently large,
    \begin{equation}\label{margulis2024.3.09}
|F|= |\mathcal{I}|\geq e^{-4\kappa t} \cdot \left[e^{-2(2g+|\Sigma|)\kappa t}/e^{-(1+4\kappa) t}\right]= e^{(1-2(2g+|\Sigma|)\kappa) t}
   \end{equation}
  where we use the factor $e^{-4\kappa t}$ to absorb the implicit constant. Thus, let $\hyperlink{2024.12.k15}{\kappa_{15}}=\max\{(3\hyperlink{2024.12.k6}{\kappa_{6}}+1),2(2g+|\Sigma|)\}$. 
 We obtain (\ref{margulis2024.4.41}).
 
 Next, one observes that 
\begin{cla}\label{margulis2024.3.13}
 $\mathcal{E}$ is contained in a long horocycle:
   \[ \mathcal{E}\subset \mathsf{E}_{(5\varpi +2)t,[0,\frac{11}{10}]}(e^{3-\kappa t}).x_{0}.\]
\end{cla}
\begin{proof}[Proof of Claim \ref{margulis2024.3.13}] By (\ref{margulis2024.3.11}) and $x\in  \mathsf{E}_{(5\varpi +1)t}.x_{0}$, we conclude that  
 \begin{align}
   \mathcal{E}&= \mathcal{E}(y,F,e^{-\kappa t},\textstyle{\frac{L(\eta)}{100}})\;\nonumber\\
  & =\bigsqcup_{w_{i1}\in F} B_{G}(e^{-\kappa t})u_{[-\frac{L(\eta)}{100},\frac{L(\eta)}{100}]}(y_{1}+w_{i1})\;\nonumber\\
      &\subset B_{G}(e^{-\kappa t})u_{[-\frac{L(\eta)}{100},\frac{L(\eta)}{100}]}  \cdot  \mathsf{E}_{t,[0,1]}(e^{-\kappa t}).x\;\nonumber\\
      &\subset  B_{G}(e^{-\kappa t})u_{[-\frac{L(\eta)}{100},\frac{L(\eta)}{100}]}  \cdot  \mathsf{E}_{t,[0,1]}(e^{-\kappa t})\cdot\mathsf{E}_{(5\varpi +1)t,[0,1]}(e^{-\kappa t}).x_{0}\;\nonumber\\
      &= B_{G}(e^{-\kappa t})u_{[-\frac{L(\eta)}{100},\frac{L(\eta)}{100}]}  \cdot B_{G}(e^{-\kappa t}) a_{t} u_{[0,1]} \cdot B_{G}(e^{-\kappa t}) a_{(5\varpi +1)t} u_{[0,1]}.x_{0}\;\nonumber\\
      &\subset  B_{G}(e^{3-\kappa t})\cdot a_{(5\varpi +2)t}\cdot u_{[0,\frac{11}{10}]}.x_{0}= \mathsf{E}_{(5\varpi +2)t,[0,\frac{11}{10}]}(e^{3-\kappa t}).x_{0}.\;  \nonumber
\end{align} 
This establishes the claim.
\end{proof}

Finally, we consider the density:
\begin{cla}\label{margulis2024.3.14}
 For any $z\in \mathcal{E}$, we have 
   \[ f_{\mathcal{E}}(e,z)\leq   e^{(N+\gamma)t}.\]
\end{cla}
\begin{proof}[Proof of Claim \ref{margulis2024.3.14}]  
Let $w\in F_{e,z}(\mathcal{E})$. Then $z,z+w\in\mathcal{E}$. By the definition of $\mathcal{E}$, there are $i,j$ so that 
\[z=g_{i}(y+w_{i1}),\ \ \ z+w=g_{j}(y+w_{j1})\]
for some $w_{i1},w_{j1}\in F$, and $g_{i},g_{j}\in B_{G}(\frac{L(\eta)}{10})$. It forces $g_{i}=g_{j}$ via (\ref{margulis2024.4.52}), and so by (\ref{margulis2024.3.10}) and (\ref{margulis2024.4.04}) (see also (\ref{margulis2024.4.38})), we get
\[w=g_{i}(w_{j1}-w_{i1})=g_{i}\mathsf{h}_{1}\mathsf{h}_{i}^{-1} w_{ji}\in g_{i}\mathsf{h}_{1}\mathsf{h}_{i}^{-1}\cdot F_{y_{i}}(t).\] 
It follows immediately  that $\|w_{ji}\|\leq 2\|w\|$. In addition, one may deduce that the map $F_{e,z}(\mathcal{E})\rightarrow F_{y_{i}}(t)$ given by $w\mapsto w_{ji}$ is injective.
Then by (\ref{margulis2024.4.05}), we have 
\[f_{\mathcal{E}}(e,z)=\sum_{w\in F_{e,z}(\mathcal{E})}\|w\|^{-\gamma}\leq 2^{\gamma}\sum_{v\in F_{y_{i}}(t)}\|v\|^{-\gamma} =2^{\gamma} f_{t,\gamma}(y_{i})\leq e^{(N+\gamma)t}\]
 The consequence follows.
\end{proof}
 This establishes the Proposition \ref{margulis2024.3.07}.  
\end{proof}
\subsection{Improving the dimension of additional invariance}
Next, we develop the Margulis function techniques and improve the dimension of additional invariance. See \cite{eskin2022margulis} for more details of Margulis functions.   

In view of Proposition \ref{margulis2024.3.07}, we focus on the skeletons of the form:  
\begin{align} 
 &  y \in\mathcal{H}_{1}(\alpha) ,\nonumber\\
 &   F\subset B_{H^{\perp}_{\mathbb{C}}(y)}(e^{-\kappa t}) \text{ a finite subset,}\nonumber\\
   &   \mathcal{E}=\mathcal{E}(y,F,e^{-\kappa t},\textstyle{\frac{L(\eta)}{10}})\subset\mathcal{H}_{1}^{(\eta)}(\alpha) \label{margulis2024.4.17}
\end{align}  
where  $e^{-\kappa t}<L(\eta)^{2}$, and $\mathcal{E}$ is defined as in  Example \ref{margulis2024.3.12}.

By Lemma \ref{margulis2024.3.03}, for $x\in \mathcal{H}^{(N_{0}^{-1})}_{1}(\alpha)$, there exists $m_{\gamma}=m_{\gamma}(N_{0})>0$ so that 
\begin{equation}\label{margulis2024.4.11}
  \int_{0}^{1}\|a_{m_{\gamma}}u_{r}w\|_{a_{m_{\gamma}}u_{r}x}^{-\gamma}dr\leq e^{-1}\|w\|_{x}^{-\gamma}.
\end{equation} 
Let $\nu=\nu(\gamma)$ be the probability measure on $G$ defined by
   \begin{equation}\label{effective2022.3.2}
      \nu(\varphi)=\int_{0}^{1}\varphi(a_{m_{\gamma}}u_{r})dr
   \end{equation}
   for all $\varphi\in C_{c}(G)$. Let $\nu^{(k)}$ be the $k$-fold convolution  of $\nu$. Then we  obtain a random walk on $G$.  
In the following, we  are going to show that the local density functions are Margulis functions with respect to the random walk $\nu^{(k)}$.
 
 \begin{lem}\label{effective2024.6.09}
     For $n\in\mathbb{N}$, $h\in\supp(\nu^{(n)})$, we have  \begin{equation}\label{effective2022.3.20}
     \Leb_{\mathcal{E}}\{z\in\mathcal{E}:hz\not\in  \mathcal{H}_{1}^{(2\eta)}(\alpha)\}\ll \eta^{\hyperlink{2024.12.k3}{\kappa_{3}}}.
   \end{equation}
 \end{lem} 
\begin{proof} 
   Write $h^{\prime}=a_{n m_{\gamma}} u_{\hat{r}}$  where $\hat{r}=\sum_{j=0}^{n-1}e^{-2jm_{\gamma}}r_{j+1}$ for some $r_{1},\ldots, r_{n}\in[0,1]$.  Note that by (\ref{margulis2024.4.25}), $k\geq \hyperlink{2024.12.k4}{\kappa_{4}}|\log \ell(y+w)|+|\log \frac{L(\eta)}{10}|+\hyperlink{2024.3.C6}{C_{6}}$. We may apply   Corollary \ref{closing2024.3.63}  with $y+w\in\mathcal{E}\subset \mathcal{H}_{1}^{(\eta)}(\alpha)$ and  $I=\hat{r}+[-\textstyle{\frac{L(\eta)}{10}}, \textstyle{\frac{L(\eta)}{10}}]$. Then we have 
   \begin{equation}\label{effective2022.3.22}
    \left|\left\{ r\in [-\textstyle{\frac{L(\eta)}{10}}, \textstyle{\frac{L(\eta)}{10}}]: h u_{r}(y+w)\not\in \mathcal{H}_{1}^{(4\eta)}(\alpha)\right\}\right|\leq \hyperlink{2024.3.C6}{C_{6}}(4\eta)^{\hyperlink{2024.12.k3}{\kappa_{3}}}\cdot \textstyle{\frac{L(\eta)}{5}}.
   \end{equation}
  Now (\ref{effective2022.3.21}), (\ref{effective2022.3.22}), and the definition of $\Leb_{\mathcal{E}}$ (\ref{effective2022.2.39}) imply that
\[   \Leb_{\mathcal{E}}\{z\in\mathcal{E}:hz\not\in  \mathcal{H}_{1}^{(2\eta)}(\alpha)\}\ll \eta^{\hyperlink{2024.12.k3}{\kappa_{3}}}\]
   for  $h\in\supp(\nu^{(k)})$. 
\end{proof}

 \begin{prop}[Margulis function]\label{margulis2024.4.16} Let   $\mathcal{E}$ be a skeleton as in (\ref{margulis2024.4.17}). \hypertarget{2024.12.k17} Then the local density function $f_{\mathcal{E}}$ is \hypertarget{2024.12.k16} a \textit{Margulis function}\index{Margulis function} with respect to the random walk $\nu^{(k)}$.   
More precisely, for     $\kappa>0$, $t>\hyperlink{2024.12.t2}{t_{2}}$, there exists   $\kappa_{16}= \kappa_{16}(\alpha,\gamma)>0$, $\kappa_{17} =\kappa_{17}(\alpha,\gamma
,\kappa)>0$ such that $\hyperlink{2024.12.k17}{\kappa_{17}}\rightarrow0$ as $\kappa\rightarrow0$, and that  
\begin{equation}\label{margulis2024.4.18}
  \iint f_{\mathcal{E}}(h,z)d\Leb_{\mathcal{E}}(z) d\nu^{(k)}(h)\leq e^{-k}\int f_{\mathcal{E}}(e,z)d\Leb_{\mathcal{E}}(z)+  e^{\hyperlink{2024.12.k17}{\kappa_{17}} t+\hyperlink{2024.12.k16}{\kappa_{16}}k}| F|
\end{equation}   
  for all $k\in\mathbb{N}$.
 \end{prop} 
\begin{proof} Note first that by assumption  $\mathcal{E}\subset \mathcal{H}_{1}^{(\eta)}(\alpha)$. Note that $\supp(\nu)\subset B_{G}(e^{2m_{\gamma}+1})$. In view of Theorem \ref{effective2023.11.3}, let $C=C(\gamma)\geq 1$ be so that
   \[\| h.w\|_{hz}\leq C\|w\|_{z}\ \ \ \text{ and }\ \ \ C^{-1}L(z)\leq L(hz)\leq CL(z)\]
   for all $h\in B_{G}(e^{2m_{\gamma}+1})$, $w\in H^{1}(M,\Sigma;\mathbb{C})$,  and $z\in \mathcal{H}_{1}(\alpha)$.

   Let $h=a_{m_{\gamma}}u_{r}$ for some $r\in[0,1]$. Let $z\in\mathcal{E}$, $h^{\prime}\in G$. First, let us assume that there exists some $w\in F_{hh^{\prime},z}(\mathcal{E})$ with $\|w\|_{hh^{\prime}z}<C^{-2}L(hh^{\prime}z)$. In view of the choice of $C$, this in particular implies that $v=h^{-1}w\in F_{ h^{\prime},z}(\mathcal{E})$. Hence, writing $c^{\ast}=C^{-2}L(hh^{\prime}z)$, by (\ref{margulis2024.4.44}), we have
   \begin{align}
f_{\mathcal{E}}(hh^{\prime},z) & \leq   \sum_{\|w\|_{hh^{\prime}z}<C^{-2}L (hh^{\prime}z)}\|w\|_{hh^{\prime}z}^{-\gamma}+C^{2\gamma}\cdot L (hh^{\prime}z)^{-\gamma}\cdot| F_{hh^{\prime},z}(\mathcal{E})| \;\nonumber\\
& \leq   \sum_{v\in F_{h^{\prime},z}(\mathcal{E})}\|  h.v\|_{hh^{\prime}z}^{-\gamma}+C^{2\gamma}\cdot L (hh^{\prime}z)^{-\gamma}\cdot| F_{hh^{\prime},z}(\mathcal{E})|.\;  \label{effective2024.4.1}
\end{align} 

If $F_{hh^{\prime},z}(\mathcal{E})=\emptyset$, then we obviously have
\begin{equation}\label{effective2024.4.2}
  f_{\mathcal{E}}(hh^{\prime},z) \leq   C^{2\gamma}\cdot L (hh^{\prime}z)^{-\gamma}.
\end{equation}

We now average (\ref{effective2024.4.1}) (\ref{effective2024.4.2}) over $r\in[0,1]$ and $z\in\mathcal{E}$ and conclude that
\begin{multline}
   \int_{0}^{1}f_{\mathcal{E}}(a_{m_{\gamma}}u_{r}h^{\prime},z)dr \leq   \sum_{w\in F_{h^{\prime},z}( \mathcal{E})}\int_{0}^{1}\|a_{m_{\gamma}}u_{r}w\|_{a_{m_{\gamma}}u_{r}h^{\prime}z}^{-\gamma}dr \\
  +C^{2\gamma}\int_{0}^{1}  (1+| F_{h^{\prime},z}(\mathcal{E})|)\cdot L (a_{m_{\gamma}}u_{r}h^{\prime}z)^{-\gamma} dr.\nonumber
\end{multline} 
Recall that, by (\ref{margulis2024.4.11}), Lemma \ref{margulis2024.3.03} and Lemma \ref{effective2024.6.09}, we may conclude that 
 \begin{align}
 & \int\sum_{w\in F_{h^{\prime},z}( \mathcal{E})}\int_{0}^{1}\|a_{m_{\gamma}}u_{r}w\|_{a_{m_{\gamma}}u_{r}h^{\prime}z}^{-\gamma}dr d\Leb_{\mathcal{E}}(z)\;\nonumber\\
=  &  \int_{h^{\prime}z\not\in  \mathcal{H}_{1}^{(2\eta)}(\alpha)}\sum_{w\in F_{h^{\prime},z}( \mathcal{E})}\int_{0}^{1}\|a_{m_{\gamma}}u_{r}w\|_{a_{m_{\gamma}}u_{r}h^{\prime}z}^{-\gamma}dr d\Leb_{\mathcal{E}}(z) \;\nonumber\\
  & + \int_{h^{\prime}z\in  \mathcal{H}_{1}^{(2\eta)}(\alpha)}\sum_{w\in F_{h^{\prime},z}( \mathcal{E})}\int_{0}^{1}\|a_{m_{\gamma}}u_{r}w\|_{a_{m_{\gamma}}u_{r}h^{\prime}z}^{-\gamma}dr d\Leb_{\mathcal{E}}(z) \;\nonumber\\
  \leq  &  \int_{h^{\prime}z\not\in  \mathcal{H}_{1}^{(2\eta)}(\alpha)} c(\gamma)\cdot f_{\mathcal{E}}(h^{\prime},z)  d\Leb_{\mathcal{E}}(z)   + e^{-2}\cdot f_{\mathcal{E}}(h^{\prime},z) \;\nonumber\\
    \leq  &  \frac{c(\gamma)}{100}\cdot\int f_{\mathcal{E}}(h^{\prime},z)   d\Leb_{\mathcal{E}}(z) + e^{-2}\cdot \int f_{\mathcal{E}}(h^{\prime},z)   d\Leb_{\mathcal{E}}(z)  \;\nonumber\\ 
    \leq & e^{-1}\cdot \int f_{\mathcal{E}}(h^{\prime},z)  d\Leb_{\mathcal{E}}(z).\;  \nonumber
\end{align} 
It follows that 
\begin{multline}
   \iint f_{\mathcal{E}}(hh^{\prime},z)d\nu(h) d\Leb_{\mathcal{E}}(z)  \leq e^{-1}\cdot \int f_{\mathcal{E}}(h^{\prime},z)  d\Leb_{\mathcal{E}}(z)\\
+C^{2\gamma}\iint (1+| F_{hh^{\prime},z}(\mathcal{E})|)\cdot L (hh^{\prime}z)^{-\gamma} d\nu(h)  d\Leb_{\mathcal{E}}(z)\nonumber
\end{multline}  
for all $h^{\prime}\in H$. Iterating this estimate, we have 
 \begin{multline}\label{margulis2024.4.12}
   \iint f_{\mathcal{E}}(h,z)d\nu^{(k)}(h)d\Leb_{\mathcal{E}}(z)\leq e^{-k}\int f_{\mathcal{E}}(e,z) d\Leb_{\mathcal{E}}(z) \\
   +C^{2\gamma}\sum_{j=1}^{k}e^{j-k}\iint (1+| F_{h,z}(\mathcal{E})|)\cdot L (hz)^{-\gamma}d\nu^{(j)}(h)d\Leb_{\mathcal{E}}(z). 
 \end{multline}
 
 Now we estimate $| F_{h,z}(\mathcal{E})|$. Let $\mathcal{E}^{+}\coloneqq\mathcal{E}(y,F,e^{1-\kappa t},\textstyle{\frac{L(\eta)}{10}})$, $a_{jm_{\gamma}}u_{r}\in\supp(\nu^{(j)})$.   
 In view of (\ref{effective2022.3.24}) and Lemma \ref{margulis2024.4.35}, we have
  \[\mathsf{Q}_{G}(e^{-2\kappa t},e^{jm_{\gamma}})\cdot (a_{jm_{\gamma}}u_{r}z+w)\subset a_{jm_{\gamma}}u_{r}\mathcal{E}^{+}.\] 
   Note that for sufficiently large $t$, $\mathcal{E}^{+}\subset \mathcal{H}_{1}^{(\eta)}(\alpha)$. For $w\in F_{a_{jm_{\gamma}}u_{r},z}(\mathcal{E})$ and sufficiently large $t$, we have  
    \begin{multline}
  (a_{jm_{\gamma}}u_{r})_{\ast}\Leb_{\mathcal{E}^{+}}\left(\mathsf{Q}_{G}(e^{-2\kappa t},e^{jm_{\gamma}}) \cdot(a_{jm_{\gamma}}u_{r}z+w)\right)\\
    =\frac{\Leb_{G}(\mathsf{Q}_{G}(e^{-2\kappa t},e^{jm_{\gamma}}))}{|F| \Leb_{G}(\mathsf{E}_{[-\frac{L(\eta)}{10},\frac{L(\eta)}{10}]}(e^{1-\kappa t}))}\gg\frac{(  e^{-2\kappa t})^{3}e^{-jm_{\gamma}}}{|F|}\nonumber
\end{multline} 
    where the implied constant is absolute.  Since $\mathcal{E}\subset \mathcal{E}^{+}$, we have  
    \begin{equation}\label{margulis2024.4.13}
      |F_{a_{jm_{\gamma}}u_{r},z}(\mathcal{E})|\leq |F_{a_{jm_{\gamma}}u_{r},z}(\mathcal{E}^{+})|\ll  e^{6\kappa t+jm_{\gamma}}\cdot| F|\leq  e^{7\kappa t+jm_{\gamma}}\cdot| F|
    \end{equation} 
       where we use the factor $e^{\kappa t}$ to absorb the implicit constant.  
       
 Finally, by (\ref{margulis2024.4.14}), we have
 \begin{equation}\label{margulis2024.4.15}
   L (a_{jm_{\gamma}}u_{r}z)^{-\gamma}\ll \ell(a_{jm_{\gamma}}u_{r}z)^{-\gamma(\hyperlink{2024.12.k6}{\kappa_{6}}+\hyperlink{2024.12.k5}{\kappa_{5}})}\ll (e^{-jm_{\gamma}}\eta)^{-\gamma(\hyperlink{2024.12.k6}{\kappa_{6}}+\hyperlink{2024.12.k5}{\kappa_{5}})}.
 \end{equation} 
 Write $c=\gamma(\hyperlink{2024.12.k6}{\kappa_{6}}+\hyperlink{2024.12.k5}{\kappa_{5}})$, and recall that $e^{-\kappa t}\leq L(\eta)^{2}$. Combining (\ref{margulis2024.4.12}), (\ref{margulis2024.4.13}) and (\ref{margulis2024.4.15}), we obtain  
    \begin{align}
 &  C^{2\gamma}\sum_{j=1}^{k}e^{j-k}\iint (1+| F_{h,z}(\mathcal{E})|)\cdot L (hz)^{-\gamma}d\nu^{(j)}(h)d\Leb_{\mathcal{E}}(z)\;\nonumber\\
 \ll &  C^{2\gamma}\sum_{j=1}^{k}e^{j-k}  \cdot e^{7\kappa t+jm_{\gamma}}| F| \cdot e^{jm_{\gamma}c}\eta^{-c}\;\nonumber\\
\ll &   C^{2\gamma}e^{c\kappa t} e^{7\kappa t}  | F|\sum_{j=1}^{k} e^{jm_{\gamma}(1+c)+j-k}   \;\nonumber\\
\leq &    e^{(c+8)\kappa t+2km_{\gamma}(1+c)}| F| \;  \nonumber
\end{align}
where we use the factor $e^{\kappa t}$ to absorb the implicit constant.  Let $\hyperlink{2024.12.k17}{\kappa_{17}} =(2c+8)\kappa$, and $\hyperlink{2024.12.k16}{\kappa_{16}}=2m_{\gamma}(1+c)$. Then we establish the claim.
\end{proof}
Now suppose that   $M>1$,     $\mathcal{E}$ is a skeleton as in (\ref{margulis2024.4.17}) and
\begin{equation}\label{margulis2024.4.22}
\max_{z\in\mathcal{E}} f_{\mathcal{E}}(e,z)\leq  e^{Mt}.
\end{equation}
Then by Proposition \ref{margulis2024.4.16}, we see that  for any $k\in\mathbb{N}$ at least one of the following holds:  
\begin{equation}\label{margulis2024.4.20}
\max_{z\in\mathcal{E}} f_{\mathcal{E}}(e,z)\leq e^{Mt} \leq e^{\hyperlink{2024.12.k17}{\kappa_{17}} t+(\hyperlink{2024.12.k16}{\kappa_{16}}+1)k}| F|,
\end{equation}
\begin{equation}\label{margulis2024.4.21}
\iint  f_{\mathcal{E}}(h,z)d\Leb_{\mathcal{E}}(z) d\nu^{(k)}(h)\leq 2e^{Mt-k}\leq e^{Mt-k+1}.
\end{equation}
The case (\ref{margulis2024.4.20}) is desired.  
 In the following, we  assume that case (\ref{margulis2024.4.21}) occurs, and manage to create a new skeleton  so that  (\ref{margulis2024.4.20}) is more likely to occur.

First, we shall convert (\ref{margulis2024.4.21}) into a pointwise version.  
\begin{lem}[Pointwise version of (\ref{margulis2024.4.21})]\label{margulis2024.4.29}  
   Let  $\mathcal{E}$ be a skeleton defined in (\ref{margulis2024.4.17}). Suppose that   for $M>1$, $\kappa>0$, $t> \hyperlink{2024.12.t2}{t_{2}}$ and 
   \begin{equation}\label{margulis2024.4.25}
      200 \hyperlink{2024.12.k4}{\kappa_{4}}|\log L(\eta)|+\hyperlink{2024.3.C6}{C_{6}}\leq k,
   \end{equation}
    we \hypertarget{2024.12.k18} have (\ref{margulis2024.4.22}) and  (\ref{margulis2024.4.21}).  
   Then there exists $\kappa_{18}=\kappa_{18}(\alpha)>0$,  and a subset $G_{\mathcal{E}}=G_{\mathcal{E}}(k)\subset\supp(\nu^{(k)})$ with
   \begin{equation}\label{margulis2024.4.24}
    \nu^{(k)}(G_{\mathcal{E}})\geq 1-e^{1-\frac{1}{8}k}
   \end{equation} 
   such that  for all $h\in G_{\mathcal{E}}$, there exists a sub-skeleton $\mathcal{E}(h)\subset\mathcal{E}$ with 
   \[\Leb_{\mathcal{E}}(\mathcal{E}(h))\geq 1-O(\eta^{\hyperlink{2024.12.k18}{\kappa_{18}}}),\] 
   so that for all $z\in \mathcal{E}(h)$, we have
         \begin{align}
B_{G}(e^{10-2\kappa t}).z & \subset \mathcal{E}\;\label{effective2022.3.14}\\
hz  & \in  \mathcal{H}_{1}^{(2\eta)}(\alpha)\;\label{margulis2024.4.27}\\
f_{\mathcal{E}}(h,z)  & \leq e^{Mt-\frac{3}{4}k}. \;  \label{margulis2024.4.47}
\end{align} 
\end{lem}
\begin{proof}
  By (\ref{margulis2024.4.21}), we have
   \[\iint  f_{\mathcal{E}}(h,z)d\Leb_{\mathcal{E}}(z) d\nu^{(k)}(h) \leq e^{Mt-k+1}.\]
   Using this estimate and \textit{Chebyshev's inequality}\index{Chebyshev's inequality}, we have
   \begin{equation}\label{effective2022.3.17}
      \nu^{(k)}\left\{h\in\supp(\nu^{(k)}):\int f(h,z)d\Leb_{\mathcal{E}}(z)> e^{Mt-\frac{7}{8}k}\right\}<e^{1-\frac{1}{8}k}.
   \end{equation}
   Now let  
   \begin{equation}\label{margulis2024.4.23}
    G_{\mathcal{E}}\coloneqq \left\{h\in\supp(\nu^{(k)}):\int f(h,z)d\Leb_{\mathcal{E}}(z)\leq e^{Mt-\frac{7}{8}k}\right\}.
   \end{equation}
   Then, (\ref{effective2022.3.17}) implies (\ref{margulis2024.4.24}).

 Let $h\in\supp(\nu^{(k)})$. Then for any $z=g_{1}u_{r}(y+w)\in\mathcal{E}$, where $g_{1}=\bar{u}_{s_{1}}a_{s_{2}}$ with $s_{1},s_{2}\in[-e^{-\kappa t},e^{-\kappa t}]$, we have
   \[hz=hg_{1}u_{r}(y+w)=g_{2}hu_{r}(y+w)\]
   where $g_{2}\in B_{G}(e^{-\kappa t})\subset B_{G}(\eta^{2})$. It follows that 
   \begin{equation}\label{effective2022.3.21}
 \ell(hu_{r}(y+w))\geq 4\eta\ \ \  \Longrightarrow \ \ \  \ell(hz) \geq 2\eta.
   \end{equation}
 Write $h=a_{k m_{\gamma}} u_{\hat{r}}$  where $\hat{r}=\sum_{j=0}^{k-1}e^{-2jm_{\gamma}}r_{j+1}$ for some $r_{1},\ldots, r_{k}\in[0,1]$.  Note that by (\ref{margulis2024.4.25}), $k\geq \hyperlink{2024.12.k4}{\kappa_{4}}|\log \ell(y+w)|+|\log \frac{L(\eta)}{10}|+\hyperlink{2024.3.C6}{C_{6}}$. We may apply   Corollary \ref{closing2024.3.63}  with $y+w\in\mathcal{E}\subset \mathcal{H}_{1}^{(\eta)}(\alpha)$ and  $I=\hat{r}+[-\textstyle{\frac{L(\eta)}{10}}, \textstyle{\frac{L(\eta)}{10}}]$. Then we have 
   \begin{equation}\label{effective2022.3.22}
    \left|\left\{ r\in [-\textstyle{\frac{L(\eta)}{10}}, \textstyle{\frac{L(\eta)}{10}}]: h u_{r}(y+w)\not\in \mathcal{H}_{1}^{(4\eta)}(\alpha)\right\}\right|\leq \hyperlink{2024.3.C6}{C_{6}}(4\eta)^{\hyperlink{2024.12.k3}{\kappa_{3}}}\cdot \textstyle{\frac{L(\eta)}{5}}.
   \end{equation}
  Now (\ref{effective2022.3.21}), (\ref{effective2022.3.22}), and the definition of $\Leb_{\mathcal{E}}$ (\ref{effective2022.2.39}) imply that
   \begin{equation}\label{effective2022.3.20}
     \Leb_{\mathcal{E}}\{z\in\mathcal{E}:hz\not\in  \mathcal{H}_{1}^{(2\eta)}(\alpha)\}\ll \eta^{\hyperlink{2024.12.k3}{\kappa_{3}}}
   \end{equation}
   for  $h\in\supp(\nu^{(k)})$. 
   
   Now set $\mathcal{E}^{-}\coloneqq \mathcal{E}(y,F,e^{-\kappa t}-e^{20-2\kappa t},\textstyle{\frac{L(\eta)}{10}})\subset\mathcal{E}$. Clearly, $\Leb_{\mathcal{E}}(\mathcal{E}^{-})\geq 1-O(e^{-\kappa t})$. Moreover, for $z\in \mathcal{E}^{-}$, we have
    \[B_{G}(e^{10-2\kappa t}).z\subset\mathcal{E}.\] 
    Next, let  $h\in G_{\mathcal{E}}$ and $\mathcal{E}^{\prime}(h)\coloneqq \mathcal{E}^{-}\cap\{z\in\mathcal{E}:hz\not\in  \mathcal{H}_{1}^{(2\eta)}(\alpha)\}$.
    Then (\ref{effective2022.3.20})  implies that $\Leb_{\mathcal{E}}(\mathcal{E}^{\prime}(h))\geq 1-O(\eta^{\hyperlink{2024.12.k3}{\kappa_{3}}}) $. Moreover, for   $z\in \mathcal{E}^{\prime}(h)$, we have 
    \[hz\in \mathcal{H}_{1}^{(2\eta)}(\alpha).\]  
Finally, let  $\mathcal{E}_{c}\coloneqq\{z\in\mathcal{E}^{\prime}(h): f_{\mathcal{E}}(h,z)>e^{Mt-\frac{3}{4}k}\}$.
   Then by (\ref{margulis2024.4.23}), we have
   \[ \Leb_{\mathcal{E}}(\mathcal{E}_{c})e^{Mt-\frac{3}{4}k}\leq   \int_{\mathcal{E}_{c}}f(h,z)d\Leb_{\mathcal{E}}(z) \leq    \int_{\mathcal{E}}f(h,z)d\Leb_{\mathcal{E}}(z)   \leq e^{Mt-\frac{7}{8}k}.\] 
We conclude from the above that $\Leb_{\mathcal{E}}(\mathcal{E}_{c})\ll e^{-\frac{1}{8}k}$. Thus, by (\ref{margulis2024.4.25}), we conclude that $\Leb_{\mathcal{E}}(\mathcal{E}_{c})\ll\eta^{25 \hyperlink{2024.12.k4}{\kappa_{4}}}$.
Put $\mathcal{E}(h)\coloneqq \mathcal{E}^{\prime}(h)\setminus \mathcal{E}_{c}$. Then let $\hyperlink{2024.12.k18}{\kappa_{18}}=\min\{25 \hyperlink{2024.12.k4}{\kappa_{4}},2,\hyperlink{2024.12.k3}{\kappa_{3}}\}$. We obtain $\Leb_{\mathcal{E}}(\mathcal{E}(h))\geq 1-O(\eta^{\hyperlink{2024.12.k18}{\kappa_{18}}})$ and 
\[f_{\mathcal{E}}(h,z)<e^{Mt-\frac{3}{4}k}\]
  for every $z\in\mathcal{E}(h)$. We establish the claim. 
\end{proof}

The following lemma gives an effective estimate of the recurrence of $h\mathcal{E}(h)$. 
\begin{lem}[{Recurrence of $h\mathcal{E}(h)$}]\label{margulis2024.4.33}
     Let the notation and assumptions be as in Lemma \ref{margulis2024.4.29}. Then there \hypertarget{2024.12.k19} exists $y(h)\in \mathcal{H}_{1}^{(\eta)}(\alpha)$, and $\kappa_{19}=\kappa_{19}(\alpha)>0$ such that
   \[h_{\ast}\Leb_{\mathcal{E}}(h\mathcal{E}(h)\cap\mathsf{Q}(y(h), e^{-2\kappa t},e^{km_{\gamma}}))\geq e^{-\hyperlink{2024.12.k19}{\kappa_{19}}\kappa t}e^{-km_{\gamma}} .\]
\end{lem}
\begin{proof}
Let $\{\hat{y}_{k}\in \mathcal{H}_{1}^{(2\eta)}(\alpha):k\in \mathcal{K}\}$ be a maximal family  of $\textstyle{\frac{1}{10}e^{-2\kappa t}}$-separated points in $\mathcal{H}_{1}^{(2\eta)}(\alpha)$. Then by the definition, we have 
\[\mathsf{B}(\hat{y}_{i},\textstyle{\frac{1}{20}e^{-2\kappa t}})\neq \mathsf{B}(\hat{y}_{j},\textstyle{\frac{1}{20}e^{-2\kappa t}}).\]
for any $i\neq j\in\mathcal{K}$. Moreover, one deduces that $\{\mathsf{B}(\hat{y}_{k},\textstyle{\frac{1}{2}e^{-2\kappa t}}):k\in \mathcal{K}\}$ is a covering of $\mathcal{H}_{1}^{(2\eta)}(\alpha)$. Since $\mu_{(1)}(\mathsf{B}(\hat{y}_{k},\textstyle{\frac{1}{20}e^{-2\kappa t}}))\asymp e^{-2(2g+|\Sigma|-2)\kappa t}$, we obtain 
\[ |\mathcal{K}|\asymp e^{2(2g+|\Sigma|-2)\kappa t}.\]

Next, for a ball $B_{G}(\textstyle{\frac{1}{2}e^{-2\kappa t}})\subset G$, let    $\{h_{s}\in B_{G}(\textstyle{\frac{1}{2}e^{-2\kappa t}}):s\in \mathcal{S}\}$ be a maximal family  of  points such that
\[\mathsf{Q}_{G}(\textstyle{\frac{1}{100}e^{-2\kappa t}},e^{km_{\gamma}})   h_{s_{1}}\cap \mathsf{Q}_{G}(\textstyle{\frac{1}{100}e^{-2\kappa t}},e^{km_{\gamma}})  h_{s_{2}}=\emptyset\]
for  any $s_{1}\neq s_{2}\in\mathcal{S}$. Moreover, one deduces that $\{\mathsf{Q}_{G}(\textstyle{\frac{1}{2}e^{-2\kappa t}},e^{km_{\gamma}})   h_{s}:s\in \mathcal{S}\}$ is a covering of $B_{G}(\textstyle{\frac{1}{2}e^{-2\kappa t}})$. Since $\Leb_{G}(\mathsf{Q}_{G}(\textstyle{\frac{1}{100}e^{-2\kappa t}},e^{km_{\gamma}})h_{s})\asymp (e^{-2\kappa t})^{3}(e^{km_{\gamma}})^{-1}$, and    Lemma \ref{margulis2024.4.35}, we obtain 
\[ |\mathcal{S}|\asymp e^{km_{\gamma}}.\]

   Combining these two coverings, we obtain a covering $\{\mathsf{Q}(h_{s}\hat{y}_{k}, e^{-2\kappa t},e^{km_{\gamma}}):s \in \mathcal{S},\ k\in \mathcal{K}\}$
   of $\mathcal{H}_{1}^{(2\eta)}(\alpha)$.   For any $s \in \mathcal{S}$,  $k\in \mathcal{K}$, we write $y_{j}=h_{s}\hat{y}_{k}$ and $\mathcal{J}$ for the  set of indexes. Hence, we obtain  a covering 
   \begin{equation}\label{margulis2024.4.28}
     \{\mathsf{Q}(y_{j}, e^{-2\kappa t},e^{km_{\gamma}}):j\in \mathcal{J}\}
   \end{equation} 
 of $ \mathcal{H}_{1}^{(2\eta)}(\alpha)$ where 
\begin{equation}\label{margulis2024.4.26}
  \#\mathcal{J}\ll e^{2(2g+|\Sigma|-2)\kappa t}e^{km_{\gamma}}.
\end{equation} 
In particular, we have $y_{j}\in\mathcal{H}_{1}^{(\eta)}(\alpha)$.
 
 Let $h\in L_{\mathcal{E}}$, and define
  \[ \mathcal{J}_{c}(h)\coloneqq\left\{j\in\mathcal{J}:h_{\ast}\Leb_{\mathcal{E}}(h\mathcal{E}(h)\cap\mathsf{Q}(y_{j}, e^{-2\kappa t},e^{km_{\gamma}}))< e^{-3(2g+|\Sigma|-2)\kappa t}e^{-km_{\gamma}}\right\}.\]
   Then by (\ref{margulis2024.4.26}), we have 
 \[h_{\ast}\Leb_{\mathcal{E}}\left[h\mathcal{E}(h)\cap\bigcup_{j\in \mathcal{J}_{c}(h)}\mathsf{Q}(y_{j}, e^{-2\kappa t},e^{km_{\gamma}})\right]<e^{-(2g+|\Sigma|-2)\kappa t}<1.\]
 Then by  (\ref{margulis2024.4.27}) and (\ref{margulis2024.4.28}), we conclude that there exists $j\in \mathcal{J}\setminus \mathcal{J}_{c}(h)$, i.e. 
 \[h_{\ast}\Leb_{\mathcal{E}}(h\mathcal{E}(h)\cap\mathsf{Q}(y_{j}, e^{-2\kappa t},e^{km_{\gamma}}))\geq e^{-3(2g+|\Sigma|-2)\kappa t}e^{-km_{\gamma}} .\]
 Let $y(h)\coloneqq y_{j}$, $\hyperlink{2024.12.k19}{\kappa_{19}}\coloneqq 3(2g+|\Sigma|-2)$.  We establish the claim.
\end{proof}

The following lemma yields a new skeleton $\mathcal{E}_{1}$ from $\mathcal{E}$, with a better bound for $f_{\mathcal{E}_{1}}(e,z)$. It will serve as our main tool to increase the dimension of additional invariance in the proof of Theorem  \ref{margulis2024.4.30}.

\begin{lem}[Induction]\label{margulis2024.4.31}  
Let the  notation and assumptions be as in Lemma \ref{margulis2024.4.29}.  
      Suppose that $\kappa>0$,  $t\geq \hyperlink{2024.12.t2}{t_{2}}$, and 
 \begin{equation}\label{margulis2024.4.46}
  |F|\geq e^{\frac{1}{2}t}.
\end{equation} 
Let $h\in G_{\mathcal{E}}$. Then there exist
\begin{itemize}
  \item  a point $y_{1}\in h\mathcal{E}(h)\cap    \mathsf{Q}(y(h), e^{-2\kappa t},e^{km_{\gamma}})$,
  \item  a finite subset $F_{1}\subset B_{H_{\mathbb{C}}^{\perp}(y_{1})}(e^{-\kappa t})$ containing $0$ with
         \[ |F_{1}|\geq|F|\cdot e^{(3-\hyperlink{2024.12.k19}{\kappa_{19}})\kappa t},\]
  \item a skeleton $\mathcal{E}_{1}=\mathcal{E}(y_{1},F_{1}, e^{-\kappa t}, \frac{L(\eta)}{10})$,
\end{itemize}  
         such that both of the following are satisfied:
         \begin{equation}\label{margulis2024.4.51}
            y_{1}+F_{1}\subset B_{G}(e^{5-2\kappa t})h\mathcal{E}(h),
         \end{equation}
          \begin{equation}\label{margulis2024.4.50}
         \max_{y\in\mathcal{E}_{1}}  f_{\mathcal{E}_{1}}(e,y) \leq \max\left\{  e^{Mt-\frac{2}{3}k}, 2|F_{1}|^{1+\frac{k\gamma m_{\gamma}+2\gamma\kappa t}{[\frac{1}{2}+(3-\hyperlink{2024.12.k19}{\kappa_{19}})\kappa ]t}} \right\}.
         \end{equation}  
\end{lem}
\begin{proof} The argument is similar to Proposition \ref{margulis2024.3.07}.
   Let $h\in G_{\mathcal{E}}$.
Then the set $h\mathcal{E}(h)\cap \mathsf{Q}(y(h), e^{-2\kappa t},e^{km_{\gamma}})$ is contained in a finite union of local $G$-orbits. Thus, we have
   \begin{equation}\label{margulis2024.4.32}
     h\mathcal{E}(h)\cap \mathsf{Q}(y(h), e^{-2\kappa t},e^{km_{\gamma}})\subset \bigsqcup_{i\in\mathcal{I}}\mathsf{Q}_{G}(e^{-2\kappa t},e^{km_{\gamma}}).(y(h)+w_{i}) \subset \mathcal{H}^{(\eta/2)}_{1}(\alpha)
   \end{equation}
   where $w_{i}\in B_{H_{\mathbb{C}}^{\perp}(y(h))}(e^{-2\kappa t})$.

   We now define the skeleton $\mathcal{E}_{1}$ via a local observation.  Let
   \begin{align}
 y^{\ast}& =y(h),\ \ \ \ \ \ w_{i}^{\ast}=w_{i},\ \ \  \ \ \ \mathsf{V}=\mathsf{Q}_{G}(e^{-2\kappa t},e^{km_{\gamma}}),\;\nonumber\\
\mathcal{E}^{\ast}& =h\mathcal{E}(h)\cap \mathsf{Q}(y(h), e^{-2\kappa t},e^{km_{\gamma}}).\;  \nonumber
\end{align}
  Applying Example \ref{margulis2024.4.36},  there exist  
   \begin{alignat*}{9}
 y_{i}&\in \mathsf{Q}_{G}(e^{-2\kappa t},e^{km_{\gamma}}).(y(h)+w_{i})& \ \ \ &  & &   \\
\mathsf{h}_{i}&\in \mathsf{Q}_{G}(e^{-2\kappa t},e^{km_{\gamma}}),& \ \ \ &  & &    \\
F_{1} &\coloneqq \{w_{i1}\in   H^{\perp}_{\mathbb{C}}(y):i\in\mathcal{I}\},& \ \ \ &  &w_{ij}&  \in  H^{\perp}_{\mathbb{C}}(y_{j}), \\
 \mathcal{E}_{1} &\coloneqq \mathcal{E}(y_{1},F_{1},e^{-\kappa t},\textstyle{\frac{L(\eta)}{10}}),  & \ \ \ &  & &
\end{alignat*} 
 such that 
  \begin{equation}\label{margulis2024.4.39}
 y_{i}\in \mathcal{E}^{\ast} \ \ \ \text{ and }\ \ \  y_{i}= \mathsf{h}_{i}\mathsf{h}_{j}^{-1}(y_{j}+w_{ij})
   \end{equation}
 for any $i,j\in\mathcal{I}$. Write $z_{i}\coloneqq h^{-1}y_{i} \in\mathcal{E}(h)$.

     For a lower bound of $|F_{1}|$,
since $\Leb_{G}(\mathsf{E}_{[-\frac{L(\eta)}{10},\frac{L(\eta)}{10}]}(e^{-\kappa t}))\asymp (e^{-\kappa t})^{2}L(\eta)$, we have $\Leb_{G}(\mathcal{E})\asymp |F|\cdot e^{-2\kappa t}L(\eta)$. In view of the definition of $\Leb_{\mathcal{E}}$, we conclude that
\[h_{\ast}\Leb_{\mathcal{E}}(\mathsf{Q}_{G}(e^{-2\kappa t},e^{km_{\gamma}}).(y(h)+w_{i}))\leq\frac{e^{-6\kappa t}e^{-k m_{\gamma}}}{|F|\cdot e^{-2\kappa t}L(\eta)}\leq e^{-3\kappa t}e^{-k m_{\gamma}}\cdot |F|^{-1}.\]
(Recall that $e^{-\kappa t}\leq L(\eta)^{2}$.) 
Using (\ref{margulis2024.4.32}) and Lemma \ref{margulis2024.4.33} (c.f. (\ref{margulis2024.4.40})), we deduce  
\begin{equation}\label{effective2022.4.7}
 |F_{1}| =|\mathcal{I}|\geq|F|\cdot e^{(3-\hyperlink{2024.12.k19}{\kappa_{19}})\kappa t}.
\end{equation}  
 
 Next, We shall show that Lemma \ref{margulis2024.4.31} holds with $\mathcal{E}_{1}$. First, by (\ref{margulis2024.4.39}), we have
\[y_{i}= \mathsf{h}_{i}\mathsf{h}_{1}^{-1}(y_{1}+w_{i1})\in B_{G}(e^{3-2\kappa t}).(y_{1}+w_{i1} )\cap h\mathcal{E}(h).\]
Therefore, by Lemma \ref{margulis2024.4.35} and (\ref{margulis2024.4.39}), we have
\[ y_{1}+w_{i1}=\mathsf{h}_{1}\mathsf{h}_{i}^{-1}y_{i} \in  \mathsf{h}_{1}\mathsf{h}_{i}^{-1}h\mathcal{E}(h)\subset B_{G}(e^{5-2\kappa t})h\mathcal{E}(h)\]
 for any $w_{i1}\in F$. This establishes the first claim in Lemma \ref{margulis2024.4.31}(1).

For the proof of  Lemma \ref{margulis2024.4.31}(2), we need the following:
\begin{cla}\label{effective2022.4.14} Let $y\in \mathsf{h} (y_{1}+w_{i1})\subset\mathcal{E}_{1}$  where $\mathsf{h}\in\mathsf{E}_{[-\frac{L(\eta)}{10},\frac{L(\eta)}{10}]}(e^{-\kappa t})$ and $w_{i1}\in F_{1}$.
Then we have
\[f_{\mathcal{E}_{1}}(e,y)\leq 2f_{\mathcal{E}}(h,z_{i})+e^{k\gamma m_{\gamma}}e^{2\gamma\kappa t}\cdot | F_{1}|.\] 
\end{cla}
\begin{proof}[Proof of Claim \ref{effective2022.4.14}]  
Let $y\in\mathcal{E}_{1}$. Then 
\begin{equation}\label{margulis2024.4.45}
 y   =\mathsf{h}(y_{1}+w_{i1})=\mathsf{h} \mathsf{h}_{1}\mathsf{h}_{i}^{-1}y_{i}=\mathsf{h}^{\prime}y_{i}
\end{equation} 
 for some  $i\in\mathcal{I}$ and $\mathsf{h}^{\prime}=\mathsf{h}\mathsf{h}_{1}\mathsf{h}_{i}^{-1}\in B_{G}(\frac{L(\eta)}{5})$.
Writing $c^{\ast}=e^{-k m_{\gamma}}e^{-2\kappa t}$, by (\ref{margulis2024.4.44}), we have
\begin{equation}\label{polynomial2022.10.12}
f_{\mathcal{E}_{1}}(e,y)\leq   \sum_{\|w\|\leq e^{-k m_{\gamma}}e^{-2\kappa t}}\|w\|^{-\gamma}+e^{k\gamma m_{\gamma}}e^{2\gamma\kappa t}\cdot|F_{1}|. 
\end{equation} 
In consequence, we need to investigate the first summation in (\ref{polynomial2022.10.12}).

 Let $w\in F_{e,y}(\mathcal{E}_{1})$, then $y+w\in\mathcal{E}_{1}$. Similar to (\ref{margulis2024.4.45}), we may write $ y+w    =\mathsf{h}^{\prime\prime}y_{j}$  for some  $j\in\mathcal{I}$ and $\mathsf{h}^{\prime\prime}\in B_{G}(\frac{L(\eta)}{5})$. 
Recall   from (\ref{margulis2024.4.39}) that
\[\mathsf{h}^{\prime}y_{i}+w=y+w    =\mathsf{h}^{\prime\prime}y_{j}= \mathsf{h}^{\prime\prime}\mathsf{h}_{j}\mathsf{h}_{i}^{-1}   (y_{i}+w_{ji}). \]
 Hence, by (\ref{margulis2024.4.52}), we have $\mathsf{h}^{\prime}= \mathsf{h}^{\prime\prime}\mathsf{h}_{j}\mathsf{h}_{i}^{-1}$ and so
\begin{equation}\label{margulis2024.4.43}
   \|w_{ji}\|=\|(\mathsf{h}^{\prime})^{-1}w\|\leq 2\|w\|\leq 2e^{-k m_{\gamma}}e^{-2\kappa t} \leq L(2\eta)\leq L(hz_{i}).
\end{equation}
where the last inequality follows from (\ref{margulis2024.4.27}).
Clearly, the map $F_{e,y}(\mathcal{E}_{1})\rightarrow H_{\mathbb{C}}^{\perp}(y_{j})$ given by $w\mapsto w_{ji}$ is well-defined and one-to-one. 
  
Recall also that $\mathsf{h}_{i},\mathsf{h}_{j}\in\mathsf{Q}_{G}(e^{-2\kappa t},e^{km_{\gamma}})$ and that (\ref{effective2022.3.14}) holds for $z_{i} \in\mathcal{E}(h)$. Therefore, as $h\in\supp(\nu^{(k)})$, in particular, it is of the form $h=a_{k m_{\gamma}}u_{r}$ for $|r|<2$, we have by Lemma \ref{margulis2024.4.35} and (\ref{effective2022.3.14})  that  
\[hz_{i}+w_{ji}=   \mathsf{h}_{i}\mathsf{h}_{j}^{-1} hz_{j}=   h\cdot B_{G}(e^{10-2\kappa t}).z_{j}\in h\mathcal{E}\]
Together with (\ref{margulis2024.4.43}), we conclude that $w_{ji}\in F_{h,z_{i}}(\mathcal{E})$. Moreover, (\ref{margulis2024.4.43}) and the fact that $w\mapsto w_{ji}$ is one-to-one imply that
\[\sum_{\|w\|\leq e^{-k m_{\gamma}}e^{-2\kappa t}}\|w\|^{-\gamma}\leq 2^{\gamma} f_{\mathcal{E}}(h,z_{i}).\]
The consequence follows.
\end{proof} 
Now by (\ref{margulis2024.4.46}) and (\ref{effective2022.4.7}), we have
\begin{equation}\label{margulis2024.4.48}
   |F_{1}|\geq|F|\cdot e^{(3-\hyperlink{2024.12.k19}{\kappa_{19}})\kappa t}\geq e^{[\frac{1}{2}+(3-\hyperlink{2024.12.k19}{\kappa_{19}})\kappa ]t}.
\end{equation} 
 
Let $y\in\mathcal{E}_{1}$, and let $z_{i}\in\mathcal{E}(h)$ be as in   Claim \ref{effective2022.4.14}. Then by (\ref{margulis2024.4.47}) we have
\[f_{\mathcal{E}}(h,z_{i})\leq   e^{Mt-\frac{3}{4}k}.\]
Thus, using Claim \ref{effective2022.4.14} and (\ref{margulis2024.4.48}), we deduce that
\begin{align}
\max_{y\in\mathcal{E}_{1}}  f_{\mathcal{E}_{1}}(e,y) & \leq 2 e^{Mt-\frac{3}{4}k}+e^{k\gamma m_{\gamma}+2\gamma\kappa t}\cdot | F_{1}|  \nonumber\\
   & \leq 2 e^{Mt-\frac{3}{4}k}+ |F_{1}|^{\frac{k\gamma m_{\gamma}+2\gamma\kappa t}{[\frac{1}{2}+(3-\hyperlink{2024.12.k19}{\kappa_{19}})\kappa ]t}}\cdot |F_{1}|\nonumber\\
     & \leq \max\left\{  e^{Mt-\frac{2}{3}k}, 2|F_{1}|^{1+\frac{k\gamma m_{\gamma}+2\gamma\kappa t}{[\frac{1}{2}+(3-\hyperlink{2024.12.k19}{\kappa_{19}})\kappa ]t}} \right\}. \nonumber
\end{align} 
This establishes (\ref{margulis2024.4.50}). 
\end{proof}

\subsection{Proof of Theorem \ref{margulis2024.4.30}}
We are now in the position to prove Theorem \ref{margulis2024.4.30}.  

\begin{proof}[Proof of Theorem \ref{margulis2024.4.30}] 

In \hypertarget{2024.04.k123}
  view  \hypertarget{2024.04.t345}  of Proposition \ref{margulis2024.3.07} and Lemma \ref{margulis2024.4.31}, let
\begin{itemize}
  \item   $\eta\in(0,N_{0}^{-1})$,
  \item $N\geq 2N_{0}$,
  \item  $x_{0}\in  \mathcal{H}_{1}(\alpha)$, 
  \item $\epsilon \in(0,\frac{1}{10})$,
  \item $\gamma \in(0,1)$,
  \item   $\hyperlink{2024.12.NN}{\varkappa}=\hyperlink{2024.12.NN}{\varkappa}(N,\gamma,\varpi)=[\frac{3}{2}(N+\gamma-\frac{1}{2}) m_{\gamma}+(5\varpi+2)]$, 
  \item $\hyperlink{2024.12.k1}{\kappa_{1}}=\hyperlink{2024.12.k1}{\kappa_{1}}(N,\gamma,\alpha,\epsilon)$ be small enough so that for any $0<\kappa\leq\hyperlink{2024.12.k1}{\kappa_{1}}$,
             \begin{equation}\label{margulis2024.4.55}
  2\left(\hyperlink{2024.12.k17}{\kappa_{17}}(\alpha,\gamma
,\kappa) +\frac{\epsilon }{10\gamma m_{\gamma}}\right)<\epsilon,
        \end{equation}
        \begin{equation}\label{margulis2024.4.62}
         \frac{\frac{\epsilon }{10 (\hyperlink{2024.12.k16}{\kappa_{16}}+1)}+2\gamma\kappa }{\frac{1}{2}+(3-\hyperlink{2024.12.k19}{\kappa_{19}})\kappa }<\epsilon,
        \end{equation}
        \begin{equation}\label{margulis2024.4.63}
          1-\left(\hyperlink{2024.12.k15}{\kappa_{15}}+  \frac{2(N+\gamma-\frac{1}{2})(150 \gamma m_{\gamma}(\hyperlink{2024.12.k16}{\kappa_{16}}+1))}{\epsilon}(\hyperlink{2024.12.k19}{\kappa_{19}}-3)\right)\kappa\geq \frac{2}{3}.
        \end{equation} 
      \item     $\hyperlink{2024.12.t1}{t_{1}}= \hyperlink{2024.12.t1}{t_{1}}(\gamma,\epsilon,\eta,\alpha,\ell(x_{0}))\geq \hyperlink{2024.12.t2}{t_{2}}$ be large enough so that  for any $t\geq \hyperlink{2024.12.t1}{t_{1}}$,
          \begin{equation}\label{margulis2024.4.61}
          \frac{\epsilon t}{100 \gamma  m_{\gamma}(\hyperlink{2024.12.k16}{\kappa_{16}}+1)} >  200 \hyperlink{2024.12.k4}{\kappa_{4}}|\log L(\eta)|+\hyperlink{2024.3.C6}{C_{6}}.
          \end{equation} 
            In particular, we have $e^{-\kappa t}\leq L(\eta)^{2}$.
    \item $k=\left\lceil \frac{\epsilon t}{100 \gamma m_{\gamma}(\hyperlink{2024.12.k16}{\kappa_{16}}+1)}\right\rceil$. (In particular, $k$ satisfies (\ref{margulis2024.4.25})  via (\ref{margulis2024.4.61}).)
      \end{itemize}    
Suppose that Theorem \ref{margulis2024.4.30}(2) does not hold. In the following, we shall show that 
there exists   $C>0$, a skeleton $\mathcal{E}$ with spine $F$ such that the local density function is bounded by 
\begin{equation}\label{margulis2024.4.57}
  \max_{z\in\mathcal{E}} f_{\mathcal{E}} (e,z)\leq C| F|^{1+\epsilon}.
\end{equation} 

\begin{enumerate}[$\bullet$] 
\item  \textbf{Construction of $\mathcal{E}_{0}$.}   
Apply Proposition \ref{margulis2024.3.07}  with $\eta, N,x_{0}, \kappa,t$ as above. 
We obtain a skeleton $\mathcal{E}_{0}$ of the form (\ref{margulis2024.4.17}):
\begin{align}
& \mathcal{E}_{0}   =\mathcal{E}(y_{0},F_{0},e^{-\kappa t},\textstyle{\frac{L(\eta)}{10}})\subset \mathsf{E}_{(5\varpi +2)t,[0,\frac{11}{10}]}(e^{3-\kappa t}).x_{0}, \nonumber\\
 &   F_{0} \subset B_{H^{\perp}_{\mathbb{C}}(y_{0})}(e^{-\kappa t}),\ \ \ \ \ \  e^{\frac{1}{2}t}\leq e^{(1-\hyperlink{2024.12.k15}{\kappa_{15}}\kappa) t}\leq |F_{0}|\leq  e^{(\hyperlink{2024.12.k15}{\kappa_{15}}-\kappa)t},\label{margulis2024.4.56}\\
 &  y_{0} \in\mathcal{H}_{1}^{(2\eta)}(\alpha) ,\nonumber\\
   &  \max_{z\in\mathcal{E}_{0}} f_{\mathcal{E}_{0}} (e,z)\leq  e^{(N+\gamma)t}  ,\label{margulis2024.4.53}
\end{align}  
where the first inequality in (\ref{margulis2024.4.56}) follows from (\ref{margulis2024.4.63}). In particular, we have $\mathcal{E}_{0}\subset\mathcal{H}_{1}^{(\eta)}(\alpha)$. Apply Proposition \ref{margulis2024.4.16} with the skeleton $\mathcal{E}_{0}$, $M=N+\gamma$ and $k$. If (\ref{margulis2024.4.20}) holds, then together with (\ref{margulis2024.4.55})(\ref{margulis2024.4.56}), we get
\begin{equation}\label{margulis2024.4.58}
    \max_{z\in\mathcal{E}_{0}} f_{\mathcal{E}_{0}} (e,z)\leq e^{\hyperlink{2024.12.k17}{\kappa_{17}} t+\frac{\epsilon }{10\gamma m_{\gamma}}t}| F_{0}|\leq | F_{0}|^{1+\frac{\hyperlink{2024.12.k17}{\kappa_{17}} t+\frac{\epsilon }{10\gamma m_{\gamma}}t}{\frac{1}{2}t}}\leq | F_{0}|^{1+\epsilon}.
\end{equation} 
Thus, we establish (\ref{margulis2024.4.57}).
\item  \textbf{Construction of $\mathcal{E}_{1}$.}    Hence, in the remaining, we suppose (\ref{margulis2024.4.21}) holds. Apply Lemma \ref{margulis2024.4.29}  with the skeleton $\mathcal{E}_{0}$. We obtain some $h_{0}\in G_{\mathcal{E}_{0}}\subset\supp(\nu^{(k)})$. Then since by (\ref{margulis2024.4.56}),
 \[|F_{0}|\geq e^{(1-\hyperlink{2024.12.k15}{\kappa_{15}}\kappa) t}\geq e^{\frac{1}{2}t}\]
 and so (\ref{margulis2024.4.46}) is satisfied. We further apply 
Lemma \ref{margulis2024.4.31} with $\mathcal{E}_{0}$ and $h_{0}$. We obtain a new skeleton $\mathcal{E}_{1}$ of the form (\ref{margulis2024.4.17}):
\begin{align}
& \mathcal{E}_{1}   =\mathcal{E}(y_{1},F_{1},e^{-\kappa t},\textstyle{\frac{L(\eta)}{10}})\subset \mathcal{H}_{1}^{(\eta)}(\alpha), \nonumber\\
 &   F_{1} \subset B_{H^{\perp}_{\mathbb{C}}(y_{1})}(e^{-\kappa t}),\ \ \ \ \ \  e^{\frac{1}{2}t}\leq e^{[1-(\hyperlink{2024.12.k15}{\kappa_{15}}+\hyperlink{2024.12.k19}{\kappa_{19}}-3)\kappa ]t} \leq |F_{1}|,\label{margulis2024.4.64}\\
 & y_{1}+F_{1}\subset B_{G}(e^{5-2\kappa t})h_{0}\mathcal{E}_{0},\nonumber\\
   &  \max_{z\in\mathcal{E}_{1}} f_{\mathcal{E}_{1}} (e,z)\leq     \max\left\{  e^{(N+\gamma-\frac{\epsilon }{150 \gamma m_{\gamma}(\hyperlink{2024.12.k16}{\kappa_{16}}+1)})t}, 2|F_{1}|^{1+\frac{k\gamma m_{\gamma}+2\gamma\kappa t}{[\frac{1}{2}+(3-\hyperlink{2024.12.k19}{\kappa_{19}})\kappa ]t}} \right\}.\nonumber
\end{align}  
where the first inequality in (\ref{margulis2024.4.64}) follows from (\ref{margulis2024.4.63}). Clearly, by the choices of $\kappa$ (\ref{margulis2024.4.62}) and $k$, if 
\[\max_{z\in\mathcal{E}_{1}} f_{\mathcal{E}_{1}} (e,z)\leq 2|F_{1}|^{1+\frac{k\gamma m_{\gamma}+2\gamma\kappa t}{[\frac{1}{2}+(3-\hyperlink{2024.12.k19}{\kappa_{19}})\kappa ]t}}\leq 2|F_{1}|^{1+\frac{\frac{\epsilon t}{10 (\hyperlink{2024.12.k16}{\kappa_{16}}+1)}+2\gamma\kappa t}{[\frac{1}{2}+(3-\hyperlink{2024.12.k19}{\kappa_{19}})\kappa ]t}} \leq 2|F_{1}|^{1+\epsilon},\] 
then we establish (\ref{margulis2024.4.57}) with $\mathcal{E}_{1}$. Thus, we suppose that 
\[ \max_{z\in\mathcal{E}_{1}} f_{\mathcal{E}_{1}} (e,z)\leq     e^{(N+\gamma-\frac{\epsilon }{150 \gamma m_{\gamma}(\hyperlink{2024.12.k16}{\kappa_{16}}+1)})t}.\]
Then we establish (\ref{margulis2024.4.22}) with $M=N+\gamma-\frac{\epsilon }{150 \gamma m_{\gamma}(\hyperlink{2024.12.k16}{\kappa_{16}}+1)}$ and $\mathcal{E}_{1}$. Repeating as above, we apply Proposition \ref{margulis2024.4.16} with  $k$ again. 
If   (\ref{margulis2024.4.20}) holds, we return to (\ref{margulis2024.4.58}) again:
 \[   \max_{z\in\mathcal{E}_{1}} f_{\mathcal{E}_{1}} (e,z)\leq e^{\hyperlink{2024.12.k17}{\kappa_{17}} t+\frac{\epsilon }{10\gamma m_{\gamma}}t}| F_{1}|\leq | F_{1}|^{1+\frac{\hyperlink{2024.12.k17}{\kappa_{17}} t+\frac{\epsilon }{10\gamma m_{\gamma}}t}{\frac{1}{2}t}}\leq | F_{1}|^{1+\epsilon}\] 
 as desired. Therefore, we meet (\ref{margulis2024.4.21}) again.
 \item  \textbf{Construction of $\mathcal{E}_{i}$.}  
 By repeating the above argument, we obtain a sequence  of elements $h_{i-1}\in \supp(\nu^{(k)})$, and a sequence of skeletons $\mathcal{E}_{i}$ of the form (\ref{margulis2024.4.17}):
\begin{align}
& \mathcal{E}_{i}   =\mathcal{E}(y_{i},F_{i},e^{-\kappa t},\textstyle{\frac{L(\eta)}{10}})\subset \mathcal{H}_{1}^{(\eta)}(\alpha), \nonumber\\
 &   F_{i} \subset B_{H^{\perp}_{\mathbb{C}}(y_{i})}(e^{-\kappa t}),\ \ \ \ \ \  e^{[1-(\hyperlink{2024.12.k15}{\kappa_{15}}+i(\hyperlink{2024.12.k19}{\kappa_{19}}-3))\kappa ]t} \leq |F_{i}|,\label{margulis2024.4.60}\\
 &  y_{i}+F_{i} \subset B_{G}(e^{5-2\kappa t})h_{i-1}\mathcal{E}_{i-1},\label{margulis2024.4.68}\\
   &  \max_{z\in\mathcal{E}_{i}} f_{\mathcal{E}_{i}} (e,z)\leq     \max\left\{  e^{(N+\gamma-\frac{i\epsilon }{150 \gamma m_{\gamma}(\hyperlink{2024.12.k16}{\kappa_{16}}+1)})t}, 2|F_{i}|^{1+\epsilon} \right\}.\label{margulis2024.4.59}
\end{align}  
By comparing (\ref{margulis2024.4.60}) and (\ref{margulis2024.4.63}), we see that if 
\begin{equation}\label{margulis2024.4.65}
  i\leq \left\lceil\frac{(N+\gamma-\frac{1}{2})(150 \gamma m_{\gamma}(\hyperlink{2024.12.k16}{\kappa_{16}}+1))}{\epsilon}\right\rceil,
\end{equation} 
then 
\[   |F_{i}|\geq e^{[1-(\hyperlink{2024.12.k15}{\kappa_{15}}+i(\hyperlink{2024.12.k19}{\kappa_{19}}-3))\kappa ]t} \geq e^{\frac{2}{3}t}>e^{\frac{1}{2}t}.\]
and so the skeletons $\mathcal{E}_{i}$ are guaranteed to apply Lemma \ref{margulis2024.4.31} again.
Moreover, together with (\ref{margulis2024.4.59}), we eventually meet (\ref{margulis2024.4.57}) for some $i$ satisfying (\ref{margulis2024.4.65}); otherwise, we must have
\[  e^{\frac{2}{3}t} \leq |F_{i}|\leq 2|F_{i}|^{1+\epsilon}\leq e^{(N+\gamma-\frac{i\epsilon }{150 \gamma m_{\gamma}(\hyperlink{2024.12.k16}{\kappa_{16}}+1)})t}   \]
but if we take the maximum of $i$ in (\ref{margulis2024.4.65}), we get 
 \[N+\gamma-\frac{i\epsilon }{150 \gamma m_{\gamma}(\hyperlink{2024.12.k16}{\kappa_{16}}+1)}\leq \frac{1}{2} \]
 which leads to a contradiction. Finally, by repeatedly using (\ref{margulis2024.4.68}), one may calculate that  
 \begin{align}
\mathcal{E}_{i}& \subset B_{G}(e^{10-\kappa t})a_{ik m_{\gamma}+(5\varpi+2)t}u_{[0,2]}x_{0}\;\nonumber\\
 & \subset B_{G}(e^{10-\kappa t})a_{[\frac{3}{2}(N+\gamma-\frac{1}{2}) m_{\gamma}+(5\varpi+2)]t}u_{[0,2]}x_{0}.\;  \nonumber
\end{align}
\end{enumerate} 
  Therefore, since we have defined $\hyperlink{2024.12.NN}{\varkappa}=\frac{3}{2}(N+\gamma-\frac{1}{2}) m_{\gamma}+(5\varpi+2)$, we conclude that there exists a skeleton $\mathcal{E}$ satisfying:
  \begin{align}
& \mathcal{E}  =\mathcal{E}(x_{1},F,e^{-\kappa t},\textstyle{\frac{L(\eta)}{10}})\subset \mathcal{H}_{1}^{(\eta)}(\alpha)\cap \mathsf{E}_{\hyperlink{2024.12.NN}{\varkappa}t,[0,2]}(e^{10-\kappa t}) .x_{0},\nonumber\\
 &   F \subset B_{H^{\perp}_{\mathbb{C}}(y)}(e^{-\kappa t}),\ \ \ \ \ \  e^{ \frac{1}{2}t} \leq |F|,\label{margulis2024.4.67}\\
   &  \max_{z\in\mathcal{E}} f_{\mathcal{E}} (e,z)\leq     2|F|^{1+\epsilon}  .\label{margulis2024.4.66}
\end{align}   
Finally, 
let $w^{\prime}\in F$, $z=y+w^{\prime}\in\mathcal{E}$. Then one calculates   \[ z+w-w^{\prime}= y+w^{\prime}+w-w^{\prime}=y+w \in\mathcal{E}.\]  
  It follows that $w-w^{\prime}\in F_{e,z}(\mathcal{E})$. Thus, we have
  \[\sum_{\substack{w\neq w^{\prime}\\ w\in F}}\|w-w^{\prime}\|^{-\gamma}\leq \sum_{w- w^{\prime}\in F_{e,z}(\mathcal{E})}\|w-w^{\prime}\|^{-\gamma}= f_{\mathcal{E}}(e,z).\]
  Together with (\ref{margulis2024.4.66}), we conclude that  
  \[\max_{w^{\prime}\in F}\sum_{\substack{w\neq w^{\prime}\\ w\in F}}\|w-w^{\prime}\|^{-\gamma}\leq  2|F|^{1+\epsilon}  .\]
  This completes the proof.
\end{proof}

\bibliographystyle{alpha}
 \bibliography{text}
\end{document}